\documentclass[a4paper]{amsart}
\pdfoutput=1
\usepackage[allcolors=blue,colorlinks]{hyperref}
\usepackage{amssymb,amsfonts,amsxtra}
\usepackage{mathrsfs}
\usepackage[numbers]{natbib}
\usepackage[all,cmtip]{xy}
\usepackage{pgfplots}
\usepackage{color}
\usepackage{stmaryrd}
\usepackage[english]{babel}

\newcommand{\LF}[2]{\langle #1,#2\rangle}
\newcommand{\PP}[2]{\langle #1|#2\rangle}
\newcommand{\DLF}[2]{\langle\!\langle #1,#2\rangle\!\rangle}

\newcommand{\bom}{\boldsymbol{\omega}}
\newcommand{\ra}{\rightarrow}
\newcommand{\lra}{\longrightarrow}
\newcommand{\ZZ}{\mathbb{Z}}
\newcommand{\NN}{\mathbb{N}}

\newcommand{\QQ}{\mathbb{Q}}
\newcommand{\RR}{\mathbb{R}}
\newcommand{\CC}{\mathbb{C}}
\newcommand{\HH}{\mathbb{H}}
\newcommand{\XX}{\mathbb{X}}
\newcommand{\Cc}{\mathcal{C}}

\newcommand{\DD}{\mathbb{D}}
\newcommand{\FF}{\mathbb{F}}
\newcommand{\Oo}{\mathcal{O}}

\newcommand{\Hh}{\mathcal{H}}
\newcommand{\QHh}{\vec{\mathcal{H}}}
\newcommand{\QUu}{\vec{\mathcal{U}}}
\newcommand{\Tt}{\mathcal{T}}
\newcommand{\TT}{\mathbb{T}}

\newcommand{\Pone}{\mathbb{P}^1}

\newcommand{\Gg}{\mathcal{G}}
\newcommand{\Kk}{\mathcal{K}}
\newcommand{\Aa}{\mathcal{A}}
\newcommand{\Ee}{\mathcal{E}}
\newcommand{\KK}{\mathbb{K}}
\newcommand{\MM}{\mathbb{M}}
\renewcommand{\AA}{\mathbb{A}}
\renewcommand{\SS}{\mathbb{S}}

\newcommand{\vSs}{\mathscr{S}}

\newcommand{\Uu}{\mathcal{U}}

\newcommand{\Bb}{\mathcal{B}}
\newcommand{\Mm}{\mathcal{M}}

\newcommand{\ramif}{\mathrm{ra}}
\newcommand{\inert}{\mathrm{in}}

\newcommand{\bt}{\mathbf{t}}

\newcommand{\qa}[1]{\mathbf{#1}}
\newcommand{\quat}[3]{\bigl(\frac{#1,\, #2}{#3}\bigr)}

\newcommand{\spitz}[1]{\langle #1\rangle}
\newcommand{\rperp}[1]{#1^{\perp}}

\newcommand{\genus}{g(\Hh)}
\newcommand{\ovp}{\bar{p}}

\newcommand{\ms}[1]{\mathscr #1}
\newcommand{\gge}[1]{\mathbf{#1}}
\DeclareMathOperator{\Knull}{K_0}
\DeclareMathOperator{\D}{D}
\DeclareMathOperator{\Coker}{Coker}
\DeclareMathOperator{\ch}{char}
\DeclareMathOperator{\ord}{ord}

\DeclareMathOperator{\Soc}{Soc}
\DeclareMathOperator{\Rad}{rad}

\DeclareMathOperator{\Pic}{Pic}
\DeclareMathOperator{\Br}{Br}
\DeclareMathOperator{\Aut}{Aut}
\DeclareMathOperator{\Gal}{Gal}
\DeclareMathOperator{\Inn}{Inn}
\DeclareMathOperator{\oInn}{\overline{Inn}}
\DeclareMathOperator{\Fix}{Fix}
\DeclareMathOperator{\vect}{vect}
\DeclareMathOperator{\coh}{coh}
\DeclareMathOperator{\ind}{ind}
\DeclareMathOperator{\Qcoh}{Qcoh}
\DeclareMathOperator{\Mod}{Mod}
\DeclareMathOperator{\qgr}{qgr}
\renewcommand{\mod}{\operatorname{mod}}
\DeclareMathOperator{\End}{End}
\DeclareMathOperator{\END}{END}
\DeclareMathOperator{\Hom}{Hom}

\DeclareMathOperator{\Ext}{Ext}
\newcommand{\ShHom}{\ms H\!om}
\newcommand{\ShEnd}{\ms E\!nd}

\DeclareMathOperator{\add}{add}

\DeclareMathOperator{\rk}{rk}
\DeclareMathOperator{\matring}{M}

\newcommand{\Der}{\mathcal{D}}

\newcommand{\bDerived}[1]{\Der^b(#1)}
\DeclareMathOperator{\gr}{gr}

\DeclareMathOperator{\wgr}{\widehat{\gr}}
\DeclareMathOperator{\Proj}{Proj}
\newcommand{\odual}[1]{#1\spcheck{}}

\newtheorem{proposition}{Proposition}[section]
\newtheorem{theorem}[proposition]{Theorem}
\newtheorem{corollary}[proposition]{Corollary}

\newtheorem{lemma}[proposition]{Lemma}

\theoremstyle{definition}
\newtheorem{definition}[proposition]{Definition}
\newtheorem{remark}[proposition]{Remark}

\newtheorem{example}[proposition]{Example}
\newtheorem{examples}[proposition]{Examples}
\newtheorem{numb}[proposition]{\!\!}

\numberwithin{equation}{section}
\begin{document}
\title[Weighted noncommutative regular projective curves]{Weighted
  noncommutative \\ regular projective curves} \author{Dirk Kussin}
\address{{Graduate School of Mathematics \\ Nagoya
    University \\ Furo-cho \\ Chikusa-ku \\ Nagoya 464-8602 \\ Japan}}
\email{dirk@math.uni-paderborn.de}
\begin{abstract}
  Let $\Hh$ be a noncommutative regular projective curve over a
  perfect field $k$. We study global and local properties of the
  Auslander-Reiten translation $\tau$ and give an explicit description
  of the complete local rings, with the involvement of $\tau$. We
  introduce the $\tau$-multiplicity $e_{\tau}(x)$, the order of $\tau$
  as a functor restricted to the tube concentrated in $x$. We obtain a
  local-global principle for the (global) skewness $s(\Hh)$, defined
  as the square root of the dimension of the function (skew-) field
  over its centre. In the case of genus zero we show how the ghost
  group, that is, the group of automorphisms of $\Hh$ which fix all
  objects, is determined by the points $x$ with $e_{\tau}(x)>1$. Based
  on work of Witt we describe the noncommutative regular (smooth)
  projective curves over the real numbers; those with $s(\Hh)=2$ we
  call Witt curves. In particular, we study noncommutative elliptic
  curves, and present an elliptic Witt curve which is a noncommutative
  Fourier-Mukai partner of the Klein bottle. If $\Hh$ is weighted, our
  main result will be formulae for the orbifold Euler characteristic,
  involving the weights and the $\tau$-multiplicities. As an
  application we will classify the noncommutative $2$-orbifolds of
  nonnegative Euler characteristic, that is, the real elliptic,
  domestic and tubular curves. Throughout, many explicit examples are
  discussed.
\end{abstract}
\maketitle
\tableofcontents

\section{Introduction}
In this article we study categories $\Hh$ which have the same formal
properties as categories $\coh(X)$ of coherent sheaves over a regular
projective curve over a field $k$. The axioms are essentially taken
from Lenzing-Reiten~\cite{lenzing:reiten:2006}. A similar (more
general) system of axioms is formulated by Stafford-van den
Bergh~\cite[Sec.~7.1]{stafford:vandenbergh:2001}. Let $k$ be a field
and $\Hh$ a category satisfying the following properties (NC~1) to
(NC~5); the conditions (NC~6) and (NC~7) will follow from the others.
\begin{itemize}
\item[(NC~1)] $\Hh$ is small, connected, abelian, and each object in
  $\Hh$ is noetherian.
\item[(NC~2)] $\Hh$ is a $k$-category with finite dimensional Hom- and
  Ext-spaces.
\item[(NC~3)] There is an autoequivalence $\tau$ on $\mathcal{H}$
  (called the \emph{Auslander-Reiten translation}) such that Serre
  duality $\Ext^1_{\mathcal{H}} (X,Y)=\D\Hom_{\mathcal{H}} (Y,\tau X)$
  holds, where $\D=\Hom_k (-,k)$.
\item[(NC~4)] $\mathcal{H}$ contains an object of infinite length.
\end{itemize}
It follows from Serre duality that $\mathcal{H}$ is a hereditary
category, that is, $\Ext^n_{\mathcal{H}}$ vanishes for all $n\geq 2$.
Let $\mathcal{H}_0$ be the Serre subcategory of $\mathcal{H}$ formed
by the objects of finite length, and let $\Hh_+$ be the full
subcategory of objects not containing a simple object. Then each
indecomposable object of $\Hh$ lies in $\Hh_+$ or in
$\Hh_0$. Moreover, $\Hh_0 =\coprod_{x\in\mathbb{X}}\Uu_x$ (for some
index set $\mathbb{X}$) where $\Uu_x$ are connected uniserial
categories, called tubes. The objects in $\Uu_x$ are called
(skyscraper sheaves) concentrated in $x$. We also write
$\Hh=\coh(\XX)$. In order to avoid so-called degenerated cases,
discussed in~\cite{lenzing:reiten:2006}, we additionally assume:
\begin{itemize}
\item[(NC~5)] $\mathbb{X}$ consists of infinitely many points.
\end{itemize}
We call $\Hh$, and also $\XX$, a \emph{weighted noncommutative regular
  projective curve} over $k$, if it satisfies the conditions~(NC~1)
to~(NC~5), and we write $\Hh=\coh(\XX)$. We recall that, because of
(NC~5), our class of noncommutative curves forms a proper subclass of
those studied in~\cite{lenzing:reiten:2006}.\medskip

Axiom (NC~5) implies, by~\cite[Cor.~2.4]{lenzing:reiten:2006}, that
for each $x\in\XX$ the number $p(x)$ of isomorphism classes of simple
objects in $\Uu_x$ is finite. Also the second part of the following
condition will \emph{follow}, from the theory of hereditary orders,
compare Theorem~\ref{thm:structure} below.
\begin{enumerate}
\item[(NC 6)] For all points $x\in\XX$ we have $p(x)<\infty$, and for
  all except finitely many we have $p(x)=1$.
  \end{enumerate}
  The numbers $p(x)$, with $p(x)>1$, are called the \emph{weights} of
  $\Hh$. Points $x$ with $p(x)>1$ are called \emph{exceptional}. Thus
  there is a finite number of exceptional points, and of so-called
  \emph{exceptional} simple sheaves $S$, that is, simple objects $S$
  with $\Ext^1(S,S)=0$. By~\cite[Prop.~4.9]{lenzing:reiten:2006} each
  object in the quotient category $\Hh/\Hh_0$ is of finite length. An
  indecomposable object $L\in\Hh$ is called a \emph{line bundle} if it
  becomes a simple object modulo $\Hh_0$. We call a line bundle
  $L\in\Hh$ \emph{special}, if for each $x\in\XX$ there is (up to
  isomorphism) \emph{precisely one} simple sheaf $S_x$ concentrated in
  $x$ with $\Ext^1(S_x,L)\neq 0$.\medskip

  If we have $p(x)=1$ for all $x$, then we call $\Hh$
  \emph{non-weighted} (or homogeneous~\cite{kussin:2009}); this can be
  also expressed as follows
\begin{list}{(NC 6')}{\setlength{\topsep}{2.2pt plus 2.2pt}}
\item $\Ext^1(S,S)\neq 0$ (equivalently: $\tau S\simeq S$)
  for each simple object $S$.
\end{list}
\begin{proposition}[Reduction to the non-weighted
  case]\label{prop:reduction}
  Let $\Hh$ be a weighted noncommutative regular projective curve with
  the exceptional points given by $x_1,\dots,x_t$, of weights
  $p_i=p(x_i)>1$. Choose for every $i=1,\dots,t$ one simple sheaf
  $S_i$ concentrated in $x_i$. Let $\vSs\subseteq\Hh$ be the system
  $\{\tau^jS_i\mid i=1,\dots,t;\ j=1,\dots,p_i-1\}$.
  \begin{enumerate}
  \item[(1)] The right perpendicular category
    $\Hh'=\rperp{\vSs}\subseteq\Hh$ is a full, exact subcategory of
    $\Hh$, and is a non-weighted noncommutative regular projective
    curve.
  \item[(2)] There is a special line bundle $L$ in $\Hh$.
  \end{enumerate}
\end{proposition}
We remark that in general there are line bundles which are not
special, cf.~\cite[Ex.~8.5.1]{kussin:2009}. Also, even in the
non-weighted cases, the group $\Aut(\Hh)$ is in general not acting
transitively on the set of line bundles.
\begin{proof}
  This is similar to the proof of~\cite[Prop.~1]{lenzing:1997b}. For
  (2) we remark that the full exact embedding $\Hh'\subseteq\Hh$
  preserves the rank. So any line bundle in $\Hh'$ gives rise to a
  special line bundle in $\Hh$.
\end{proof}
It follows, cf.\ Proposition~\ref{prop:weight-insertion}, that each
\emph{weighted} noncommutative regular projective curve $\Hh$ over $k$
is obtained from a \emph{non-weighted} noncommutative regular
projective curve $\Hh'$ over $k$ by insertion of weights into a finite
number of points of $\Hh'$, in the sense of the $p$-cycle construction
from~\cite{lenzing:1998} (we refer also
to~\cite[Sec.~6.1]{kussin:2009}). We will always consider a pair
$(\Hh,L)$ with $L$ a special line bundle, which we consider as the
structure sheaf of $\Hh$. (Later we will require additional properties
on $L$, cf.~\ref{numb:structure-sheaf}; but in the first sections this
will not be needed.)\medskip

The quotient category $\widetilde{\Hh}=\Hh/\Hh_0$ is semisimple with
one simple object, given by the class $\widetilde{L}$ of $L$ (or any
line bundle), thus $\widetilde{\Hh}\simeq\mod(k(\Hh))$ for the skew
field $k(\Hh)=\End_{\widetilde{\Hh}}(\widetilde{L})$, which we call
the \emph{function field}. Moreover, we have
$\Hh/\Hh_0\simeq\Hh'/\Hh'_0$, thus $k(\Hh)\simeq k(\Hh')$.\medskip

It \emph{follows} from~\cite{artin:stafford:1995,lenagan:1994}, cf.\
Remark~\ref{rem:ample-pair-weighted}, that:
\begin{itemize}
\item[(NC~7)] The function field $k(\Hh)$ is of finite dimension over
  its centre $Z(k(\Hh))$, which is an algebraic function field in one
  variable over $k$.
\end{itemize}
For quotability we put this on record:
\begin{proposition}\label{prop:2-additional-properties}
  Each weighted noncommutative regular projective curve $\Hh$ over a
  field $k$ (defined by (NC~1)--(NC~5)) satisfies (NC~6) and (NC~7) as
  well.
\end{proposition}
We will later show (Theorem~\ref{thm:function-field-determines}) that,
as in the commutative case, $\Hh$, if non-weighted, is uniquely
determined by its function field $k(\Hh)$, moreover:
\begin{theorem}
  There is a bijection between the set of isomorphism classes of
  non-weighted noncommutative regular projective curves over $k$ and
  the set of isomorphism classes of central skew field extensions of
  algebraic function fields in one variable over $k$.
\end{theorem}
(This was also recently shown
in~\cite[Thm.~6.7]{burban:drozd:gavran:2015}.) Thus the study of
noncommutative regular projective curves is equivalent to the study of
such skew field extensions. We call the natural
number $$s(\Hh)=[k(\Hh):Z(k(\Hh))]^{1/2}$$ the (global)
\emph{skewness} of $\Hh$. Moreover, there we have
$Z(k(\Hh))\simeq k(X)$ for a unique regular projective curve over $k$
(we refer to~\ref{numb:centre-curve}). We call $X$ the \emph{centre
  curve} of $\Hh$. \medskip

Since we are mainly interested in arithmetic effects, we will mostly
deal with this non-weighted case. We will then almost always omit the
term ``non-weighted''; instead we will use the term ``weighted'' for
the general case, which we will treat mainly in the last chapter. In
different terminology ``weighted'' is called or related to
``orbifold'' or ``stacky''. In order to stress this connection, we
keep the term ``weighted'' in our general notion, although weights are
a built-in feature of general noncommutative curves, cf.\
Proposition~\ref{prop:2-additional-properties}.\medskip

For the rest of this introduction let $\Hh$ be a noncommutative
regular projective curve \emph{over a perfect field} $k$, and
\emph{non-weighted} if not otherwise specified. For each point $x$ we
denote by $S_x$ the unique simple sheaf concentrated in $x$.\medskip

The present paper aims for being a quite detailed introduction to
noncommutative curves, working out a new approach, presenting numerous
new results and discussing many explicit examples. In our approach the
main focus is on the functor $\tau$, the Auslander-Reiten translation,
which is of course a global datum of the category $\Hh$. We will study
local properties of this functor. In order to do this, we describe the
structure of the tubes $\Uu_x$ (the full subcategories of skyscaper
sheaves concentrated in one point $x$) explicitly. The
Auslander-Reiten translation $\tau$ is acting on each $\Uu_x$, and it
serves as the Auslander-Reiten translation on $\Uu_x$, which is itself
a hereditary category with Serre duality. The tubes are the most
basic, non-trivial examples of connected uniserial length
categories. P.~Gabriel~\cite{gabriel:1973} introduced the species of
such a category. In the case of a homogeneous tube with one simple
object $S$ this species is just the $D$-$D$-bimodule $\Ext^1(S,S)$,
where $D=\End(S)$. As the starting point of our local study of $\tau$
we determine these bimodules explicitly, by using results of
Lenzing-Zuazua~\cite{lenzing:zuazua:2004} on Serre duality. This is
done in Section~\ref{sec:bimodule-of-tube}. In
Section~\ref{sec:complete-local-rings} we use this to determine the
complete local rings as certain twisted power series rings.
\begin{theorem}
  For each point $x\in\XX$ the full subcategory $\Uu_x$ of skyscraper
  sheaves concentrated in $x$ is equivalent to the category of finite
  length modules over the skew power series ring
  $\End(S_x)[[T,\tau^-]]$. Here the twist $\tau^-$, with
  $Tf=\tau^-(f)T$ for all $f\in\End(S_x)$, is given by the restriction
  of the inverse Auslander-Reiten translation $\tau^-\colon\Hh\ra\Hh$
  to the simple object $S_x$ concentrated in $x$.
\end{theorem}
The clou is that the twist is always given by the (inverse)
Auslander-Reiten translation. From this we obtain almost at once a
local-global principle of skewness, in
Section~\ref{sec:local-global-principle}. Namely, we get that the
restriction of $\tau$ to $\Uu_x$ is of finite order, denoted by
$e_{\tau}(x)$, which we call the $\tau$-multiplicity in $x$. Then:
\begin{theorem}[Local-global principle of skewness]\label{thm:main-theorem}
  For each point $x\in\XX$ we have
  $$s(\Hh)=e(x)\cdot e^{\ast}(x)\cdot e_{\tau}(x),$$ where
  $e(x)=[\Ext^1(S_x,L):\End(S_x)]$,
  $e^{\ast}(x)=[\End(S_x):Z(\End(S_x))]^{1/2}$. 
\end{theorem}
As before, the clou is the involvement of the global functor
$\tau$. It should be noted that the multiplicities $e(x)$ were
introduced in representation theory of finite dimensional algebras by
Ringel~\cite{ringel:1979}.\medskip

In Section~\ref{sec:orders} we import a theorem of Reiten-van den
Bergh~\cite{reiten:vandenbergh:2002} which states that $\Hh$ is
equivalent to $\coh(\Aa)$, the coherent $\Aa$-modules, for a sheaf
$\Aa$ of maximal $\Oo_X$-orders in a central skew field over the
function field $k(X)$ of the centre curve $X$. We will sketch the
proof. In this section we also show (based on work by Artin-de
Jong~\cite{artin:dejong:2004}) the already mentioned important fact
that each noncommutative regular projective curve is uniquely
determined by its function field. We then illustrate that many results
and relations, well-known in the theory of orders, follow almost
automatic by our explicit constructions before. In particular, we see
that the $\tau$-multiplicities are just the ramification indices of
$\Aa$, for which a similar formula is well-known in certain
situations~\cite{reiner:2003}. Thus, our approach via $\tau$ sheds
also new light on orders and ramifications.\medskip

In Section~\ref{sec:dualizing-sheaf} we review some facts on the
different and dualizing sheaves. Using a result of van den Bergh-van
Geel~\cite{vandenbergh:vangeel:1984} we see that the Auslander-Reiten
translation lies in the Picard-shift group, $$\tau\in\Pic(\Hh),$$
which is defined to be the subgroup of the automorphism (class) group
$\Aut(\Hh)$ generated by the tubular shifts $\sigma_x$ in the sense of
Meltzer~\cite{meltzer:1997} and Lenzing-de la
Pe\~{n}a~\cite{lenzing:delapena:1999} (in this context agreeing with
the Seidel-Thomas twists~\cite{seidel:thomas:2001}). Moreover, we show
that $\Pic(\Hh)$ is essentially determined by $\Pic(X)$, the Picard
group (of line bundles) over the centre curve $X$.
\begin{theorem}
  There is an exact sequence
  \begin{equation*}
    1\ra\Pic(X)\ra\Pic(\Hh)\ra\prod_{x\in\XX}\ZZ/e_{\tau}(x)\ZZ\ra 1 
 \end{equation*}
 of abelian groups. 
\end{theorem}
This has, for instance, the effect, if $\Hh$ is, say, elliptic and $X$
is of genus zero (we will see such an example over the real numbers
later), that then $\Pic(\Hh)$ is finitely generated abelian of rank
one.\medskip

In Section~\ref{sec:genus} we define Euler characteristic and genus of
a noncommutative regular projective curve. Our definition of the genus
is different and made in a more straight-forward fashion than the
definitions in~\cite{vandenbergh:vangeel:1984}
and~\cite{marubayashi:vanoystaeyen:2012}; the latter are based
on~\cite{witt:1934b}. Our proof of the Riemann-Roch theorem is then
almost trivial. We also present a formula by Artin-de
Jong~\cite{artin:dejong:2004} for the Euler characteristic, without
restriction on the characteristic of the base-field. We will, in
contrast to~\cite{artin:dejong:2004}, normalize the Euler
characteristic, so that it becomes a Morita invariant. This seems to
be more natural, particularly when studying noncommutative curves (or
orbifolds) over the real numbers.\medskip

We show several general results concerning the elliptic case (genus
one, Euler characteristic zero). In particular, the classification of
indecomposable objects is similar to Atiyah's classification of
indecomposable vector bundles for elliptic curves over an
algebraically closed field~\cite{atiyah:1957}. One major difference
here is that it is possible that a noncommutative elliptic curve may
have a non-trivial Fourier-Mukai partner. We will exhibit such
examples later over the real numbers.\medskip

In Section~\ref{sec:genus-zero} we treat quite detailed certain
aspects of the genus zero case. This is the case which is also
motivated by representation theory of finite dimensional algebras,
since this case is characterized by admitting tilting objects. One of
the most important techniques in representation theory is the
Auslander-Reiten theory: the concepts of almost split
sequences~\cite{auslander:reiten:1975} and the Auslander-Reiten
translation $\tau$ are the most prominent brands of this
theory. Almost split sequences are strongly linked to Serre duality,
see~\cite{reiten:vandenbergh:2002}. Our main focus in the present
paper is on the study of the ghost group $\Gg(\Hh)$, that is, the
subgroup of $\Aut(\Hh)$ given by those automorphisms fixing the
structure sheaf $L$ and all simple sheaves $S_x$ ($x\in\XX$). In
representation theory of finite dimensional algebras it is often
assumed that the base-field is algebraically closed. Then many
problems and questions are already determined combinatorially (say, by
working with dimension vectors instead of representations). This will
typically fail over general base-fields, and the ghost group can be
regarded as a measure for this failure. Good, explicit knowledge of
the ghosts, the members of the ghost group, combined with the
combinatorial methods, is therefore important for exploring categories
of finite dimensional modules. Several of the problems posed
in~\cite{kussin:2009} will be solved. Concerning the ghost group our
main result is the following.
\begin{theorem}
  Let $\Hh$ be of genus zero. Assume that there is a point $x$ such
  that the tubular shift $\sigma_x$ is efficient in the sense
  of~\cite{kussin:2009}. Then the ghost group $\Gg(\Hh)$ is finite,
  generated by Picard-shifts ${\sigma_x}^{-d(y)}\circ{\sigma_y}^{ }$,
  which are of order $e_{\tau}(y)$, where $y$ runs through the points
  $y\neq x$ with $e_{\tau}(y)>1$. 
\end{theorem}
Typically, $x$ itself will be a point with $e_{\tau}(x)>1$, so that
then there are at most two further points $y_1,\,y_2$ with
$e_{\tau}(y_i)>1$, and then
$\Gg(\Hh)\simeq C_{e_{\tau}(y_1)}\times C_{e_{\tau}(y_2)}$.\medskip 

In order to become able to treat many interesting examples of higher
genus also, we work out the whole picture of noncommutative regular
projective curves over the real numbers. This is based on work by
E.~Witt~\cite{witt:1934} on central skew field extensions of real
algebraic function fields in one variable. These skew fields
correspond to noncommutative real regular projective curves, which we
call (unless commutative) Witt curves. It seems that Witt's
function-theoretic study~\cite{witt:1934} was never fully exploited in
order to study noncommutative curves over the reals. By Witt's
theorem~\cite{witt:1934} these curves correspond to Klein surfaces,
with each of its ovals (=boundary components) divided into a finite
number of segments, labelled with alternating signs ``$+$'' and
``$-$'' (in this context we also call them Witt surfaces if at least
one ``$-$'' occurs). We prove a Riemann-Hurwitz formula for the genus
(in our definition) of Witt curves. We classify all genus zero and all
genus one Witt curves. The latter will be done in
Section~\ref{sec:elliptic} by classifying topologically the
noncommutative real elliptic curves.
\begin{theorem}
  The Klein bottle has as a Fourier-Mukai partner a Witt curve given
  by the annulus with two differently signed ovals.
\end{theorem}
The theorem, a non-weighted analogue of~\cite{kussin:2000}, describes
a situation where the conclusion of a theorem of
Bondal-Orlov~\cite{bondal:orlov:2001} does not hold. It also shows
that the recent result~\cite{lopez_martin:2014} does not extend to the
non-algebraically closed base-fields.\medskip

We also show that in all elliptic cases the Auslander-Reiten
translation $\tau$ has finite order, more precisely, given by $1$,
$2$, $3$, $4$ or $6$, depending on the specific example; of course,
over the reals only $1$ and $2$ occur.\medskip

We end the paper with Section~\ref{sec:weighted} about the weighted
cases. We are convinced that separated treatments of the non-weighted
and the weighted cases makes the whole theory more transparent; the
focus in the non-weighted cases lies on arithmetic properties (like
the multiplicities and $\tau$-multiplicities, ghost group, etc.), and
then the weighted case is of more combinatorial nature. Here our main
result are formulae for the (normalized) orbifold Euler
characteristic, of two types:
\begin{theorem}
  Let $\Hh$ be a weighted noncommutative regular projective curve. Let
  $X$ be the centre curve, $\Hh_{nw}$ the underlying non-weighted
  curve. For the normalized orbifold Euler characteristic
  $\chi'_{orb}(\Hh)$ we have
 \begin{eqnarray*}
    \label{eq:orbifold-euler-char-formula-intro}
    \chi'_{orb}(\Hh) & = & \chi'(X)-\frac{1}{2}
    \sum_{x}\Bigl(1-\frac{1}{p(x)e_{\tau}(x)}\Bigr)[k(x):k]\\
    \label{eq:orbifold-euler-char-formula-weights-intro}
      & = & \chi'(\Hh_{nw})-\frac{1}{2}
      \sum_{x}\frac{1}{e_{\tau}(x)}\Bigl(1-\frac{1}{p(x)}\Bigr)[k(x):k].
  \end{eqnarray*}
\end{theorem}
(Here, the $k(x)$ are the residue class fields over the centre curve
$X$.) Over the real numbers this yields a formula for the Euler
characteristic of noncommutative (compact) two-dimensional orbifolds,
extending the formula in the classical case from Thurston's
book~\cite{thurston:2002}:
\begin{corollary}[General Riemann-Hurwitz formula]
  Let $\Hh$ be a noncommutative real $2$-orbifold with underlying
  compact Riemann, Klein or Witt surface $\Hh_{nw}$. Then
    $$\begin{array}{l}
        \chi'_{orb}(\Hh)\ =\ \chi'(\Hh_{nw})-
        \frac{1}{4}\cdot\sum_{x}\bigl(1-\frac{1}{p(x)}\bigr)-\frac{1}{2}\cdot\sum_{y}
        \bigl(1-\frac{1}{p(y)}\bigr)-\sum_{z}
        \bigl(1-\frac{1}{p(z)}\bigr),
  \end{array}$$
  where $x$ runs over the ramification points, $y$ over the other
  boundary points, and $z$ over the inner points.
\end{corollary}
As an application we classify all weighted noncommutative regular
projective curves $\Hh$ with $\chi'_{orb}(\Hh)=0$ over the real
numbers; up to parameters there are $39$ cases, $8$ elliptic and $31$
tubular ones. $17$ have $s(\Hh)=1$ and $22$ have $s(\Hh)=2$.
\begin{theorem}
  Each tubular curve has (fractional) Calabi-Yau dimension $n/n$,
  where $n$ is the maximum of the numbers $p(x)e_{\tau}(x)$. The
  weight-ramification vector, given by the numbers
  $p(x)e_{\tau}(x)>1$, each counted $[k(x):k]$-times, is a derived
  invariant of a tubular curve.
\end{theorem}
We start with Section~\ref{sec:localization} by showing several basic
facts about noncommutative curves (partially extending results
from~\cite{reiten:vandenbergh:2002}) like the existence of homogeneous
coordinate rings (so that these curves are in particular
noncommutative projective schemes in the sense of
Artin-Zhang~\cite{artin:zhang:1994}). We explain two kinds of
localizations, one ring-theoretic (Ore-Asano), the other categorical
(Serre-Grothendieck-Gabriel), and show that both yield the same. They
result in the non-complete rings $R_x$, associated with each point
$x\in\XX$. These rings are noncommutative Dedekind domains with a
unique non-zero prime ideal, but in general not local. Their
completions are (Morita-equivalent to) local rings, which we are going
to describe as stated above.\medskip

We emphasize that many of our main results are in full generality,
without perfectness or separability assumption. We also elaborate in
detail an enlightning inseparable Example~\ref{ex:inseparable}.

\section{Basic concepts}
Let $\Hh$
be a weighted noncommutative regular projective curve over the field
$k$.
\begin{numb}[Rank function]
  Let $\Hh\ra\widetilde{\Hh}=\Hh/\Hh_0=\mod(k(\Hh))$,
  $X\mapsto\widetilde{X}$ be the quotient functor. The
  $k(\Hh)$-dimension on $\Hh/\Hh_0$ induces the \emph{rank function}
  $\rk\colon\Knull(\Hh)\ra\ZZ$ of $\Hh$. For an indecomposable object
  $E\in\Hh$ we have $\rk(E)=0$ if $E\in\Hh_0$ and $\rk(E)>0$ if
  $E\in\Hh_+$. In particular, an object $E\in\Hh$ has rank $0$ if and
  only if it is of finite length. An indecomposable object $L$ with
  $\rk(L)=1$ is called a \emph{line bundle}. The function field
  $k(\Hh)$ is isomorphic to the endomorphism ring of $\widetilde{L}$
  in $\widetilde{\Hh}$.
\end{numb}
\begin{numb}[Almost split sequences]
  We recall that a short exact sequence
\begin{equation}
  \label{eq:ass}
  \mu\colon\ 0\ra
  A\stackrel{u}\ra B\stackrel{v}\ra C\ra 0
\end{equation}
in $\Hh$ is called \emph{almost split}, \cite{auslander:reiten:1975},
if it does not split, if $A$ and $C$ are indecomposable, and if every
morphism $X\ra C$, which is not a split epimorphism, factors through
$v$. (Then also the dual factorization property holds.) Then $A$ is,
up to isomorphism, uniquely determined by $C$, and conversely. For
every indecomposable $C$ (resp.\ $A$) there is an almost split
sequence~\eqref{eq:ass} ending (starting) in $C$ (in $A$); then $\tau
C=A$ and $\tau^- A=C$, which define mutually quasiinverse
autoequivalences $\tau,\,\tau^-\colon\Hh\ra\Hh$, which appear in the
Serre duality. For categories of coherent sheaves $\tau$ is also known
as Serre functor; we will reserve this term for the derived category
of $\Hh$.

The almost split sequences are fundamental in the definition of the
\emph{Auslander-Reiten quiver} of $\Hh$: its vertices are the
isomorphism classes of indecomposable objects in $\Hh$, and the arrows
between classes of indecomposables are given by the so-called
\emph{irreducible} morphisms, which are the components of the maps
which occur in the corresponding almost split sequences.
\end{numb}
\begin{numb}[Homogeneous tubes]\label{nr:tubes}
  Let $x\in\XX$ and $\Uu=\Uu_x$ be the corresponding connected
  uniserial category in $\Hh_0$. Assuming $p(x)=1$, there is up to
  isomorphism precisely one simple object $S=S_x$ in $\Uu$. Such
  categories are also called \emph{homogeneous tubes}. For each
  $n\geq 1$ we denote by $S[n]$ the (up to isomorphism) unique
  indecomposable object in $\Uu$ of length $n$. (We additionally set
  $S[0]:=0$.) Thus we have $\Uu=\add(\{S[n]\mid n\geq 1\})$. We have
  injections $\iota_n\colon S[n]\ra S[n+1]$ and surjections
  $\pi_n\colon S[n+1]\ra S[n]$. The Auslander-Reiten translation
  satisfies $\tau S[n]\simeq S[n]$. We will usually identify them,
  $\tau S[n]=S[n]$. We then have almost split sequences
  $\mu_1\colon 0\ra S\stackrel{\iota_1}\lra S[2]\stackrel{\pi_1}\lra
  S\ra 0$, and for $n\geq 2$:
  $$\mu_n\colon 0\ra S[n]\stackrel{(\pi_{n-1},\iota_n)^t}\lra
  S[n-1]\oplus S[n+1]\stackrel{(\iota_{n-1},\pi_n)}\lra S[n]\ra 0.$$
  The $\iota_n$ and $\pi_n$ are the irreducible maps in $\Uu$.
\end{numb}
\begin{numb}[$\tau$-multiplicity]
  Let $\Uu_x$ be a homogeneous tube with simple object $S_x$. The
  Auslander-Reiten translation $\tau$ restricts to an autoequivalence
  of $\Uu_x$. Up to isomorphism it fixes all indecomposable objects
  $S_x[n]$. If we consider a skeleton of $\ind(\Uu_x)$, we can assume
  that equality $\tau S_x[n]=S_x[n]$ holds for all $n\geq 1$. The
  action on morphisms induces, in particular, an automorphism of
  $D_x=\End(S_x)$, that is, an element $\tau$ in $\Aut(D_x/k)$. We
  define $\Gal(D_x/k)=\Aut(D_x/k)/\Inn(D_x/k)$, the factor group
  modulo inner automorphisms. By the theorem of
  Skolem-Noether~\cite[12.6]{pierce:1982}, restriction to the centre
  $Z(D_x)$ yields an injective homomorphism
  $\Gal(D_x/k)\ra\Gal(Z(D_x)/k)$. We call
  \begin{equation}
    \label{eq:def-tau-multi}
    e_{\tau}(x)=\ord_{\Gal(D_x/k)}(\tau),
  \end{equation}
  the order of (the class of) $\tau$ in $\Gal(D_x/k)$, the
  $\tau$-\emph{multiplicity} of $x$.
\end{numb}
\begin{numb}[Picard-shifts]
  Let $x\in\XX$ be a point and $\Uu=\Uu_x$ be a homogeneous tube in
  $\Hh$. The indecomposable objects in $\Uu$ form the Auslander-Reiten
  component containing the simple object $S=S_x$ with support
  $\{x\}$. Then $\End(S)$ is a division algebra over $k$, and
  $\Ext^1(S,S)$ is one-dimensional as $\End(S)$-vector space. Thus $S$
  is, in the terminology of~\cite{seidel:thomas:2001}, a spherical
  object. (In~\cite{seidel:thomas:2001} only the case $\End(S)=k$ is
  considered.) For every object $E$ in $\Hh$, which has no
  indecomposable summand in $\Uu$, one has the $S$-universal extension
  $0\ra E\ra E(x)\ra E_x\ra 0$ of $E$ with
  $E_x=\Ext^1(S,E)\otimes_{\End(S)}S$. The assignment $E\mapsto E(x)$
  induces an autoequivalence $\sigma_x\colon\Hh\ra\Hh$, called the
  \emph{tubular} (or \emph{Picard-}) \emph{shift} associated with $x$,
  coming with a natural transformation
  $1_{\Hh}\stackrel{x}\ra\sigma_x$. We refer
  to~\cite{meltzer:1997,lenzing:delapena:1999,kussin:2009}. This
  coincides with the notion of a \emph{Seidel-Thomas twist},
  \cite{seidel:thomas:2001}. We will later see
  (Lemma~\ref{lem:isomorphic-point-functors}) that such a functor
  $\sigma_x$ is given as the tensor product with a certain
  bimodule. For $n\in\ZZ$ we write $E(nx)={\sigma_x}^n(E)$. The
  assignment $E\mapsto E_x$ is also
  functorial. Following~\cite[4.2]{lenzing:1998} we call $E_x$ the
  \emph{fibre} of $E$; if $f\in\Hom(E,E')$, we call
  $f_x\in\Hom(E_x,E'_x)$ the corresponding \emph{fibre map}.\medskip

  Let $x\neq y$ be two points. It is well-known that
  $\sigma_x\circ\sigma_y\simeq\sigma_y\circ\sigma_x$ holds, and that
  the restriction of $\sigma_x$ to $\Uu_y$ is isomorphic to the
  identity functor on $\Uu_y$. As a natural but non-trivial result we
  will show in Corollary~\ref{cor:tau-minus=sigma-x}, that over a
  perfect field $\sigma_x$ acts functorially on $\Uu_x$ like the
  inverse Auslander-Reiten translation $\tau^-$. We will also point
  out in Example~\ref{ex:inseparable} that this is not true in general
  over non-perfect fields. We denote by $\Pic(\Hh)$ the
  \emph{Picard-shift group}, that is, the subgroup of the automorphism
  (class) group $\Aut(\Hh)$ of $\Hh$ generated by all tubular shifts
  $\sigma_x$ ($x\in\XX$); it is an abelian group. We call a vector
  bundle $F$ a Picard-shift of a vector bundle $E$ if there is
  $\sigma\in\Pic(\Hh)$ such that $F\simeq\sigma(E)$. In particular,
  the group $\Pic(\Hh)$ acts on the set of isomorphism classes of line
  bundles on $\Hh$. If $\Hh=\coh(X)$ with $X$ commutative, then
  $\Pic(\Hh)$ is isomorphic to the \emph{Picard group} $\Pic(X)$,
  given by the set if isomorphism classes of line bundles with the
  tensor product. In general the action of $\Pic(\Hh)$ on line bundles
  is neither transitive nor faithful. We will see that in case $\Hh$
  is multiplicity free (all $e(x)=1$), the action is transitive. The
  question of faithfulness is strongly linked to the study of the
  \emph{ghost group} $\Gg(\Hh)$, the subgroup of $\Aut(\Hh)$, given by
  those $\sigma$ fixing the structure sheaf $L$ and all simple sheaves
  $S_x$. We also consider the automorphism group $\Aut(\XX)$, the
  subgroup of $\Aut(\Hh)$ given by those $\sigma$ fixing the structure
  sheaf $L$. All three, $\Pic(\Hh)$, $\Gg(\Hh)$, $\Aut(\XX)$, are
  normal subgroups of $\Aut(\Hh)$.\medskip

For the generalization of Picard-shifts with respect to
non-homogeneous tubes we refer to~\cite{lenzing:delapena:1999}
and~\cite{kussin:2009}. 
\end{numb}
\begin{numb}[Multiplicity and Comultiplicity]
  Let $L$ be a special line bundle so that $\Ext^1(S_x,L)\neq 0$ holds
  for all $x\in\XX$. The dimensions
  \begin{equation}
    \label{eq:def-e(x)}
    e(x)=[\Ext^1(S_x,L)\colon\End(S_x)]
  \end{equation}
  are called \emph{multiplicities}, \cite{ringel:1979},
  \cite{lenzing:delapena:1999}, \cite{kussin:2009}. In particular, we
  have the $S_x$-universal extension
  \begin{equation}
    \label{eq:S-universal-L}
    0\ra L\stackrel{\pi_x}\ra L(x)\ra{S_x}^{e(x)}\ra 0
  \end{equation}
  of $L$. The number
  \begin{equation}
    \label{eq:comultiplicity}
    e^{\ast}(x)=[\End(S_x):Z(\End(S_x))]^{1/2}
  \end{equation}
  we called \emph{comultiplicities} in~\cite{kussin:2009}, since (in
  case of genus zero) for almost all $x\in\XX$ the product of $e(x)$
  and $e^{\ast}(x)$ coincides with the skewness $s(\Hh)$,
  \cite[Cor.~2.3.5]{kussin:2009}. It was left open
  in~\cite{kussin:2009} whether $e(x)\cdot e^{\ast}(x)$ is always a
  divisor of $s(\Hh)$, not to speak about what the description of the
  cofactor could be. To answer this question, and without being
  restricted to the case of genus zero, was one of the main
  motivations for this article.

  We remark that the comultiplicity, like the skewness, can be
  expressed in terms of polynomial identity (PI) degree.
\end{numb}
\begin{numb}[Orbit algebras]
  In (noncommutative) algebraic geometry orbit algebras are important
  tools for constructing homogeneous coordinate rings. We refer to the
  survey~\cite{stafford:vandenbergh:2001}. If $E$ is an object in
  $\Hh$ and $\sigma\colon\Hh\ra\Hh$ an endofunctor, then we denote by
  $\Pi(E,\sigma)$ the positively $\ZZ$-graded \emph{orbit algebra}
  $\bigoplus_{n\geq 0}\Hom(E,\sigma^n E)$. The multiplication is
  defined on homogeneous elements $f\colon E\ra\sigma^mE$,
  $g\colon E\ra\sigma^n E$ by the rule
  $g\ast f=\sigma^m(g)\circ f\colon E\ra\sigma^{m+n}E$.

  The special cases we are interested in are $\Pi(L,\sigma_x)$ with
  $L$ the structure sheaf and $\sigma_x\colon\Hh\ra\Hh$ a
  Picard-shift. Then the homogeneous element $\pi_x$
  from~\eqref{eq:S-universal-L} is central,
  \cite[Lem.~1.7.1]{kussin:2009}. We denote the (homogeneous) ideal of
  $\Pi(L,\sigma_x)$ generated by $\pi_x$ by $P_x$. We will later see
  that $P_x$ is a homogeneous prime ideal. Whereas
  in~\cite{kussin:2009}, \cite{kussin:2008} we fixed one
  autoequivalence $\sigma$ (with additional good properties) and one
  coordinate algebra $\Pi(L,\sigma)$ for $\Hh$, we will in this paper
  for every point $x$ make use of its ``own'' orbit algebra
  $\Pi(L,\sigma_x)$ in order to investigate the numbers $e(x)$,
  $e^{\ast}(x)$ and $e_{\tau}(x)$.
\end{numb}
\begin{numb}[PI-degree]
  We will make use (in Section~\ref{sec:local-global-principle}) of
  some ring-theoretic tools like the polynomial identity (PI)
  degree. We will never use the original definition. Instead in our
  special situation we could take the following two properties~(i)
  and~(ii) as an equivalent definition for the PI-degree. If $R$ is a
  ring (always assumed to be associative and with identity) we denote
  by $Z(R)$ its centre.
  \begin{enumerate}
  \item[(i)] If $D$ is a skew field which is of finite dimension over
    its centre, then the PI-degree of $D$ equals the square root of
    this dimension, \cite[Thm.~1.5.23]{rowen:1980}. Moreover, the
    PI-degree of the matrix ring $\matring_n(D)$ is $n$ times the
    PI-degree of $D$, \cite[1.5.16]{rowen:1980}.
  \item[(ii)] If $R$ is a noetherian domain, so that its quotient
    division ring $Q(R)$ is of finite dimension over its centre, then
    Posner's theorem~\cite[Thm.~7]{amitsur:1967} tells us that the
    PI-degree of $R$ equals the PI-degree of $Q(R)$.
  \end{enumerate}
\end{numb}
By $\mod(R)$ we denote the category of finitely presented (right)
$R$-modules, by $\mod_0(R)$ the full subcategory of the modules of
finite length. Usually, finite length is equivalent to finite
dimension over the base-field $k$.\medskip

We conclude the section with a motivating simple but non-trivial
example, illustrating some of our results.
\begin{example}
  Let $R=\CC[X;Y,\sigma]$ be the twisted graded polynomial algebra
  over $k=\RR$.  Here the variables $X$ and $Y$ are of degree one, $X$
  central, and $Yz=\sigma(z)Y$ for all $z\in\CC$, where
  $\sigma(z)=\bar{z}$ is the complex conjugation. The quotient
  category $\Hh=\qgr(R)=\mod^{\ZZ}(R)/\mod^{\ZZ}_0(R)$ is a
  noncommutative regular projective curve. Its function field is
  $k(\Hh)=\CC(t,\sigma)$, which is the quotient division ring (of
  degree zero fractions) of $R$. The centre of $R$ is $\RR[X,Y^2]$,
  the centre of $k(\Hh)$ is $\RR(t^2)$. For the skewness we obtain
  $s(\Hh)=2$. In this example we have an explicit description for all
  the points of $\XX$. These correspond bijectively to the homogeneous
  prime elements in $R$ (up to scalars). With the exception of $Y$,
  all prime elements belong to the centre; they are listed in
  Table~\ref{tab:didactic-examp}. For a point $x$ we write $D_x$ for
  the endomorphism ring of the corresponding simple object $S_x$. Then
  $e^{\ast}(x)=[D_x:Z(D_x)]^{1/2}$. The number $e(x)$ coincides with
  the number of irreducible factors of the corresponding prime
  element. It is shown in~\cite[Cor.~5.4.4]{kussin:2009} (and will be
  again shown in this paper in a broader context) that the
  Auslander-Reiten translation $\tau$ is given as the product
  $\tau={\sigma_x}^{-1}{\sigma_y}^{-1}$ of two (inverse)
  Picard-shifts, where (from now on) $x$ and $y$ are the points
  corresponding to the primes $X$ and $Y$, respectively. It follows
  readily that $e_{\tau}(p)=1$ for all points $p\neq x,\,y$.
  Moreover, the ghost group $\Gg(\Hh)$ is shown to be of order $2$,
  generated by $\gamma={\sigma_x}{\sigma_y}^{-1}$.
  \begin{table}[h]
    \centering
    $$\begin{array}{l|cccc|c}
        \text{prime/point}\ x & D_x & e(x) & e^{\ast}(x) &
                                                           e_{\tau}(x)
        & D_x [[T,\tau^-]]\\ 
        \hline
        X,\,Y & \CC & 1 & 1 & 2 & \CC[[T,\sigma]]\\
        (Y-\sqrt{\alpha}X)(Y+\sqrt{\alpha}X),\
        \alpha>0 & \RR & 2 & 1 & 1 & \RR[[T]] \\ 
        Y^2-\alpha X^2,\  \alpha<0 & \HH& 1 & 2 & 1 & \HH[[T]] \\
        (Y^2-zX^2)(Y^2-\bar{z}X^2),\ z\in\CC\setminus\RR & \CC & 2 & 1
                                                         & 1 & \CC[[T]] 
     \end{array}$$
    \caption{$k(\Hh)=\CC(t,\sigma)$}
    \label{tab:didactic-examp}
  \end{table}
  It can be seen directly (though
  not trivially, cf.\ \cite[Cor.~5.4.3]{kussin:2009}) that
  $e_{\tau}(x)=2=e_{\tau}(y)$. Of course, having already computed
  $e(x)$ and $e^{\ast}(x)$ (and the same for $y$) it will generally
  follow from Theorem~\ref{thm:main-theorem}.\medskip

  This example is a special case of the treatment of genus zero curves
  in Section~\ref{sec:genus-zero}, and it is also a special case of
  what we call a Witt curve, treated in
  Section~\ref{sec:Witt-curves}. These are obtained by Klein surfaces
  (=certain quotients of compact Riemann surfaces) together with a
  so-called $\pm$-configuration. These were studied by Witt in his
  seminal paper~\cite{witt:1934}. We will show in
  Corollary~\ref{cor:e-tau=e-ramif} that the $\tau$-multiplicities
  coincide with the ramification indices of the function (skew)
  field. That in the present example $x$ and $y$ are the only
  ramification points (having index $2$) then also follows from Witt's
  work.\medskip

  We sketch another way for computing $e_{\tau}(x)$: we can localize
  $R$ with respect to the homogenous prime ideal $P=RX$, considering
  only fractions of degree zero, denoting this ring by $R_x$. This is
  (in this special case!) a local ring whose maximal (left and right)
  ideal $J$ is generated by $\pi=XY^{-1}$, which satisfies
  $\pi z=\bar{z}\pi$ for all $z\in\CC$. The $J$-adic completion then
  is easily seen to be
  $\widehat{R}_x=\CC[[\pi,\sigma]]\simeq\CC[[T,\sigma]]$ (as indicated
  in the table). The centre is given by $\RR[[\pi^2]]=\RR[[T^2]]$. In
  the language of valuations, $\pi$ is a uniformizer for the
  completion, $\pi^2$ a uniformizer of the centre. We see readily that
  the ramification index $e_{\ramif}(x)$ of $x$, defined by
  $\widehat{R}_x \pi^2=\widehat{J}^{e_{\ramif}(x)}$, equals $2$.
  Since the tube to $x$ is given by $\Uu_x=\mod_0(\widehat{R}_x)$ it
  is not difficult to see (cf.\ Theorem~\ref{thm:order-sigma-ramif})
  that the twist $\sigma$ induces the Picard-shift
  ${\sigma_x}_{|\Uu_x}$, restricted to the tube $\Uu_x$. Therefore
  $\tau$ acts like ${\sigma_x}^{-1}$ as functor on $\Uu_x$ with order
  $2$, that is, $e_{\tau}(x)=2$. (A similar argument holds for $y$,
  considering $\pi^{-1}=YX^{-1}$.) We will use the notion of the
  \emph{different} in Sec.~\ref{sec:dualizing-sheaf}, defined as the
  Weil divisor $\Delta=\sum_p (e_{\ramif}(p)-1)\cdot p$. In the
  present example it follows from the preceding computations that
  $\Delta=1x+1y$.
\end{example}

\section{Homogeneous coordinate rings and localizations}
\label{sec:localization}
We assume that $(\Hh,L)$ is a (non-weighted) noncommutative regular
projective curve over the field $k$. In this section we show that, via
the Serre construction, $\Hh$ is a noncommutative noetherian
projective scheme in the sense of Artin-Zhang~\cite{artin:zhang:1994},
and accordingly $\XX$ a projective spectrum. Moreover, via
localization we study rings locally at a point $x\in\XX$.
\begin{lemma}
  Each vector bundle has a line bundle filtration.
\end{lemma}
\begin{proof}
  We refer to~\cite[Prop.~1.6]{lenzing:reiten:2006}.
\end{proof}
\begin{lemma}\label{lem:cokernel-pi-n}
  Let $0\ra L\stackrel{\pi}\lra L(x)\ra S^e\ra 0$ be the $S$-universal
  sequence of $L$ with $S=S_x$ and $e=e(x)$. For $n\geq 1$ we have the
  exact sequence $$0\ra L\stackrel{\pi^n}\lra L(nx)\ra S[n]^e\ra 0.$$
\end{lemma}
\begin{proof}
  By induction on $n$. For $n=1$ the assertion is trivial. Let
  $n>1$. Write $\mu\colon 0\ra L\stackrel{\pi^n}\lra L(nx)\lra E\ra
  0$. By induction hypothesis, from the snake lemma we obtain that $E$
  appears as the middle term of a short exact sequence
  \begin{equation}
    \label{eq:cokernel-sequence}
    0\ra S^e \ra
    E\ra S[n-1]^e\ra 0. 
  \end{equation}
  Write $E=E_1\oplus\ldots\oplus E_m$ with $E_i=S[\ell_i]$
  indecomposable. By uniseriality we have $\Soc(E_i)=S$. This yields
  $S^e=\Soc (S^e)\subseteq\Soc(E)=S^m$, and thus $m\geq e$. On the
  other hand, assume that $m>e$. Let $u_i\colon S\ra
  S[\ell_i]=E_i\stackrel{j_i}\lra E$ a monomorphism. By the definition
  of $e=e(x)$, there are $f_1,\dots,f_m$ in $\End(S)$, not all of them
  zero, such that $0=\sum_{i=1}^m \mu\cdot u_i f_i
  =\mu\cdot\bigl(\sum_{i=1}^m u_i f_i\bigr)$. Denoting $\sum_{i=1}^m
  u_i f_i$ by $0\neq h\colon S\ra E$, the short exact sequence
  $\mu\cdot h$ splits, and we obtain, that $S$ embeds into $L(nx)$,
  which gives a contradiction. We conclude $m=e$. Let
  $R=\End(S[\infty])$ be the complete local ring with maximal ideal
  $\mathfrak{m}$ such that $\Uu=\mod_0(R)$. Since $S^e$ is annihilated
  by $\mathfrak{m}$ and $S[n-1]^e$ by $\mathfrak{m}^{n-1}$, we deduce
  from sequence~\eqref{eq:cokernel-sequence} that $E$ is annihilated
  by $\mathfrak{m}^n$, and thus all $\ell_i\leq n$. Since the length
  of $E$ is $n\cdot e$, we get $\ell_i =n$ for all $i$. This completes
  the proof of the lemma.
\end{proof}
\begin{lemma}\label{lem:nonzero-Hom}
  Let $E$ be an indecomposable vector bundle and $S$ be a simple
  sheaf. Then $\Hom(E,S)\neq 0$.
\end{lemma}
\begin{proof}
  Using connectedness of $\Hh$ this is shown like
  in~\cite[(S11)]{lenzing:delapena:1999}
  or~\cite[Cor.~IV.1.8]{reiten:vandenbergh:2002}.
\end{proof}
\begin{lemma}\label{lem:morphisms-line-bundles-large-n}
  Let $L$ and $L'$ be line bundles, and let $x\in\XX$ be a point. Then
  $\Hom(L(-nx),L')\neq 0$ for $n\gg 0$.
\end{lemma}
\begin{proof}
  By the preceding lemma we have an exact sequence $0\ra L(-nx)\ra
  L\ra S[n]^e\ra 0$ for each $n\geq 0$. Applying $\Hom(-,L')$ gives
  $0\ra\Hom(L,L')\ra\Hom(L(-nx),L')\ra\Ext^1(S[n]^e,L')\ra\Ext^1(L,L')$. By
  Lemma~\ref{lem:nonzero-Hom} we have $d:=\dim_k \Hom(L',S)>0$, and
  thus $\dim_k\Ext^1(S[n]^e,L')=dne\gg 0$ for $n\gg 0$. From this
  follows the claim.
\end{proof}
\begin{lemma}\label{lem:Serre-construction}
  For each $x\in\XX$ the pair $(L,\sigma_x)$ is ample in the sense
  of~\cite{artin:zhang:1994}. Accordingly,
  \begin{equation}
    \label{eq:Serre-construction}
    \Hh\simeq\frac{\mod^{\ZZ}(\Pi(L,\sigma_x))}{\mod_0^{\ZZ}(\Pi(L,\sigma_x))}.
  \end{equation}
  In particular, a noncommutative regular projective curve $\Hh$ is a
  noncommutative projective scheme in the sense of
  Artin-Zhang~\cite{artin:zhang:1994}.
\end{lemma}
\begin{proof}
  (Compare the proof of~\cite[Lem.~IV.4.1]{reiten:vandenbergh:2002})
  We have the inverse system $\dots\ra L(-2x)\ra L(-x)\ra L$ of
  subobjects with zero
  intersection. By~\cite[Lem.~IV.1.3]{reiten:vandenbergh:2002}) there
  is a line bundle $L'\subseteq L$ such that $\Ext^1(U,L)=0$ for all
  subobjects (line bundles) $U\subseteq L'$. Moreover, for $n\gg 0$ we
  have $L(-nx)\subseteq L'$, and we conclude $\Ext^1(L(-nx),L)=0$.

  Let $E\in\Hh$. Let $F\subseteq E$ be the largest subobject such
  there is an epimorphism $G:=\oplus_{i=1}^t L(-\alpha_ix)\ra F$, and
  let $C=E/F$. We assume that $C\neq 0$, and will show that this
  yields a contradiction. If $C$ is of finite length, then it follows
  from Lemma~\ref{lem:cokernel-pi-n} that a finite direct sum of
  copies of $L$ maps onto $C$. Thus we can assume that $C$ is a vector
  bundle, and it suffices to assume that $C$ is a line bundle. We have
  an exact sequence $0\ra K\ra G\ra F\ra 0$ with $G$ a finite direct
  sum of $\sigma_x$-shifts of $L$. By the preceding paragraph there is
  $n_0$ such that $\Ext^1(L(-nx),G)=0$, and then $\Ext^1(L(-nx),F)=0$
  for all $n\geq n_0$.

  By Lemma~\ref{lem:morphisms-line-bundles-large-n} we have a
  non-trivial morphism $L(-mx)\ra C$ for some $m\geq n_0$. Since
  $\Ext^1(L(-mx),F)=0$, this lifts to a non-trivial morphism
  $L(-mx)\ra E$, giving a contradiction.
\end{proof}
\begin{lemma}
  The homogeneous ideal $P_x$ in $\Pi(L,\sigma_x)$ generated by
  $\pi_x$ is prime.
\end{lemma}
\begin{proof}
  This follows like in~\cite[Thm.~1.2.3]{kussin:2009}. We only need to
  show that for $n\gg 0$ sufficiently large we have $\Hom(L,\tau
  L(-nx))=0$, as in~\cite[Lem.~1.2.2]{kussin:2009}. To this end, by
  Lemma~\ref{lem:morphisms-line-bundles-large-n} for $n\gg 0$ there is
  a non-zero morphism $g\colon\tau L(-nx)\ra L$. We assume that there
  is a non-zero morphism $f\colon L\ra\tau L(-nx)$. Both, $f$ and $g$,
  are monomorphisms, and $g\circ f\colon L\ra L$ is an isomorphism,
  thus $g$ is an isomorphism. Enlarging $n$ further, we see that there
  is $m>0$ such that $L$ and $L(mx)$ are isomorphic. But then,
  repeating the argument just given, also $(\pi_x)^m$ would be an
  isomorphism. But this is not true by Lemma~\ref{lem:cokernel-pi-n},
  giving a contradiction. Thus $\Hom(L,\tau L(-nx))=0$.
\end{proof}
\begin{lemma}
  For each $x\in\XX$ the ring $\Pi(L,\sigma_x)$ is a graded noetherian
  domain which has a central prime element $\pi_x$ of degree one, and
  the quotient division ring of degree-zero fractions $s^{-1}r$ (with
  $r,\,s$ homogeneous of the same degree, $s\neq 0$) is the function
  field $k(\Hh)$. 
\end{lemma}
\begin{proof}
  Noetherianness follows from the proof
  of~\cite[Prop.~1.4.4]{kussin:2009} also in this more general setting
  (right-noetherianness is also shown
  in~\cite[Thm.~4.5]{artin:zhang:1994}). Since non-zero morphisms
  between line bundles are monomorphisms, the orbit algebra
  $R=\Pi(L,\sigma_x)$ is a graded domain. By the preceding lemma the
  homogeneous element $\pi_x$ is central and prime. The assertion
  about the function field follows like
  in~\cite[Lem.~IV.4.1~Step~4]{reiten:vandenbergh:2002}. (We remark
  that like in~\cite[Lem.~IV.4.1~Step~3]{reiten:vandenbergh:2002} the
  Gelfand-Kirillov dimension of the finitely graded (in the sense
  of~\cite{artin:stafford:1995}) $k$-algebra $R$ is two, and
  then~\cite[Thm.~0.1]{artin:stafford:1995} implies that (NC~7)
  holds.)
\end{proof}
\begin{remark}\label{rem:ample-pair-weighted}
  Assume that $\Hh$ satisfies, more generally, conditions (NC~1) to
  (NC~5), and let $L$ be a line bundle. Then similar statements of
  most of the preceding results remain true, with similar proofs. More
  precisely, for $\sigma$ a suitable product of Picard-shifts, we get
  an ample pair $(L,\sigma)$, and $R=\Pi(L,\sigma)$ is a projective
  coordinate algebra for $\Hh$ of Gelfand-Kirillov dimension two, and
  the zero component of the graded quotient division ring of $R$ is
  the function field $k(\Hh)$; we refer
  to~\cite[Lem.~IV.4.1]{reiten:vandenbergh:2002}. Then~\cite[Thm.~0.1]{artin:stafford:1995}
  implies (NC~7).
\end{remark}
\begin{lemma}
  Let $x\in\XX$ be of multiplicity $e(x)$ and with simple sheaf $S_x$.
  \begin{enumerate}
  \item[(1)] For a non-zero homogeneous element $s\in\Pi(L,\sigma_x)$
    the following conditions are equivalent:
  \begin{itemize}
  \item $s\in\Cc(P_x)$, that is, $s$ is regular modulo $P_x$.
  \item The cokernel of $s$ lies in $\coprod_{y\neq x}\Uu_y$.
  \item The fibre map $s_x\in\End({S_x}^{e(x)})$ is an isomorphism.
  \end{itemize}
  \item[(2)] The set $\Cc(P_x)$ is a denominator set.
  \item[(3)] For the graded localization
    $R^{\gr}_x=\Pi(L,\sigma_x)_{\Cc(P_x)}$ the graded Jacobson radical
    is generated by the central element $\pi_x 1^{-1}$ and is the only
    non-zero graded prime ideal.
  \item[(4)] As graded rings,
    $R_x^{\gr}/\Rad^{\gr}(R_x^{\gr})\simeq\matring_{e(x)}\bigl(\END(S_x)\bigr)$,
    where $\END(S_x)$ is the graded skew field
    $\bigoplus_{n\in\ZZ}\Hom(S_x,S_x(n))$.
  \end{enumerate}
\end{lemma}
\begin{proof}
  (1) The equivalence of the three conditions is shown
  in~\cite[Lem.~2.2.1]{kussin:2009}.

  (2) Since $\Pi(L,\sigma_x)$ is graded noetherian (left and right)
  and $\pi_x$ is a central, this follows from a graded version
  of~\cite[Thm.~4.3]{smith_pf:1982}.

  (3), (4) (For analogous ungraded statements we refer
  to~\cite[Thm.~4.3.18]{mcconnell:robson:2001}
  and~\cite[Lem.~14.18]{goodearl:warfield:2004}.) Localizing the
  universal exact sequence $0\ra L\stackrel{\pi_x}\lra L(x)\lra
  {S_x}^{e(x)}\ra 0$ we get, like in~\cite[Prop.~2.2.8]{kussin:2009}, a
  short exact sequence $0\ra R_x^{\gr}\stackrel{\cdot\pi_x}\lra
  R_x^{\gr}(x)\lra {S_x}^{e(x)}\ra 0$ of graded $R_x^{\gr}$-modules,
  where $S_x$ is simple. As graded rings thus $R_x^{\gr}/(\pi_x
  1^{-1})\simeq\matring_{e(x)}^{\gr}(\END(S_x))$. The graded Jacobson
  radical $\Rad^{\gr}(R_x^{\gr})$ is the principal ideal generated by
  $\pi_x 1^{-1}$: clearly, $1-\pi_x r\in\Cc(P_x)$ for each $r$, so
  that $\pi_x 1^{-1}$ lies in the radical. The canonical surjective
  ring homomorphism $R_x^{\gr}/(\pi_x 1^{-1})\ra
  R_x^{\gr}/\Rad^{\gr}(R_x^{\gr})$ is an isomorphism, by simplicity of
  the graded ring on the left hand side.
\end{proof}
We denote by $R_x$ the degree-zero component of the localization
$\Pi(L,\sigma_x)_{\Cc(P_x)}$.
\begin{proposition}\label{prop:Rx-modulo-radical}
  Let $\XX$ be a noncommutative regular projective curve over the
  field $k$. Let $x\in\XX$ be a point. There is an
  isomorphism $$R_x/\Rad(R_x)\simeq\matring_{e(x)}\bigl(\End(S_x)\bigr)$$
  of rings.
\end{proposition}
\begin{proof}
  Since $\Rad(R_x)$ is the degree zero part of
  $\Rad^{\gr}(R_x^{\gr})$, the assertion follows from the preceding
  lemma.
\end{proof}
\begin{numb}\label{numb:categorical-localization}
  For $x\in\XX$ we denote by
  $\Hh_x=\Hh/\spitz{\coprod_{y\neq x}\Uu_y}$ the quotient category,
  modulo a Serre subcategory, where all tubes except $\Uu_x$ are
  ``removed'', and by $p_x\colon\Hh\ra\Hh_x$ the quotient functor.
\end{numb}
\begin{lemma}
  The object $L_x=p_x(L)$ is an indecomposable projective generator of
  $\Hh_x$. Accordingly, for the ring $V_x=\End_{\Hh_x}(L_x)$ we have
  $\Hh_x\simeq\mod(V_x)=\mod_+(V_x)\vee\mod_0(V_x)$, with
  $\mod_0(V_x)\simeq\Uu_x$ the finite length modules and $\mod_+(V_x)$
  the finitely generated torsionfree modules.
\end{lemma}
\begin{proof}
  Let $y\in\XX$ be another point, $y\neq x$. Then $\pi_y$ induces an
  isomorphism $L_x(-y)\simeq L_x$. Using ampleness of the pair
  $(L,\sigma_y)$ we see that $L_x$ is a generator for $\Hh_x$. It is
  easy to see that for each exact sequence $\eta\colon 0\ra A\ra B\ra
  L\ra 0$ in $\Hh$ the exact sequence $p_x(\eta)$ in $\Hh_x$ splits,
  showing that $L_x$ is a projective object in $\Hh_x$. From this it
  follows that $\Hom_{\Hh_x}(L_x,-)\colon\Hh_x\ra\mod(V_x)$ is an
  equivalence. It is easy to see that $p_x$ induces an injective
  homomorphism $V_x\ra k(\Hh)$ of rings, and thus $V_x$ is a
  domain. We infer that $L_x$ is indecomposable.
\end{proof}
\begin{proposition}
  There is an isomorphism of rings $R_x\simeq V_x$.
\end{proposition}
\begin{proof}
  By using the definition of morphisms in the quotient category we see
  easily $R_x\subseteq V_x$. Let $0\ra L'\stackrel{s}\lra L\ra C\ra 0$
  be an exact sequence in $\Hh$ with $L'$ a line bundle and
  $C\in\coprod_{y\neq x}\Uu_y$. By ampleness of $(L,\sigma_x)$ there
  is an epimorphism $f=(f_1,\dots,f_n)\colon\bigoplus_{i=1}^n
  L(-\alpha_i x)\ra L'$ (with $\alpha_i\geq 1$). If we assume that
  each $C_i=\Coker(f_i)$ has a non-zero summand in $\Uu_x$, then there
  is an epimorphism $C_i\ra S_x$, and thus we can write
  $f_i=\pi_x\circ f'_i$. But then $f=\pi_x\circ f'$ is not surjective,
  giving a contradiction. Thus there is $i$ such that
  $C_i\in\coprod_{y\neq x}\Uu_y$. We conclude that there is $f\colon
  L(-\alpha x)\ra L'$ such that $s\circ f\colon L(-\alpha x)\ra L$ is
  a non-zero homogeneous element in $\Pi(L,\sigma_x)$ with
  $\Coker(sf)\in\coprod_{y\neq x}\Uu_y$, that is,
  $sf\in\Cc(P_x)$. Thus we can write $rs^{-1}=(rf)(sf)^{-1}$, from
  which we infer the converse inclusion.
\end{proof}
\begin{corollary}
  For each $x\in\XX$ we have $\Uu_x\simeq\mod_0(R_x)$. \qed
\end{corollary}
\begin{corollary}\label{cor:nc-dedekind}
  Each ring $R_x$ is a noncommutative Dedekind domain with unique
  non-zero prime ideal given by $\Rad(R_x)$.
\end{corollary}
\begin{proof}
  $R_x$ is right hereditary since $\mod(R_x)\simeq\Hh_x$ is
  hereditary. Since $R_x$ is noetherian, by~\cite[Cor.~3]{small:1967}
  it is also left hereditary. It follows from
  Proposition~\ref{prop:Rx-modulo-radical} that the radical
  $J=\Rad(R_x)$ is the only (two-sided) maximal ideal. As
  in~\cite[4.3.20]{mcconnell:robson:2001} one shows $\bigcap_{n\geq 0}
  J^n=0$. Thus, if $r\in R_x$, $r\neq 0$, then there is $v(r)=n$ with
  $r\in J^n$ but $r\not\in J^{n+1}$. Let $I$ be a non-zero idempotent
  ideal in $R_x$. Let $0\neq r\in I$ with $v(r)$ minimal. From the
  condition $I=I^2$ we get $v(r)=0$, and then $I=R_x$ since $J$ is
  maximal. Thus $R_x$ is Dedekind
  by~\cite[5.6.3]{mcconnell:robson:2001}. By~\cite[5.2.9]{mcconnell:robson:2001}
  each non-zero ideal is of the form $J^n$, and it follows that $J$ is
  the only non-zero prime ideal.
\end{proof}
We also consider the category
$\QHh=\Qcoh\XX=\frac{\Mod^{\ZZ}(\Pi(L,\sigma_x))}{\Mod_0^{\ZZ}(\Pi(L,\sigma_x))}$
of quasicoherent sheaves, where $\Mod_0$ denotes the localizing Serre
subcategory of torsion (that is, locally finite length) graded
modules. This is a hereditary, locally noetherian Grothendieck
category. In this we can consider the \emph{Pr\"ufer sheaf}
$S_x[\infty]$ for $x\in\XX$, which is the union
$\bigcup_{n\geq 1}S_x[n]$, that is, the direct limit of the direct
system $(S_x[n],\iota_n)$, and thus is a quasicoherent torsion sheaf. We
now have the main result of this section.
\begin{proposition}\label{prop:main-ring-iso}
  For the $\Rad(R_x)$-adic completion of $R_x$ we
  have $$\widehat{R}_x\simeq\matring_{e(x)}\bigl(\End(S_x[\infty])\bigr).$$ 
\end{proposition}
\begin{proof}
  By~\cite[Thm.~21.31]{lam:1991} the completion $\widehat{R}_x$ is
  semiperfect, it satisfies, with $D_x=\End(S_x)$,
  \begin{equation}
    \label{eq:completion-mod-rad}
    \widehat{R}_x/\Rad(\widehat{R}_x)\simeq
    R_x/\Rad(R_x)\simeq\matring_{e(x)}(D_x),
  \end{equation}
  and from~\cite[Thm.~23.10]{lam:1991} it follows that
  $\widehat{R}_x\simeq\matring_{e(x)}(\widehat{E}_x)$ for a complete
  local ring $\widehat{E}_x$. Moreover, for the categories of finite
  length modules we have
  $\mod_0(\widehat{E}_x)\simeq\mod_0(\widehat{R}_x)\simeq\mod_0(R_x)\simeq\Uu_x$. The
  result follows now, since the complete local ring
  $\End(S_x[\infty])$ is uniquely determined such that
  $\mod_0(\End(S_x[\infty]))\simeq\Uu_x$,
  by~\cite[IV.~Prop.~13]{gabriel:1962}.
\end{proof}

\section{Serre duality and the bimodule of a homogeneous tube}
\label{sec:bimodule-of-tube}
In~\cite{gabriel:1973} P.\ Gabriel defined the \emph{species} of a
uniserial category $\Uu$. In the most basic situation, when there is
(up to isomorphism) only one simple object $S$ in $\Uu$, like in the
case of a homogeneous tube, then this species is just the bimodule
${}_{\End(S)}\Ext^1(S,S)_{\End(S)}$. In order to describe this
bimodule more precisely, we derive from~\cite{lenzing:zuazua:2004}
some general facts about Serre duality.

We call a $k$-bilinear map $\PP{-}{-}\colon V\times W\ra k$ a
\emph{perfect pairing}, if for each non-zero $x\in V$ there exists
$y\in W$ with $\PP{x}{y}\neq 0$, and if for each non-zero $y\in W$
there is $x\in V$ with $\PP{x}{y}\neq 0$.  Let $\Hh$ be a
noncommutative regular projective curve over the field $k$.  For each
indecomposable object $X\in\Hh$ we fix an almost split sequence
$\mu_X\colon 0\ra\tau X\ra E\ra X\ra 0$ and a $k$-linear map
$\kappa_X\colon\Ext^1(X,\tau X)\ra k$ with $\kappa_X(\mu_X)\neq
0$. Similarly, for $Y\in\Hh$ indecomposable and an almost split
sequence $\mu_{\tau^{-}Y}\colon 0\ra Y\ra F\ra\tau^{-}Y\ra 0$ we fix
$\kappa_{\tau^{-}Y}\colon\Ext^1(\tau^{-}Y,Y)\ra k$ with
$\kappa_{\tau^{-}Y}(\mu_{\tau^{-}Y})\neq
0$. Then
$$\PP{-}{-}\colon\Ext^1(X,Y)\times\Hom(\tau^{-}Y,X)\ra k,\
(\eta,f)\mapsto\kappa_{\tau^{-}Y}(\eta\cdot f)$$ is a perfect pairing,
and similarly so is
$$\PP{-}{-}\colon\Hom(Y,\tau X)\times\Ext^1(X,Y)\ra k,\
(g,\eta)\mapsto\kappa_{X}(g\cdot\eta).$$ From these perfect pairings
we obtain Serre duality
\begin{equation}
  \label{eq:serre-duality-isos}
  \Hom(Y,\tau X)\stackrel{\psi_{XY}}\longrightarrow\D\Ext^1(X,Y)\stackrel{\phi_{XY}}
  \longleftarrow\Hom(\tau^{-}Y,X),
\end{equation}
where $\psi_{XY}\colon
f\mapsto\PP{f}{-}$ and $\phi_{XY}\colon g\mapsto\PP{-}{g}$ are
isomorphisms, natural in $X$ and $Y$.
\begin{proposition}\label{prop:AR-sequence}
  Let $X\in\Hh$ be indecomposable such that $\End(X)$ is a skew
  field. Denote by $\mu\colon 0\ra\tau X\stackrel{u}\ra
  E\stackrel{v}\ra X\ra 0$ the almost split sequence ending in
  $X$. For all $f\in\End(X)$ we have $$\tau(f)\cdot\mu=\mu\cdot f.$$
\end{proposition}
\begin{proof}
  The isomorphism $\psi_{XY}$ from~\eqref{eq:serre-duality-isos} is
  natural in $X$ and $Y$ and thus, in particular, an isomorphism of
  $\End(X)-\End(Y)$-bimodules. Then we have the following rules:
  \begin{equation*}
    \PP{f}{g\eta}=\PP{fg}{\eta},\quad\PP{f\eta}{g}=\PP{f}{\eta
      g},\quad\PP{\eta}{fg}=\PP{\eta f}{g},\quad\PP{f}{\eta g}=\PP{\tau(g)f}{\eta}
  \end{equation*}
  (compare~\cite[(3.2)]{lenzing:zuazua:2004}). The last equality is
  just $\End(X)$-linearity. Moreover, by definition of the
  $\End(X)-\End(Y)$-bimodule structure on $\D\Ext^1(X,Y)$ we have
  $\PP{g}{\eta f}=f\cdot\PP{g}{\eta}$ for all $f\in\End(X)$. Let now
  $Y=\tau X$ and $\mu\in\Ext^1(X,\tau X)$ be the almost split
  sequence. Since $D=\End(X)\simeq\End(\tau X)$ is a skew field,
  $M=\Ext^1(X,\tau X)$ is a onedimensional $D-D$-bimodule, in
  particular $D\mu=M=\mu D$. Thus for each $f\in D=\End(X)$ there is a
  unique $f'\in D=\End(\tau X)$ such that $f'\mu=\mu f$. We have to
  show that $f'=\tau(f)$. Let $\eta\in\Ext^1(X,\tau X)$. Then there is
  an $h\in\End(X)$ such that $\eta=\mu\cdot h$. First we have
  $$\PP{f'}{\mu}=\PP{1}{f'\mu}=\PP{1}{\mu
    f}=\PP{\mu}{f}=\PP{\tau(f)}{\mu},$$ and
  then $$\PP{f'}{\eta}=\PP{f'}{\mu\cdot
    h}=h\cdot\PP{f'}{\mu}=h\cdot\PP{\tau(f)}{\mu}=\PP{\tau(f)}{\mu\cdot
    h}=\PP{\tau(f)}{\eta}.$$ Since $\PP{-}{-}$ is a perfect pairing we
  conclude $f'=\tau(f)$, finishing the proof.
\end{proof}
\begin{corollary}\label{cor:AR-f-tau-f}
  Let $\End(X)$ be a skew field. Let $f\in\End(X)$ such that there is a
  commutative diagram $$\xymatrix{ \mu\colon & 0 \ar @{->}[r] & \tau
    X\ar @{->}[r]^-{u} \ar @{->}[d]_-{f'} & E\ar @{->}[r]^-{v} \ar
    @{->}[d]^-{g} & X \ar @{->}[r] \ar @{->}[d]^-{f} & 0\\
    \mu\colon & 0 \ar @{->}[r] & \tau X\ar @{->}[r]^-{u} & E\ar
    @{->}[r]^-{v} & X\ar @{->}[r] & 0.}$$ Then $f'=\tau(f)$ holds.
\end{corollary}
\begin{proof}
  We show that from the assumptions $f'\cdot\mu=\mu\cdot f$
  follows. By the preceding proposition then
  $f'\cdot\mu=\tau(f)\cdot\mu$, thus $(f'-\tau(f))\cdot\mu=0$; since
  $\mu\neq 0$ and $\End(\tau X)$ is a skew field, this yields
  $f'=\tau(f)$.

  If $f\neq 0$, then $f$ is an isomorphism, and since $\mu$ does not
  split, $f'\neq 0$ follows easily. Dually, if $f=0$, then also $f'=0$
  holds. Thus we may assume that $f$, and then also $f'$ and $g$ are
  isomorphisms. One computes that $$f'\cdot\mu\colon 0\ra\tau
  X\stackrel{u}\longrightarrow E\stackrel{f^{-1}v}\longrightarrow X\ra
  0$$ and $$\mu\cdot f\colon 0\ra\tau
  X\stackrel{uf'^{-1}}\longrightarrow E\stackrel{v}\longrightarrow
  X\ra 0.$$ Then we have the commutative exact diagram:
  $$\xymatrix{ f'\cdot\mu\colon & 0 \ar @{->}[r] & \tau
    X\ar @{->}[r]^-{u} \ar @{=}[d] & E\ar @{->}[r]^-{f^{-1}v} \ar
    @{->}[d]^-{g^{-1}} & X \ar @{=}[d] \ar @{->}[r] & 0\\
    \mu\cdot f\colon & 0 \ar @{->}[r] & \tau X\ar @{->}[r]^-{uf'^{-1}} & E\ar
    @{->}[r]^-{v} & X\ar @{->}[r] & 0,}$$ thus $f'\cdot\mu=\mu\cdot f$
  as claimed.
\end{proof}
As a special case we get the following description of the bimodule of
a homogeneous tube. 
\begin{corollary}
  Let $\Uu$ be a homogeneous tube in $\Hh$ with simple object $S=\tau
  S$, almost split sequence $\mu\colon 0\ra S\ra S[2]\ra S\ra 0$ and
  division algebra $D=\End(S)$. Then the bimodule of $\Uu$, that is,
  the $D$-$D$-bimodule $E=\Ext^1(S,S)$, is given by
  $E=D\cdot\mu=\mu\cdot D$ with relations $\mu\cdot d=\tau(d)\cdot\mu$
  (for all $d\in D$), where $\tau\in\Aut(D/k)$ is induced by the
  Auslander-Reiten translation. \qed
\end{corollary}

\section{Tubes and their complete local rings}
\label{sec:complete-local-rings}
Let $k$ be a field. If $D$ is a division algebra over $k$ and
$\sigma\in\Aut(D/k)$, then we denote by $D[[T,\sigma]]$ the ring of
formal power series $\sum_{n\geq 0}a_n T^n$ over $D$, subject to the
relation $Ta=\sigma(a)T$ for all $a\in D$. Such rings occur naturally
in the study of generalized uniserial algebras over a perfect field,
cf.\ \cite{kupisch:1975}. 

Let $(R,\mathfrak{m})$ be a (not necessarily commutative) local ring
with Jacobson radical $\mathfrak{m}$. We write
$\gr(R)=\bigoplus_{n\geq 0}\mathfrak{m}^n/\mathfrak{m}^{n+1}$. This is
a graded local ring, with graded Jacobson radical given by
$\gr_+(R)=\bigoplus_{n\geq 1}\mathfrak{m}^n/\mathfrak{m}^{n+1}$. Its
$\gr_+(R)$-adic completion is given by $\wgr(R)=\prod_{n\geq
  0}\mathfrak{m}^n/\mathfrak{m}^{n+1}$, with multiplication given by
the Cauchy product.
\begin{proposition}\label{prop:complete-local-rings-general-homogeneous}
  Let $k$ be a field. Let $\Uu$ be a (homogeneous) tube over a
  noncommutative regular projective curve over $k$ with simple object
  $S$ and $D=\End(S)$ the endomorphism skew field. Let
  $\tau\in\Gal(D/k)$ be the automorphism (modulo inner) induced by the
  Auslander-Reiten translation $\tau$. Let $S[\infty]$ be the
  corresponding Pr\"ufer sheaf and $R=\End(S[\infty])$ its
  endomorphism ring. 

  Then $R$ is a complete local domain with maximal ideal
  $\mathfrak{m}=R\pi=\pi R$, where $\pi$ is a surjective endomorphism
  of $S[\infty]$ having kernel $S$. Each one-sided ideal is two-sided,
  and if non-zero then of the form $\mathfrak{m}^n=R\pi^n=\pi^n
  R$. Moreover, $\Uu\simeq\mod_0 (R)$, and there are isomorphisms
  $$\gr(R)\simeq D[T,\tau^{-}]\quad\text{(of graded rings)
    and}\quad\wgr(R)\simeq D[[T,\tau^{-}]].$$
\end{proposition}
\begin{proof}
  (1) With the notations from~\ref{nr:tubes}, the Pr\"ufer object
  $S[\infty]$ is the direct limit of the direct system
  $(S[n],\iota_n)$. The direct limit closure $\QUu$ of $\Uu$ in $\QHh$
  is a hereditary locally finite Grothendieck category in which
  $S[\infty]$ is an indecomposable injective cogenerator. Its
  endomorphism ring $R$ is the inverse limit of the inverse system of
  rings $(\End(S[n],p_n)$, where the $p_n$ are the surjective
  restriction maps.

  It is well-known (we refer to~\cite{amdal:ringdal:1968ab},
  \cite{ringel:1975}, \cite{ringel:1979}, \cite{gabriel:1973},
  and~\cite[Prop.~4.10]{vandenbergh:2001}) that $R$ is a complete
  local domain with $\Uu\simeq\mod_0(R)$, having the properties stated
  in the proposition. The two isomorphisms of rings remain to
  show. Since $\Uu$ is hereditary, $\wgr(R)$ is,
  by~\cite[8.5]{gabriel:1973}, isomorphic to the complete tensor
  algebra (\cite[7.5]{gabriel:1973}) $\Omega$ of the species of $\Uu$,
  which is given by the $D$-$D$-bimodule $E=\Ext^1(S,S)$.

  (2) We now determine the complete tensor algebra of the bimodule
  $E$. Let $\mu\in E$ denote the almost split sequence $0\ra S\ra
  S[2]\ra S\ra 0$. We have $E=D\mu=\mu D$, and from
  Proposition~\ref{prop:AR-sequence} we get $\mu\cdot
  d=\tau(d)\cdot\mu$ for each $d\in D$. For each natural number $n$
  there is a canonical isomorphism ${}_{D}D^n_D\otimes{}_{D}E_D\simeq
  ({}_{D}D_D\otimes{}_{D}E_D)^n\simeq{}_{D}E^n_D$, where $E^n$ as left
  $D$-module is isomorphic to $D^n$, and the right $D$-module
  structure on $E^n$ is given by $\tau$-twist: $(x_1,\dots,x_n)\cdot
  d=(x_1\tau(d),\dots,x_n\tau(d))$.

  We denote by $\Uu'$ the category of \emph{small representations}
  (\cite{gabriel:1973}) of the species given by the division ring $D$
  and one loop labelled by the $D$-$D$-bimodule $E$. The
  indecomposable of length $n$ is given by $S'[n]$, which is the
  representation $D^n\stackrel{g}\lra E^n$, with $g$ nilpotent right
  $D$-linear and given by the indecomposable Jordan matrix
$$J_n=J_n(0)=
\begin{pmatrix}
  0 & 1 & 0 & \dots & 0\\
  0 & 0 & 1 & \dots & 0\\
  \vdots & \vdots & & \ddots & \vdots\\
  0 & 0 & 0 & \dots & 1\\
  0 & 0 & 0 & \dots & 0
\end{pmatrix}
$$ to the eigenvalue $0$. That is, we have $$g(
\begin{pmatrix}
  x_1\\
  \vdots\\
  x_n
\end{pmatrix})=J_n \cdot\begin{pmatrix}
  \tau(x_1)\\
  \vdots\\
  \tau(x_n)
\end{pmatrix}.$$ If $$\xymatrix{D^n \ar @{->}[d]_-{f}\ar @{->}[r]^-{g}
  & E^n \ar @{->}[d]^-{f'} \\
  D^n \ar @{->}[r]^-{g} & E^n}$$ is an endomorphism of $S'[n]$, where
$f$ and $f'$ are right $D$-linear maps given by the same $n\times
n$-matrix $A=(a_{ij})$ with entries in $D$, then for the matrices
$A\cdot J_n=J_n\cdot A^{\tau}$ holds, where
$A^{\tau}=(\tau(a_{ij}))$. Similar relations hold for morphisms of
representations. This yields that we can $A$ write as
\begin{equation}
  \label{eq:LKB}
  A=a_0\cdot I_n+a_1 \cdot J_n+a_2\cdot (J_n)^2+\dots+ a_{n-1}\cdot
  (J_n)^{n-1}
\end{equation}
  with unique $a_0,\dots,a_{n-1}\in D$, and where
  $a\cdot (J_n)^{\ell}$ is given by replacing the side-diagonal given
  by $1,\,1,\dots,1$ by the elements
  $a,\,\tau^{-}(a),\dots,\tau^{-(n-\ell-1)}(a)$. Clearly $(a\cdot
  I_n)\cdot (b\cdot I_n)=(ab)\cdot I_n$ holds, and $D\cdot I_n$ forms
  a subalgebra of $\End(S'[n])$ isomorphic to
  $D$. Moreover,
  \begin{equation}
    \label{eq:T-relation}
    (J_n)^{\ell}\cdot (a\cdot I_n)=\tau^{-\ell}(a)\cdot
  (J_n)^{\ell}
  \end{equation}
  holds, and the Jacobson radical of $\End(S'[n])$ is generated, as a
  left ideal and as a right ideal, by the map with matrix $J_n$.

  We denote by $\iota'_n\colon
  S'[n]\ra S'[n+1]$ and $\pi'_n\colon S'[n+1]\ra S'[n]$ the morphisms given
  by the matrices
$$\begin{pmatrix}
  I_n \\
 \hline
 0  
\end{pmatrix}\quad\text{and}\quad
(0 \mid I_n),
$$ respectively. This yields that $\iota'_n\circ\pi'_n$ is given by
multiplication with $J_{n+1}$, and $\pi'_n\circ\iota'_n$ by $J_n$. If
$f'$ is the restriction of the endomorphism $f$ to $S'[n-1]$, that is,
$\iota'_{n-1}\circ f'=f\circ\iota'_{n-1}$ holds, and if $A'$ denotes
the corresponding $(n-1)\times (n-1)$-matrix, then
equation~\eqref{eq:LKB} yields $$A'=a_0\cdot I_{n-1}+a_1 \cdot
J_{n-1}+a_2\cdot (J_{n-1})^2+\dots+ a_{n-2}\cdot (J_{n-1})^{n-2}.$$
Thus, sending $f$ to its restriction $f'$, gives a, clearly
surjective, homomorphism $p'_n\colon\End(S'[n])\ra\End(S'[n-1])$ of
$k$-algebras.

Starting with the almost split sequence $\mu'_1$ with end term
$S'=S'[1]$ we have the following direct system of short exact
sequences
  $$\xymatrix{ \mu'_1\colon & 0 \ar
    @{->}[r] & S'\ar @{->}[r]^-{\iota'_1}\ar @{=}[d] & S'[2]\ar
    @{->}[r]^-{\pi'_1} \ar
    @{->}[d]^-{\iota'_2} & S' \ar @{->}[r] \ar @{->}[d]^-{\iota'_1} & 0\\
    \mu'_2\colon & 0 \ar @{->}[r] & S'\ar @{->}[r]\ar @{=}[d] &
    S'[3]\ar @{->}[r]^-{\pi'_2} \ar
    @{->}[d]^-{\iota'_3} & S'[2] \ar @{->}[r] \ar @{->}[d]^-{\iota'_2} & 0\\
    \mu'_3\colon & 0 \ar @{->}[r] & S'\ar @{->}[r]\ar @{=}[d] &
    S'[4]\ar @{->}[r]^-{\pi'_3} \ar @{->}[d]^-{\iota'_4} & S'[3] \ar
    @{->}[r] \ar @{->}[d]^-{\iota'_3} &
    0\\
    \vdots & & \vdots & \vdots & \vdots & }$$ and its direct limit
  $$\mu'_{\infty}\colon\quad\quad 0\ra
  S'\ra S'[\infty]\stackrel{\pi'}\lra S'[\infty]\ra 0.$$ The ring
  $R'=\End(S'[\infty])$ is isomorphic to the inverse limit of the
  $\End(S'[n])$ (with respect to the inverse system given by the
  $p'_n$). As in~(1), $R'$ is a complete local ring with maximal ideal
  $\mathfrak{m}'=R'\pi'=\pi' R'$, and with
  $\mod_0(R')\simeq\Uu'$. Thus there is an automorphism
  $\sigma\in\Aut(R'/k)$ with $\pi' f=\sigma(f)\pi'$ for all $f\in R'$.
  
  Each $f\in\End(S'[\infty])$ has a unique expression
  $f=(f_1,f_2,f_3,\dots)$ with $f_n\in\End(S'[n])$ and
  $f_n=f_{|S'[n]}$ the restriction of $f$ to $S'[n]$ for each $n$. The
  restriction of $\pi'$ to $S'[n]$ is given by
  $\iota'_{n-1}\circ\pi'_{n-1}$, hence by the matrix $J_n$. We
  conclude that $f$ has a unique expression as formal power
  series $$f=\sum_{n=0}^{\infty}a_n\pi'^n.$$
  From~\eqref{eq:T-relation} we deduce $$\pi' a=\tau^-(a)\pi'$$ for
  all $a\in D$. Thus $R'=D[[\pi',\tau^-]]\simeq D[[T,\tau^-]]$. On the
  other hand, by~\cite[7.5]{gabriel:1973} the complete tensor algebra
  $\Omega$ of $E$ is a complete local ring also satisfying
  $\Uu'\simeq\mod_0(\Omega)$. It follows,
  by~\cite[IV.~Prop.~13]{gabriel:1962}, that
  $R'\simeq\Omega\simeq\wgr{R}$. Finally this yields $\gr(R)\simeq
  D[T,\tau^-]$ as graded rings.
\end{proof}
In the separable (perfect) case, we can apply the Wedderburn-Malcev
theorem~\cite[Thm.~11.6]{pierce:1982} in order to get the
following. 
\begin{proposition}\label{prop:R=wgrR}
  With the notations of the preceding proposition, assume that $D/k$
  is separable (that is, $Z(D)/k$ is a separable field
  extension). Then $R\simeq\wgr(R)$.
\end{proposition}
In different words, $R$ is obtained as the complete tensor algebra of
the species of the tube $\Uu$, so that $\Uu$ can be recovered from its
species. Without the separability assumption the statement is wrong in
general, cf.\ Example~\ref{ex:inseparable}.
\begin{proof}
  The proof is based on~\cite[8.4]{gabriel:1973}. For $n\geq 1$ let
  $B_n$ be the finite dimensional $k$-algebra $\End(S[n])\simeq
  R/\mathfrak{m}^n$, where $R$ is the endomorphism ring of $S[\infty]$
  with maximal ideal $\mathfrak{m}$. The Wedderburn-Malcev theorem
  implies that the projection $B_n \ra B_n/\Rad(B_n)\simeq D$
  splits. Thus $B_n=D_n\oplus\Rad(B_n)$, with a subalgebra $D_n$ of
  $B_n$ isomorphic to $D$. Then $B_n$ becomes a $D$-$D$-bimodule, and
  $\Rad(B_n)$ contains a subbimodule which is isomorphic to
  $V_n=\Rad(B_n)/\Rad^2(B_n)$. Thus there is a surjective homomorphism
  from the tensor algebra of $V_n$, and then also from the complete
  tensor algebra $\Omega$ of the species of $\Uu$, onto $B_n$. We get
  an isomorphism $\Omega/\Rad^n(\Omega)\simeq B_n$. A more detailed
  analysis shows that this can be done inductively in such a way that
  we obtain an isomorphism of inverse systems of rings. Taking inverse
  limits we get $\Omega\simeq R$. This finishes the proof.
\end{proof}
If $\Uu=\Uu_x$ satisfies this separability condition, we call $x$
(resp.\ $\Uu$) a \emph{separable} point (tube). If $k$ is a perfect
field, then all points are separable.
\begin{theorem}\label{thm:pruefer-end-skew-power}
  Let $k$ be a field. Let $\Uu$ be a \emph{separable} tube over a
  noncommutative regular projective curve over $k$ with simple object
  $S$ and $D=\End(S)$ the endomorphism skew field. Let
  $\tau\in\Gal(D/k)$ be the automorphism induced by the
  Auslander-Reiten translation $\tau$. Let $S[\infty]$ be the
  corresponding Pr\"ufer sheaf. Then $$\End(S[\infty])\simeq
  D[[T,\tau^-]].$$ In particular,
  $\Uu\simeq\mod_0\bigl(D[[T,\tau^-]]\bigr)$. \qed
\end{theorem}
We set
\begin{gather*}
  \Aut_{\tau}(D/k)=\{\sigma\in\Aut(D/k)\mid\sigma\tau=\tau\sigma\},\\
  \Inn_{\tau}(D/k)=\{\iota_u\in\Inn(D/k)\mid u\in\Fix(\tau)\},
\end{gather*}
and finally
$\Gal_{\tau}(D/k)=\Aut_{\tau}(D/k)/\Inn_{\tau}(D/k)$. Clearly
$\Inn_{\tau}(D/k)=\Inn(D/k)\cap\Aut_{\tau}(D/k)$, so that
$\Gal_{\tau}(D/k)$ can be regarded as a subgroup of
$\Gal(D/k)$. Trivially $\tau\in\Aut_{\tau}(D/k)$ holds, so that the
order of $\tau$ in $\Gal(D/k)$ is the same as the order of $\tau$ in
$\Gal_{\tau}(D/k)$. 
\begin{corollary}\label{cor:aut-group-tube}
  Let $x$ be a separable point, $\Uu=\Uu_x$ and $D=\End(S_x)$. Then
  $\Aut(\Uu/k)\simeq\Gal_{\tau}(D/k)$.
\end{corollary}
\begin{proof}
  We have $\gr(R)\simeq D[T,\tau^{-}]$ and
  $\Uu\simeq\mod_0^{\ZZ}(\gr(R))/s^{\ZZ}$, the orbit category with
  respect to the degree shift $s$; in different words, this is the
  category of finite dimensional $\gr(R)$-modules which are
  annihilated by some power of $T$. A graded automorphism of $\gr(R)$
  is uniquely determined by its action on degrees zero and one, and is
  thus of the form $$\sum a_i T^i\mapsto\sum f(a_i) N_i(b)T^i,$$ with
  $f\in\Aut(D/k)$ and $b\in D^{\times}$, satisfying $f\tau^- (a)\cdot
  b=b\cdot\tau^- f(a)$ for all $a\in D$. Here $N_i(b)$ is defined as
  $b\cdot\tau^-(b)\dots\tau^{-(i-1)}(b)$. We define the group of
  graded inner automorphisms of $\gr(R)$, denoted by $\oInn(\gr(R))$,
  generated by automorphisms of the form $\iota_u$, $r\mapsto
  u^{-1}ru$ ($u\in D^{\times})$, and by automorphisms induced by
  $T^n\mapsto N_n(b)T^n$ (with $b\in Z(D)^{\times}$). Each graded
  automorphism $\sigma=(f,b)$ of $\gr(R)$ induces an autoequivalence
  $F^{\sigma}$ on $\Uu$, and $F^{\sigma}\simeq 1_{\Uu}$ if and only if
  $\sigma$ is a graded inner automorphism. We refer
  to~\cite[Prop.~3.2.3]{kussin:2009} for a similar statement. On the
  other hand, each automorphism of $\Uu$ is uniquely determined by its
  action on the bimodule $\Ext^1(S,S)$, and thus on $R/\mathfrak{m}=D$
  and $\mathfrak{m}/\mathfrak{m}^2$, and thus induces a graded
  automorphism of $\gr(R)=D\spitz{\mathfrak{m}/\mathfrak{m}^2}$.

  Considering the skeleton of $\Uu$ and requiring that automorphisms
  are the identity on objects (e.g.\ equality $\tau S=S$), the
  automorphism $F^{\sigma}$ on $S$ commutes with $\tau$ on $S$, which
  follows from the diagram in Corollary~\ref{cor:AR-f-tau-f}.  We thus
  can assume that $f\in\Aut_{\tau}(D/k)$. Then also $b\in
  Z(D)^{\times}$. We write $\Aut_{\tau}(\gr(R))$ for the subgroup of
  the automorphisms with these properties, and
  $\oInn_{\tau}(\gr(R))=\oInn(\gr(R))\cap\Aut_{\tau}(\gr(R))$. We
  conclude
  $\Aut(\Uu/k)\simeq\Aut_{\tau}(\gr(R))/\oInn_{\tau}(\gr(R))\simeq\Gal_{\tau}(D/k)$,
  finishing the proof.
\end{proof}
\begin{corollary}\label{cor:tau-minus=sigma-x}
  Let $x$ be a point. 
  \begin{enumerate}
  \item[(1)] The functors $\tau^-$ and $\sigma_x$, restricted to the
    simple $S_x$, yield the same elements in $\Gal(\End(S_x)/k)$.
  \item[(2)] Let $x$ be separable. Then the functors $\tau^-$ and
    $\sigma_x$, restricted to $\Uu_x$, are isomorphic.
  \end{enumerate}
\end{corollary} 
The separability assumption in~(2) is essential,
cf.~Example~\ref{ex:inseparable}.
\begin{proof}
  (1) We write $S=S_x$ and $D=\End(S)$. By~\cite[0.4.2]{kussin:2009} there
  is a natural isomorphism
  $u\colon\sigma_x S\stackrel{\sim}\ra\Ext^1(S,S)\otimes_{D}S$, and
  for $f\in D$ the endomorphism $\sigma_x(f)$ corresponds to
  $\eta\otimes s\mapsto f\eta\otimes s$. Let $\eta\in\Ext^1(S,S)$.
  There is $d\in D$ with $\eta=\mu\cdot d$, where $\mu$ is the almost
  split sequence starting and ending in $S$. For $f\in Z(D)$ we have
  $f\eta\otimes s=\mu\tau^-(f)d\otimes s=\mu
  d\otimes\tau^-(f)(s)=\eta\otimes\tau^-(f)(s)$.
  We conclude $\sigma_x (f)=u^{-1}\tau^-(f)u=\tau^-(f)$ for all
  $f\in Z(D)$. Thus the restrictions of $\sigma_x$ and $\tau^-$ to
  $Z(D)$ yield the same element in $\Gal(Z(D)/k)$. By the
  Skolem-Noether theorem the restrictions of $\sigma_x$ and $\tau^-$
  to $D$ yield the same element in $\Gal(D/k)$. 

  (2) This follows from~(1) together with the preceding corollary.
\end{proof}
\begin{definition}
  Let $\XX$ be a noncommutative regular projective curve over a
  field. We call a point $x\in\XX$ a \emph{separation point}, if it is
  separable and $e_{\tau}(x)>1$ holds.
\end{definition}
\begin{corollary}\label{cor:Pic-gen-separation}
  Let $U\subseteq\XX$ be a subset such that $\Pic(\Hh)$ is generated
  by $\sigma_x$ ($x\in U$). Then $U$ contains all separation
  points. \qed
\end{corollary}

\section{Local-global principle of skewness}
\label{sec:local-global-principle}
For a point $x$ of a noncommutative regular projective curve over an
arbitrary field we write $e^{\ast\ast}(x)$ for the PI-degree of
$\End(S_x[\infty])$.
\begin{theorem}[General skewness
  principle]\label{thm:general-skewness-equation} 
  Let $\Hh$ be a noncommutative regular projective curve over an
  arbitrary field $k$. For all points $x\in\XX$ the following hold:
  \begin{enumerate}
  \item[(1)] $e(x)\cdot e^{\ast\ast}(x)=s(\Hh)$.
  \item[(2)] $e^{\ast}(x)$ divides $e^{\ast\ast}(x)$.
  \end{enumerate}
\end{theorem}
\begin{proof}
  (1) By Proposition~\ref{prop:main-ring-iso} the PI-degree of
  $\widehat{R}_x$ is $e(x)\cdot
  e^{\ast\ast}(x)$. By~\cite[Thm.13]{braun:1990} the ring $R_x$ and
  its completion $\widehat{R}_x$ have the same PI-degree. The
  PI-degree of $R_x$ coincides with the PI-degree of its quotient
  division ring $k(\Hh)$, which is $s(\Hh)$. Thus we get the equation.

  (2) By a theorem of Bergman-Small
  (see~\cite[Thm.~1.10.70]{rowen:1980}), applied to the surjective
  ring homomorphism $R_x\ra R_x/\Rad(R_x)\simeq M_{e(x)}(D_x)$, the
  PI-degree of the factor, which is $e(x)\cdot e^{\ast}(x)$, divides
  the PI-degree of $R_x$, which is $s(\Hh)$. Together with~(1) we get
  that $e^{\ast}(x)$ divides $e^{\ast\ast}(x)$.
\end{proof}
\begin{lemma}\label{lem:complete-in-x}
  Let $x$ be a point with associated skew field
  $D_x=\End(S_x)$. Denote by
  \begin{equation}
    \label{eq:skew-laurent}
    \widehat{D}_x=D_x((T,\tau^-))
  \end{equation}
  the skew Laurent power series ring over $D_x$ in the variable
  $T$. It is a skew field of dimension $e^{\ast}(x)^2\cdot
  e_{\tau}(x)^2$ over its centre. Moreover, it is $v_x$-complete,
  where the valuation $v_x$ is given by $v_x(\sum_m^{\infty}a_i
  T^i)=(1/2)^{\ell}$, with $\ell$ the infimum of indices $i$ with
  $a_i\neq 0$.
\end{lemma}
\begin{proof}
  Let $r=e_{\tau}(x)$ and $\sigma^{-r}(d)=u^{-1}du$ for some
  $u\in\Fix(\tau)^{\times}$. By~\cite[19.7]{pierce:1982}, the centre
  of $D_x((T,\tau^-))$ is given by
  \begin{equation}
    \label{eq:laurent-centre}
    \widehat{K}_x=K_x((uT^r))\quad\text{with}\quad 
    K_x=Z(D_x)\cap\Fix(\tau^-). 
  \end{equation}
  From this the assertion about the
  centre follows. Completeness is shown in~\cite[19.7]{pierce:1982}.
\end{proof}
\begin{proposition}\label{prop:relationships}
  Let $x$ be a separable point.
  \begin{enumerate}
  \item[(1)] We have $e^{\ast\ast}(x)=e^{\ast}(x)\cdot e_{\tau}(x)$.
  \item[(2)] $e_{\tau}(x)$ coincides with 
    \begin{enumerate}
    \item[(i)] the order of
    $\tau\in\Aut(\Uu_x/k)$, the group of (isomorphism classes of)
    autoequivalences on the tube $\Uu_x$;
   \item[(ii)] the order of the cyclic group $\Gal(Z(D_x)/K_x)$, generated by
    $\tau$.
    \end{enumerate}
  \item[(3)] If $\phi\in\Aut(\Hh)$, and $\phi(S_x)=S_y$, then
    $e_{\tau}(x)=e_{\tau}(y)$.
  \end{enumerate}
\end{proposition}
\begin{proof}
  (1) We have $\End(S_x[\infty])\simeq D_x[[T,\tau^-]]$. By Posner's
  theorem (see~\cite[Thm.~7]{amitsur:1967}) the PI-degree of
  $D_x[[T,\tau^-]]$ coincides with the PI-degree of its quotient
  division ring, which is $D_x((T,\tau^-))$. The assertion follows
  from the preceding lemma.

  (2) (i) follows from Corollary~\ref{cor:aut-group-tube}, (ii) from
  the preceding lemma.

  (3) It follows that $y$ is also separable, and the equality of
  $\tau$-multiplicites is obtained from (2)~(i).
\end{proof}
We will see in the next section, that (ii) just means that
$e_{\tau}(x)$ coincides with the ramification index of $x$ with
respect to the maximal order $\Aa$ associated with $\Hh$.
\begin{example}
  If $k=\RR$ and the tube $\Uu$ is either of the form
  $\mod_0\CC[[T]]$, or $\mod_0\CC[[T,\tau^-]]$ with $\tau^-$ of order
  two, then in both cases $\Aut(\Uu/k)\simeq\Gal(\CC/\RR)=C_2$,
  generated by complex conjugation. In the first case $\tau$ acts
  trivially on $\Uu$, in the second it generates $\Aut(\Uu/k)$.
\end{example}
The following local-global principle is the main result on skewness.
\begin{theorem}[Local-global principle of
  skewness]\label{thm:special-skewness-equation} 
  Let $\Hh$ be a noncommutative regular projective curve over a
  field. Then for each separable point $x\in\XX$ the
  formula $$e(x)\cdot e^{\ast}(x)\cdot e_{\tau}(x)=s(\Hh)$$ holds.
\end{theorem}
For a perfect field, we obtain Theorem~\ref{thm:main-theorem}.
\begin{proof}
  Follows from Theorem~\ref{thm:general-skewness-equation}~(1) and
  Proposition~\ref{prop:relationships}~(1).
\end{proof}
The preceding theorems give answers
to~\cite[Probl.~2.3.11+4.3.10]{kussin:2009}. It follows, by the way,
that for each separable point the multiplicity $e(x)$ does not depend
on the line bundle $L$ used in its definition, since
$e(x)=\frac{s(\Hh)}{e^{\ast}(x)e_{\tau}(x)}$.

\section{Maximal orders and ramifications}\label{sec:orders}
Following~\cite{reiten:vandenbergh:2001,reiten:vandenbergh:2002}, we
will use in this section an alternative description of noncommutative
curves in terms of hereditary and maximal orders. Here our main result
is that the $\tau$-multiplicities $e_{\tau}(x)$ coincide with the
ramification indices of the underlying maximal order $\Aa$. We will
temporarily, in Theorem~\ref{thm:structure}, also permit weighted
curves. This will allow to characterize the non-weighted situation in
terms of orders. Namely, the weights $p(x)$ correspond to the local
types of the, in general, hereditary order $\Aa$, which measure the
deviation of $\Aa_x$ from being maximal. For excellent expositions on
orders we refer to~\cite{auslander:goldman:1960},
\cite{brumer:1963,brumer:1964}, \cite{harada:1963},
\cite{reiner:2003}, \cite{schilling:1950}, and the
unpublished~\cite{artin:dejong:2004}.
\begin{numb}\label{numb:regular-vs-smooth}
  By a (commutative) \emph{curve} we mean a one-dimensional scheme
  over $k$, which we always assume to be integral, separated and of
  finite type over $k$. A curve $X$ is \emph{regular} (or
  \emph{non-singular}) if all local rings $\Oo_{X,x}$ are regular,
  equivalently, discrete valuation domains; in particular they are
  hereditary. We remark that if $k$ is a perfect field, regularity is
  equivalent to smoothness;
  c.f.~\cite[I.5.3.2]{demazure:gabriel:1980}.
\end{numb}
\begin{numb}[The centre curve]\label{numb:centre-curve}
  Let $(\Hh,L)$ be a noncommutative regular projective curve over the
  field $k$ with point set $\XX$ and function field $D=k(\Hh)$. Let
  $K=Z(k(\Hh))$ be the centre of $D$. There is a unique (commutative)
  regular complete curve $X=C_K$ with function field $k(X)=K$ and
  whose points are in bijective correspondence with the discrete
  valuations of $K/k$; we refer
  to~\cite[Prop.~(7.4.18)]{grothendieck:1961b},
  also~\cite[I.5.3.7]{demazure:gabriel:1980}. By Chow's lemma (we
  refer to~\cite[Ex.~II.4.10]{hartshorne:1977}
  and~\cite[5.6]{grothendieck:1961b}) there is a (irreducible)
  \emph{projective} curve $X'$ and a surjective, birational morphism
  $\pi\colon X'\ra X$ over $k$ (in particular: $X'=\Proj(S')$ where
  the commutative graded ring $S'$ is generated in degrees $0$ and
  $1$). By~\cite[Cor.~(4.4.9)]{grothendieck:1961ca} we have that $\pi$
  is even an isomorphism. In particular, $X$ itself is projective over
  $k$. We call $X$ the \emph{centre curve} of $\Hh$ (or $\XX$). If
  $\Oo=\Oo_X$ is the structure sheaf of $X$, we denote by
  $(\Oo_x,\mathfrak{m}_x)$ the local rings ($x\in X$) and by
  $k(x)=\Oo_x/\mathfrak{m}_x$ the residue class fields. 
\end{numb}
\begin{example}
  Let $k=\RR$ be the field of real numbers and $R=\CC[X;Y,\sigma]$ the
  twisted polynomial algebra, graded by total degree, where $X$ is
  central and $Yz=\sigma(z)Y$ for each $z\in\CC$, with
  $\sigma(z)=\bar{z}$ the complex conjugation. Then
  $\Hh=\mod^{\ZZ}(R)/\mod_0^{\ZZ}(R)$ is a noncommutative regular
  projective curve (we refer to~\cite{kussin:2009} for more
  details). The function field is $\CC(T,\sigma)$, its centre given by
  $\RR(T^2)$. The centre of $R$ is $S=\RR[X,Y^2]$, and $\Proj(S)$ is
  the centre curve of $\Hh$. It is isomorphic to the projective
  spectrum of $S'=\RR[X,Y]$, graded by total degree, having function
  field $\RR(T)$, which as $\RR$-algebra is isomorphic to
  $\RR(T^2)$. Thus the centre curve of $\Hh$ is isomorphic to the
  projective line $\Pone(\RR)$.
\end{example}
\begin{numb}[The categorical centre]
  We also have the centre of $\Hh$ in the categorical sense, namely
  $Z(\Hh)=\End(1_{\Hh})$, the ring of natural endotransformations of
  the identity functor $1_{\Hh}$.  
\end{numb}
\begin{lemma}
  The categorical centre $Z(\Hh)$ is a field, of finite dimension over
  $k$. For each line bundle $L'$, the assignment
  $\alpha\mapsto\alpha_{L'}$ yields a $k$-monomorphism from $Z(\Hh)$
  into $\End(L')$.
\end{lemma}
Therefore we can usually assume without loss of generality that $k$ is
the centre of $\Hh$.
\begin{proof}
  We proceed as in~\cite[(S~19)]{lenzing:delapena:1999}. If $\alpha$
  is a non-zero element in the centre, then $\alpha_{L'}$ is non-zero
  for each line bundle $L'$: if otherwise $\alpha_{L'}=0$, then it
  follows that also $\alpha_{L'(nx)}=0$ for all $x\in\XX$ and
  $n\in\ZZ$. Using ampleness (cf.\ Lemma~\ref{lem:Serre-construction}
  and Remark~\ref{rem:ample-pair-weighted}) we the get easily
  $\alpha=0$, contradiction. Since $\End(L')$ is a skew field,
  $\alpha_{L'}$ is an isomorphism. Using line bundle filtrations and
  the fact that each simple object is the cokernel of a monomorphism
  between line bundles, we obtain that $\alpha_F$ is an isomorphism
  for each object $F\in\Hh$. Thus $\alpha$ is invertible. Finally, for
  all $\alpha,\,\beta\in Z(\Hh)$ we clearly have
  $\alpha_{L'}\beta_{L'}=\beta_{L'}\alpha_{L'}$, and hence $Z(\Hh)$ is
  commutative. 
\end{proof}
\begin{proposition}\label{prop:coord-ring-finite-module-centre}
  Let $(\Hh,L)$ be a noncommutative regular projective curve.
  \begin{enumerate}
  \item[(1)] For each $x\in\XX$ the graded ring $\Pi(L,\sigma_x)$ is
    finitely generated as module over its centre.
  \item[(2)] If $s(\Hh)=1$, then $\Pi(L,\sigma_x)$ is commutative.
  \end{enumerate}
\end{proposition}
\begin{proof}
  Like in~\cite[Prop.~4.3.3]{kussin:2009} we have a graded inclusion
  $\Pi(L,\sigma_x)\subseteq k(\Hh)[T]$, where $T$ is a central
  variable. From this, (2) follows immediately, and (1) follows
  with~\cite{lenagan:1994}
  and~\cite[Thm.~0.1(ii)]{artin:stafford:1995}.
\end{proof}
\begin{corollary}
  Let $(\Hh,L)$ be a noncommutative regular projective curve over $k$
  with $s(\Hh)=1$. Then there is a (commutative) regular projective
  curve $X$ over $k$ such that $\Hh\simeq\coh(X)$, and the points of
  $\XX$ are in bijective correspondence with the closed points of $X$.
\end{corollary}
\begin{proof}
  Let $S$ be the commutative graded ring $\Pi(L,\sigma_x)$ for some
  $x\in\XX$ and $X=\Proj(S)$. 
\end{proof}
\begin{corollary}\label{cor:algebr-closed}
  Let $k$ be an algebraically closed field. Then $\Hh$ is a
  noncommutative regular projective curve over $k$ if and only if
  $\Hh$ is equivalent to the category $\coh(X)$ of coherent sheaves
  over a (commutative) regular (=smooth) projective curve $X$.
\end{corollary}
\begin{proof}
  By Tsen's theorem~\cite{tsen:1933} we have $s(\Hh)=1$.
\end{proof}
\begin{numb}[The centre curve in the weighted case]
  We assume that $\Hh$ satisfies (NC~1) to (NC~5). By
  Remark~\ref{rem:ample-pair-weighted} also (NC~7) holds. Thus the
  centre of the function field $k(\Hh)$ is of the form $k(X)$, for a
  unique regular projective curve $X$, which we also call the
  \emph{centre curve} of $\Hh$ in this weighted case. Similarly,
  $R=\Pi(L,\sigma)$, with $\sigma$ a suitable product of
  Picard-shifts, is module-finite over its centre, by the same
  arguments given in
  Proposition~\ref{prop:coord-ring-finite-module-centre}. Part~(2) of
  the next theorem below will show that also (NC~6) is satisfied.
\end{numb}
\begin{numb}[Orders over the centre curve]
  Let $X$ be the centre curve with function field $K=k(X)$. Let $A$ be
  a finite dimensional central simple $K$-algebra. As
  in~\cite{artin:dejong:2004} we call a torsionfree, coherent
  $\Oo_X$-algebra $\Aa$ an $\Oo_X$-\emph{order} in $A$, if the generic
  fibre of $\Aa$ is isomorphic to $A$, or equivalently, if
  $\Aa\otimes_{\Oo_X}\!K\simeq A$. An order $\Aa$ is called
  \emph{maximal} if it is not contained properly in another
  order. Then all stalks $\Aa_x=\Aa\otimes\Oo_x$ are maximal
  $\Oo_x$-orders in $A$. An order $\Aa$ is called \emph{hereditary},
  if all stalks $\Aa_x$ are hereditary $\Oo_x$-orders in $A$. Each
  maximal order is hereditary. The $\Oo_X$-order $\Aa$ is called an
  \emph{Azumaya algebra} of degree $n$, if $\Aa$ is locally-free of
  rank $n^2$, and if for each $x\in X$ the geometric fibre
  $\Aa(x)=\Aa_x\otimes_{\Oo_x}\!k(x)=\Aa_x/\Rad(\Aa_x)$ is a full matrix
  algebra with centre $k(x)$. Equivalently
  (by~\cite[Prop.~1.9.2]{artin:dejong:2004}): For each $x$ we have
  $[\Aa(x):k(x)]=n^2$. Azumaya algebras over $\Oo_X$ are maximal
  orders (by~\cite[Prop.~1.8.2]{artin:dejong:2004}).
\end{numb}
We now have the following fundamental description of noncommutative
regular projective curves, essentially due to Reiten-van den
Bergh~\cite[Prop.~III.2.3]{reiten:vandenbergh:2002}.
\begin{theorem}\label{thm:structure}
  Let $k$ be a field. 
  \begin{enumerate}
  \item[(1)] For a $k$-category $\Hh$ the following two conditions are
    equivalent:
  \begin{enumerate}
  \item $\Hh$ is a weighted noncommutative regular projective curve
    over $k$.
  \item There is a (commutative) regular projective curve $X$ over
    $k$, a (finite dimensional) central simple $k(X)$-algebra $A$ and
    a torsionfree coherent sheaf $\Aa$ of hereditary
    $\Oo=\Oo_X$-orders in $A$ such that $\Hh\simeq\coh(\Aa)$, the
    category of coherent $\Aa$-modules.
  \end{enumerate}
\item[(2)] If the equivalent conditions in~(1) hold, then $X$ is the
  centre curve. Accordingly, the points of $\XX$ correspond
  bijectively to the closed points of $X$, and for each $x\in\XX$ its
  weight $p(x)$ is the local type (in the sense
  of~\cite[p.~369]{reiner:2003}) of the hereditary $\Oo$-order $\Aa$
  at $x$. Accordingly, $p(x)>1$ if and only if $\Aa_x$ is not maximal,
  and there is only a finite number of such points $x$.
\item[(3)] In~(1) we have that $\Hh$ is non-weighted if and only if
  $\Aa$ is a maximal $\Oo$-order in $A$.
  \end{enumerate}
\end{theorem}
\begin{proof}
  (1) This is shown like
  in~\cite[Prop.~III.2.3]{reiten:vandenbergh:2002}. For the fact that
  the centre of a hereditary order is a Dedekind domain, we refer
  to~\cite[Thm.~2.6]{harada:1963}. By~\cite{vandenbergh:vangeel:1984}
  the category $\coh(\Aa)$ has Serre duality. We recall the
  construction of $A$ and $\Aa$ if $\Hh$ is given. Let $X$ be the
  underlying centre curve. Let $R$ be a positively $\ZZ$-graded
  coordinate algebra of $\Hh$, module-finite over its centre $S$. Let
  $x_1,\dots,x_t$ be a set of homogeneous generators of $S$ over the
  field $S_0$. Let $n$ be the least common multiple of their
  degrees. Then $$T=
  \begin{pmatrix}
    R & R(1) & \dots & R(n-1)\\
    R(-1) & R & \dots & R(n-2)\\
   \vdots & \vdots & \ddots  &  \vdots\\
    R(-n+1) & R(-n+2) & \dots & R
  \end{pmatrix}$$ is graded Morita-equivalent to $R$ and strongly
  $\ZZ$-graded; thus $\mod^{\ZZ}(T)\simeq\mod(T_0)$. Let $\Tt$ be the
  corresponding sheaf of graded rings.  We set $A=\matring_n(k(\Hh))$,
  which is of finite dimension over its centre $k(X)$, and
  $\Aa=\Tt_0\subseteq A$, equipped canonically with the structure of
  an $\Oo_X$-module.

  (2) The assertion is clear from the structure of hereditary
  orders~\cite{brumer:1963,brumer:1964}, we refer also
  to~\cite[Ch.~9]{reiner:2003}, and the Auslander-Goldman
  criterion~\cite[Thm.~2.3]{auslander:goldman:1960} for
  maximality. (This in particular shows that (NC~6) follows from
  (NC~1) to (NC~5).)

  (3) This follows from (2), since by~\cite[(40.8)]{reiner:2003} the
  order $\Aa$ is maximal if and only if all $\Aa_x$ are maximal.
\end{proof}
We switch back to the non-weighted case. The next theorem
extends~\cite[Prop.~(7.4.18)]{grothendieck:1961b} to this
noncommutative setting, and it gives a positive answer
to~\cite[Probl.~4.3.9]{kussin:2009}, even in this much more general
context. It is an easy consequence of well-known results in the theory
of maximal orders. It was also shown recently
in~\cite[Thm.~6.7]{burban:drozd:gavran:2015}.
\begin{theorem}\label{thm:function-field-determines}
  Two noncommutative regular projective curves $\Hh$ and $\Hh'$ over a
  field $k$ are isomorphic (that is, they are equivalent as
  $k$-categories) if and only if their function fields $k(\Hh)$ and
  $k(\Hh')$ are isomorphic.
\end{theorem}
\begin{proof}
  If $\Hh\simeq\Hh'$, then $\Hh_0\simeq\Hh'_0$ and
  $\Hh/\Hh_0\simeq\Hh'/\Hh'_0$, and consequently $k(\Hh)$ and
  $k(\Hh')$ are isomorphic. Assume conversely, that the function
  fields $k(\Hh)$ and $k(\Hh')$ are isomorphic and have the common
  centre $K=k(X)$. By parts~(1) and~(3) of the preceding theorem,
  there are maximal orders $\Aa$ and $\Aa'$ in Morita-equivalent
  central simple $K$-algebras $A$ and $A'$, respectively, such that
  $\Hh\simeq\coh(\Aa)$ and $\Hh'\simeq\coh(\Aa')$. Since $X$ is a
  normal curve, by~\cite[Prop.~1.9.1~(ii)]{artin:dejong:2004} (for an
  affine version we refer to~\cite[Cor.~(21.7)]{reiner:2003}; for a
  similar result on hereditary orders over a smooth curve we refer
  to~\cite[Thm.~7.6]{chan:ingalls:2004}) it follows, that $\Aa$ and
  $\Aa'$ are Morita-equivalent, that is (by definition),
  $\Qcoh(\Aa)\simeq\Qcoh(\Aa')$. Then clearly
  $\coh(\Aa)\simeq\coh(\Aa')$, and thus $\Hh\simeq\Hh'$ follows.
\end{proof}
Since maximal $\Oo_X$-orders over a regular projective curve $X$ in a
central simple $k(X)$-algebra always exist,
by~\cite[Prop.~1.8.2]{artin:dejong:2004}, we have even more:
\begin{corollary}
  The assignments $$\Hh\mapsto k(\Hh)\quad\text{and}\quad
  A\mapsto\coh(\Aa),$$ where $\Aa$ is a maximal order in $A$ (whose
  centre is of the form $k(X)$), induce mutually inverse bijections
  between the sets of
  \begin{itemize}
  \item noncommutative regular projective curves over $k$, up to
    equivalence of categories; and
  \item algebraic function skew fields of one variable over $k$, up to
    isomorphism. \qed
  \end{itemize}
\end{corollary}
Let $\Hh=\coh(\Aa)$ be a noncommutative regular projective curve with
a maximal order $\Aa$ in $A=\matring_n(k(\Hh))$ as above. Let $K=k(X)$
be the centre, as above. Let $x\in\XX$ be separable. We write
$D_x=\End(S_x)$ and denote by $\widehat{E}_x=\End(S_x[\infty])$ the
endomorphism ring of the corresponding Pr\"ufer sheaf, which is a
complete local domain, the maximal ideal generated by $\pi_x$. By
Proposition~\ref{prop:main-ring-iso} we have
$\widehat{R}_x\simeq\matring_{e(x)}(\widehat{E}_x)$. If
$\widehat{K}_x$ denotes the quotient field of the
$\mathfrak{m}_x$-adic completion $\widehat{\Oo}_x$ of $\Oo_x$, then
\begin{equation}
  \label{eq:A-completion}
  A\otimes_K\widehat{K}_x\simeq
\matring_n\bigl(k(\Hh)\otimes_K\widehat{K}_x\bigr)\simeq\matring_{n\cdot
  e(x)}(\widehat{D}_x),
\end{equation}
with $\widehat{D}_x$ a skew field (unique up to isomorphism) with
centre $\widehat{K}_x$; compare
Proposition~\ref{prop:completions-equal-laurent} below. Analogously to
the global situation we make the following local definition.
\begin{definition}
  We call the number
\begin{equation}
  \label{eq:def-s-x}
  s(x)=[\widehat{D}_x:\widehat{K}_x]^{1/2}
\end{equation}
the \emph{local skewness} at $x$.
\end{definition}
By~\eqref{eq:A-completion} we get the following relationship
between global and local skewness
\begin{equation}
  \label{eq:e-times-s}
  s(\Hh)=e(x)\cdot s(x),  
\end{equation}
and with the skewness principle we have
\begin{equation}
  \label{eq:local-skewness-identification}
    s(x)=e^{\ast}(x)\cdot e_{\tau}(x).
\end{equation}
The following results make the situation quite explicit.
\begin{lemma}\label{lem:local-centre}
  \begin{enumerate}
  \item[(1)] $\Oo_x$ is the centre of $R_x$. 
  \item[(2)] Let $\vSs_x$ be the multiplicative set
    $\Oo_x\setminus\{0\}$. The central localization
    $\vSs_x^{-1}R_x$ is equal to the function field $k(\Hh)$.
  \end{enumerate}
\end{lemma}
\begin{proof}
  (1) We have $\mod(\Aa_x)\simeq\Hh_x\simeq\mod(R_x)$. Hence $\Aa_x$
  and $R_x$ have same centres. So it is sufficient to show that
  $\Oo_x$ is the centre of $\Aa_x=\Aa\otimes_{\Oo}\Oo_x$. Let
  $U\subseteq X$ be an affine open subset with $x\in U$. Then $\Oo(U)$
  is a Dedekind domain with quotient field $K$, and $\Aa(U)$ is a
  maximal $\Oo(U)$-order in $A$, whose centre $\Aa(U)\cap K$ is
  $\Oo(U)$, since $\Oo(U)$ is integrally closed. Now, localization at
  $x$ is compatible with the centres~\cite[Prop.~1.7.4]{rowen:1980}
  and yields the result.

  (2) $k(\Hh)$ is the quotient division ring of
  $R_x$. By~\cite[Thm.~1.7.9]{rowen:1980} the ring $\vSs_x^{-1}R_x$ is
  a skew field with centre $K$. From
  $R_x\subseteq\vSs_x^{-1}R_x\subseteq k(\Hh)$ the result follows.
\end{proof}
\begin{proposition}\label{prop:completions-equal-laurent}
  Let $x$ be separable. The skew field $\widehat{D}_x$ and its centre
  $\widehat{K}_x$ from~\eqref{eq:A-completion} agree with the (skew)
  Laurent power series rings in~\eqref{eq:skew-laurent}
  and~\eqref{eq:laurent-centre}, respectively. Moreover,
  $k(x)=K_x:=Z(D_x)\cap\Fix(\tau)$.
\end{proposition}
\begin{proof}
  Write $r=e_{\tau}(x)$. Using the preceding lemma, we apply
  $\vSs_x^{-1}$ (that is, central localization) to the isomorphism
  $R_x\otimes_{\Oo_x}\widehat{\Oo}_x\simeq\matring_{e(x)}(\widehat{E}_x)$.
  By $\Hom$-tensor properties of the localization~\cite[II.\S
  2.7]{bourbaki:1989}, we obtain the isomorphism $Q(R_x)\otimes_K
  \vSs_x^{-1}\widehat{\Oo}_x\simeq\matring_{e(x)}(\vSs_x^{-1}\widehat{E}_x)$,
  where $Q(-)$ stands for quotient division ring. Moreover, since
  $\widehat{\Oo}_x$ and $\widehat{E}_x$ are (skew) power series rings
  in one variable by Theorem~\ref{thm:pruefer-end-skew-power}, clearly
  $\vSs_x^{-1}\widehat{\Oo}_x=Q(\widehat{\Oo}_x)$ and
  $\vSs_x^{-1}\widehat{E}_x=Q(\widehat{E}_x)$ hold, since in each case
  the uniformizer becomes invertible, so that we get the corresponding
  (skew) Laurent series rings. We conclude $\widehat{D}_x\simeq
  Q(\widehat{E}_x)\simeq D_x((T,\tau^-))$. Moreover, for the centre we
  deduce $k(x)((T))\simeq\widehat{K}_x\simeq K_x((uT^r))$, from which
  also $K_x=k(x)$ follows.
\end{proof}
We can now derive well-known identities for well-studied local
invariants, and also their relationship with the
$\tau$-multiplicity. As usual, define the
\begin{equation}
  \label{eq:def-inert-deg}
  \text{\emph{inertial degree}}\quad
 f_{\inert}(x)=[\widehat{D}_x/\Rad(\widehat{D}_x):k(x)]=[D_x:k(x)]
\end{equation}
and the 
\begin{equation}
  \label{eq:def-ramif-index}
  \text{\emph{ramification index}}\quad
  e_{\ramif}(x)=[\,\Gamma_{\widehat{D}_x}:\Gamma_{\widehat{K}_x}] 
\end{equation}
(the index of the discrete value group $\Gamma_{\widehat{K}_x}$ in
$\Gamma_{\widehat{D}_x}$) of the skew field part of the completion of
$A$ in $x$. If $e_{\ramif}(x)>1$, then $x$ is called a
\emph{ramification point} of $A$. By
Proposition~\ref{prop:completions-equal-laurent} it is easy to see
that the ramification index coincides with the number $e$ such that
$\mathfrak{m}_x \widehat{E}_x=\Rad(\widehat{E}_x)^e$, and also with
the number
\begin{equation}
  \label{eq:e-prime}
    e'(x)=[Z(D_x):k(x)].
\end{equation}
\begin{corollary}\label{cor:e-tau=e-ramif}
  Let $x$ be a separable point. Then $$e_{\tau}(x)=e_{\ramif}(x).$$
\end{corollary}
The assumption is essential, cf.\ Example~\ref{ex:inseparable}, and
Theorem~\ref{thm:order-sigma-ramif}.
\begin{proof}
  Follows from Propositions~\ref{prop:relationships}
  and~\ref{prop:completions-equal-laurent}.
\end{proof}
\begin{corollary}\label{cor:separ-ramif}
  Let $\Hh=\coh(\Aa)$ be a noncommutative regular projective curve over
  the perfect field $k$. The separation points of $\Hh$ are just the
  ramification points of $\Aa$, and there are only finitely many of
  them.
\end{corollary}
\begin{proof}
  It is well-known in the theory of maximal orders that the number of
  ramification points is finite, see
  e.g.~\cite[p.~372]{reiner:2003}. In
  Proposition~\ref{prop:only-finitely-many-separation-points} below a
  standard argument for this will be given.
\end{proof}
We also conclude
\begin{corollary}
  If $x$ is a separable point, then
  $$e_{\ramif}(x)\cdot f_{\inert}(x)=s(x)^2.$$
\end{corollary}
\begin{proof}
  Use
  $f_{\inert}(x)=[D_x:k(x)]=[D_x:Z(D_x)]\cdot
  [Z(D_x):k(x)]=e^{\ast}(x)^2\cdot e_{\ramif}(x)$.
\end{proof}
We call $\XX$ (or $\Hh$) \emph{unramified} (resp.\
$\tau$-\emph{unramified}) if $e_{\ramif}(x)=1$ (resp.\
$e_{\tau}(x)=1$) for all $x\in\XX$; if $k$ is perfect both notions
agree. We call it \emph{multiplicity free} if $e(x)=1$ for all
$x\in\XX$.
\begin{corollary}\label{cor:unramif=tau-identity}
Let $k$ be a perfect field.
  \begin{enumerate}
  \item[(1)] We have \begin{equation*}
  \tau=\prod_{x\in\XX}{\sigma_x}^{e_{\tau}(x)-1}\quad\quad\text{on}\
  \Hh_0.
\end{equation*} 
\item[(2)] $\XX$ is unramified if and only if
  $\tau_{|\Hh_0}\simeq 1_{\Hh_0}$.
  \end{enumerate}
  \end{corollary}
\begin{proof}
  (1) follows from Corollary~\ref{cor:tau-minus=sigma-x}. (2) is then
  clear by the definition of $e_{\tau}(x)$.
\end{proof}
The following general result expresses $e_{\ramif}(x)$ as the local
order of a certain functor, namely of $\sigma_x$ on $\Uu_x$.
\begin{theorem}\label{thm:order-sigma-ramif}
  Let $\Hh$ be a noncommutative regular projective curve over a field
  $k$. Let $\Uu=\Uu_x$ be a tube and $V=(V,\pi,\sigma)$ be the
  associated complete discrete valuation domain
  (\cite{amdal:ringdal:1968ab}) with $\Uu=\mod_0(V)$ from
  Proposition~\ref{prop:complete-local-rings-general-homogeneous},
  where $V\pi=\pi V$ is the maximal ideal and $\sigma\colon V\ra V$
  the automorphism given by $\pi r=\sigma(r)\pi$. Then the
  Picard-shift functor $\sigma_x$, restricted to $\Uu$, is induced by
  $\sigma$. Its order in $\Aut(\Uu/k)$ equals the order of $\sigma$ in
  $\Aut(V/k)$ modulo inner automorphisms, and equals the ramification
  index $e_{\ramif}(x)$.
\end{theorem}
\begin{proof}
  Given $x$, we form the orbit algebra $\Pi(L,\sigma_x)$, which has a
  central prime element $\pi_x$ of degree
  one. By~\cite[Thm.~3.1.2]{kussin:2009} multiplication with $\pi_x$
  yields the natural transformation
  $1_{\Hh}\stackrel{x}\ra\sigma_x$. Extending $\sigma_x$ to the direct
  limit closure of $\Uu$ (so working in the category $\Qcoh(\Aa)$ of
  quasicoherent $\Aa$-modules), we see that the natural sequence
  in~\cite[0.4.2(5)]{kussin:2009} for the injective object $S[\infty]$
  becomes $0\ra S\ra S[\infty]\stackrel{\pi}\lra S[\infty]\ra 0$. We
  conclude that $\sigma_x$ on $\Uu$ is induced by the automorphism
  $\sigma$. The statement about the orders follows from the
  observation that $\pi^n$ is central up to a unit if and only if
  $\sigma^n$ is inner.
\end{proof}

\section{Dualizing sheaf and the Picard-shift group}
\label{sec:dualizing-sheaf}
\begin{numb}[Structure sheaf]\label{numb:structure-sheaf} 
  Let $\Hh=\coh(\Aa)$ with $\Aa$ a
  maximal order in $k(\Hh)$, and with centre curve $X$ and
  $\Oo=\Oo_X$. We will now specify our structure sheaf $L\in\Hh$, namely
  \begin{equation}
    \label{eq:structure-sheaf}
    L_{\Aa}=\Aa_{\Aa}.
  \end{equation}
  Hence $\Aa\simeq\ShEnd_{\Aa}(L)$ holds. \emph{From now on we will
    always assume this.} We remark that $L$ is locally-free of finite
  rank, both over $\Aa$ and over $\Oo$.
\end{numb}
\begin{numb}
  We denote by $\Pic(\Aa)$ the group of all isomorphism classes of
  invertible $\Aa$-$\Aa$-bimodules in $\coh(\Aa)$, with multiplication
  given by the tensor product over $\Aa$. The neutral element is the
  class of $\Aa$. Each invertible bimodule ${}_{\Aa}\Mm_{\Aa}$ gives
  rise to the (exact) autoequivalence $t_{\Mm}=-\!\otimes_{\Aa}\!\Mm$
  of $\Hh$, and $t_{\Mm}\simeq 1_{\Hh}$ if and only if $\Mm\simeq\Aa$
  as bimodules.
\end{numb}
\begin{numb}[Divisors]
  Let $\delta=\sum_{x\in\XX}\delta_x\cdot x\in\ZZ^{(\XX)}$ be a (Weil)
  divisor. In~\eqref{eq:S-universal-L} we defined the line bundle
  $L(x)=\sigma_x(L)$, which extends canonically to
  $L(\delta)=\prod_{x\in\XX}{\sigma_x}^{\delta_x}(L)$. This definition
  of the line bundle $L(\delta)$ is clearly dual to the definition
  in~\cite[p.~34]{vandenbergh:vangeel:1985}. Moreover,
  $\Aa(x)_{\Aa}=L(x)$ defines an $\Aa$-$\Aa$-bimodule, with left
  action induced by left multiplication in $\Aa$.
\end{numb}
For $x\in\XX$ let $t_x\colon\coh(\Aa)\ra\coh(\Aa)$ denote the functor
$-\!\otimes_{\Aa}\!\Aa(x)$, that is, $t_x=t_{\Aa(x)}$.
\begin{lemma}\label{lem:isomorphic-point-functors}
  For all $x$ the functors $\sigma_x$ and $t_x\colon\Hh\ra\Hh$ are
  isomorphic.
\end{lemma}
\begin{proof}
  We proceed like in the proof of the theorem of Eilenberg-Watts,
  \cite[Thm.\ II.(2.3)]{bass:1968}. We have, as right $\Aa$-modules,
  $t_x(L)\simeq
  L\!\otimes_{\Aa}\!\Aa(x)\simeq\Aa(x)=L(x)=\sigma_x(L)$.
  Both functors are autoequivalences, and exact ($\Aa$ is a
  locally-free $\Oo$-module). $\sigma_x(L)$ is a right $\Aa$-module;
  it can be made into a bimodule in the canonical way, and this
  bimodule agrees with $\Aa(x)$. For each $E\in\coh(\Aa)$ we have a
  natural morphism
  $f_E\colon
  E\simeq\ShHom_{\Aa}(\Aa,E)\ra\ShHom_{\Aa}(\Aa(x),\sigma_x(E))$
  induced by $\sigma_x$; for this we remark, that this can be indeed
  defined locally, since $\sigma_x$ fixes all tubes, and is therefore
  compatible with localizations in the sense
  of~\ref{numb:categorical-localization}; then one can imitate the
  (more general) proof
  of~\cite[Thm.~19.5.4]{kashiwara:schapira:2006}. Under the natural
  isomorphisms
  $\Hom_{\Aa}(E,\ShHom_{\Aa}(\Aa(x),\sigma_x(E)))\simeq
  \Hom_{\Aa}(E\otimes_{\Aa}\Aa(x),\sigma_x(E))$
  it corresponds to a natural morphism
  $g_E\colon t_x(E)\ra\sigma_x(E)$, thus we have a natural
  transformation $g\colon t_x\ra\sigma_x$.  This is an isomorphism on
  $L_{\Aa}$, which is locally a progenerator for $\coh(\Aa)$. Since
  both functors also preserve finite coproducts, it follows that $g_E$
  is an isomorphism for every object $E\in\coh(\Aa)$.
\end{proof}
We recall from Theorem~\ref{thm:structure}~(2) that there is a
bijection between the closed points of the centre curve $X$ and the
points of $\XX$. By abuse of notation we use the same symbol $x$ for
$x\in X$ and the corresponding point $x\in\XX$.
\begin{theorem}\label{thm:Picard-sequence}
  Let $\Hh$ be a noncommutative regular projective curve over the
  field $k$. Let $X$ be the centre curve. Then there is an exact
  sequence
  \begin{equation}
    \label{eq:Picard-sequence}
    1\ra\Pic(\coh(X))\stackrel{\iota}\lra\Pic(\Hh)\stackrel{\phi}
    \lra\prod_{x}\ZZ/e_{\ramif}(x)\ZZ\ra 1 
 \end{equation}
 of abelian groups. Here, $\phi(\sigma)=\sigma_{|\Hh_0}$ and $\iota$
 sends a Picard-shift $s_x$ of $\coh(X)$, for a point $x\in X$, to
 ${\sigma_x}^{e_{\ramif}(x)}$, for the corresponding point
 $x\in\XX$. Moreover, $\Pic(\coh(X))\simeq\Pic(X)$, the Picard group of
 isomorphism classes of line bundles in $\coh(X)$ with the tensor
 product.
\end{theorem}
\begin{proof}
  Since $\sigma_x$ on $\Uu_x$ has order $e_{\ramif}(x)$, it is clear
  that $\phi$ induces a surjective homomorphism as indicated, and its
  kernel is given by $\spitz{{\sigma_x}^{e_{\ramif}(x)}\mid
    x\in\XX}$. We have to show that
  $s_x\mapsto{\sigma_x}^{e_{\ramif}(x)}$ yields an isomorphism between
  $\Pic(\coh(X))$ and this kernel. Surjectivity is clear. For
  well-definedness and injectivity we have to show that a word in the
  $s_x$ is trivial if and only if the corresponding word in the
  ${\sigma_x}^{e_{\ramif}(x)}$ is trivial. Let $x\in X$ and
  $\Aa_x=\Aa\otimes_{\Oo}\Oo_x$. It follows from
  Corollary~\ref{cor:nc-dedekind} (its proof) and
  Lemma~\ref{lem:local-centre} that
  $\mathfrak{m}_x\Aa_x=\Rad(\Aa_x)^e$ for some natural number
  $e$. Forming completions we see that $e=e_{\ramif}(x)$. From this we
  deduce that
  $$\Oo(x)\!\otimes_{\Oo}\!\Aa\simeq\Aa(e_{\ramif}(x)\cdot x)$$ as
  $\Aa$-$\Aa$-bimodules. For $\delta=\sum_{x\in X}\delta_x\cdot x$
  define $\overline{\delta}=\sum_{x\in\XX}e_{\ramif}(x)\delta_x\cdot x$.
  Thus $\Oo(\delta)\otimes_{\Oo}\!\Aa\simeq\Aa(\overline{\delta})$.
  We hence have that $\Oo(\delta)\simeq\Oo$ implies
  $\Aa(\overline{\delta})\simeq\Aa$. Moreover,
  $s_x\simeq -\!\otimes_{\Oo}\!\Oo(x)$ and
  $\sigma_x\simeq -\!\otimes_{\Aa}\!\Aa(x)$, by the lemma. So
  $s_x\mapsto{\sigma_x}^{e_{\ramif}(x)}$ gives a well-defined
  homomorphism. For an $\Aa$-$\Aa$-bimodule $M$ we define, as
  in~\cite[(37.28)]{reiner:2003}, the $\Oo$-$\Oo$-subbimodule
  $M^{\Aa}$ of $M$, locally, by consisting of all $x\in M$ such that
  $\alpha x=x\alpha$ holds for all $\alpha\in\Aa$. By
  Lemma~\ref{lem:local-centre}~(1), we have $\Aa^{\Aa}=\Oo$. Then
  $(\Oo(\delta)\otimes_{\Oo}\Aa)^{\Aa}=\Oo(\delta)$. If now
  $\Aa(\overline{\delta})\simeq\Aa$ as bimodules, then we obtain
  $\Oo(\delta)=(\Oo(\delta)\otimes_{\Oo}\Aa)^{\Aa}\simeq\Aa^{\Aa}=\Oo$. Thus
  our map is also injective.

  The last statement about $\Pic(X)$ is, since $X$ is commutative, easy
  to show.
\end{proof}
\begin{numb}[The $\Oo$-dual]
  For $E\in\Hh$ let $\odual{E}=\ShHom_{\Oo}(E,\Oo)$ denote the
  $\Oo$-dual of $E$.
\end{numb}
\subsection*{The different}
By~\cite{vandenbergh:vangeel:1984} and well-known Hom-tensor relations
\begin{equation}
  \label{eq:dualizing-sheaf-isos}
  \bom_{\Aa}:=\ShHom_{\Oo}(\Aa,\bom_X)\simeq\bom_X\!\otimes_{\Oo}\!\odual{\Aa}
  \simeq\varphi^{\ast}\bom_X\!\otimes_{\Aa}\!\Aa(\Delta)
\end{equation}
is the dualizing sheaf for $\coh(\Aa)$, with
$\varphi^{\ast}\colon\coh(X)\ra\coh(\Aa)$ the functor
$-\!\otimes_{\Oo}\!\Aa$ and $\Delta$ a divisor, which is called the
\emph{different}.

The ``difference'' between $\tau_{|\Hh_0}$ and
$\tau=\bom_{\Aa}\!\otimes\!-$, globally on $\Hh$, is given by the
Auslander-Reiten translation on the centre $\coh(X)$. Moreover, it
follows from Theorem~\ref{thm:Picard-sequence}, that on $\Hh_0$ the
functor $\varphi^{\ast}\bom_X\!\otimes_{\Aa}\!-$ is the identity. We
have $\bom_X=\Oo_X(\gamma)$, where $\gamma$ is the canonical divisor
on $X$ (we refer to~\cite[VIII.~Prop.~1.13]{altman:kleiman:1970}). We
then get
$\bom_{\Aa}=\Aa\bigl(\overline{\gamma}+\Delta\bigr)$. Formulated in
terms of Picard-shifts we obtain the following.
\begin{theorem}
  Let $\Hh$ be a noncommutative regular projective curve over a field
  $k$. Then
  \begin{enumerate}
  \item[(1)] $\tau\in\Pic(\Hh)$. 
  \item[(2)] Each $e_{\tau}(x)$ divides $e_{\ramif}(x)$.
  \end{enumerate}
\end{theorem}
It is shown in Example~\ref{ex:inseparable} that in general
$e_{\tau}(x)\neq e_{\ramif}(x)$. 
\begin{proof}
  (1) Follows from the preceding discussions. (2) follows from (1) and
  Theorem~\ref{thm:order-sigma-ramif}. More precisely we have that the
  order of $\tau$ in $\Gal(\End(S_x)/k)$ divides the order of $\tau$
  in $\Aut(\Uu_x/k)$, which divides the order of $\sigma_x$ in
  $\Aut(\Uu_x/k)$.
\end{proof}
In order to have ``good'' ramifications, we assume now that either $k$
is perfect, or that the characteristic of $k$ does not divide the
skewness $s(\Hh)$ (cf.\ \cite[1.3.9]{artin:dejong:2004}, \cite[Ch.~5,
Sec.~6]{schilling:1950}). The divisor
$\Delta=\sum_x (e_{\ramif}(x)-1)\cdot x$ is called the
\emph{different} of $\Hh$. It is, locally in $x$, induced by the exact
sequence
\begin{equation}
  \label{eq:exact-sequence-different}
  0\ra L\xrightarrow{{\pi_x}^{e_{\ramif}(x)-1}} L((e_{\ramif}(x)-1)x)\lra
  S_x[e_{\ramif}(x)-1]^{e(x)}\ra 0
\end{equation}
from Lemma~\ref{lem:cokernel-pi-n}. It follows then that the cokernel
$C$ of the injective (reduced) trace map $\Aa\ra\odual{\Aa}$ locally
in $x$ has $k$-dimension
\begin{equation}
  \label{eq:dim-cokernel-trace}
  \dim_k
  C_x=e_{\ramif}(x)(e_{\ramif}(x)-1)e(x)^2e^{\ast}(x)^2[k(x):k]\stackrel{(\ast)}=s(\Hh)^2
 \bigl(1-\frac{1}{e_{\tau}(x)}\bigr)[k(x):k],
\end{equation}
with equation $(\ast)$ holding in the separable case. From this we
immediately get:
\begin{proposition}\label{prop:only-finitely-many-separation-points}
  There are only finitely many separation points. \qed
\end{proposition}
\begin{theorem}\label{thm:general-tau-picard}
  Let $\Hh$ be a noncommutative regular projective curve over a field
  $k$ which is perfect or of characteristic prime to $s(\Hh)$. Let
  $\gamma=\sum_x \gamma_x\cdot x$ be the canonical divisor of the
  centre curve $X$. For
  $\overline{\gamma}=\sum_{x\in\XX}\gamma_xe_{\ramif}(x)\cdot x$ we
  write $\sigma^{|\overline{\gamma}|}$ for the corresponding
  Picard-shift. Then
  \begin{equation}
    \label{eq:general-tau-picard}
    \tau\ =\
    \sigma^{|\overline{\gamma}|}\cdot\prod_{x}{\sigma_x}^{e_{\ramif}(x)-1}\
    =\ \prod_{x}{\sigma_x}^{e_{\ramif}(x)(\gamma_x+1)-1}.
  \end{equation}
\end{theorem}

\section{The genus and the Euler characteristic}\label{sec:genus}
We recall that if $k$ is algebraically closed, e.g.\ $k=\CC$, then
there is the well-known relation $\chi(X)=2(1-g(X))$ between the Euler
characteristic $\chi(X)$ and the genus $g(X)=\dim_k\Ext^1(\Oo,\Oo)$ of
the regular projective curve (or compact Riemann surface) $X$. Let
$(\Hh,L)$ be a noncommutative regular projective curve over the field
$k$. We set $\kappa=\dim_k\End(L)$. The \emph{Euler form} is defined
by
$$\LF{E}{F}=\dim_k\Hom(E,F)-\dim_k \Ext^1(E,F)$$ for objects
$E,\,F\in\Hh$. We call $$\genus=\dim_{\End(L)}\Ext^1(L,L)$$ the
\emph{genus} of $\Hh$ and
\begin{equation}
  \label{eq:char-genus}
  \chi'(\Hh)=\frac{1}{s(\Hh)^2}\cdot\LF{L}{L}=\frac{\kappa}{s(\Hh)^2}\cdot
  (1-\genus)
\end{equation}
the \emph{normalized Euler characteristic} of $\Hh$ (over $k$). Note
that this definition depends on the base-field $k$, which can be
by-passed by assuming that $k$ is the centre of $\Hh$. If $\Aa$ is
Azumaya of degree $s$ over $\Oo$, for instance, a maximal order in
$\matring_s(\Oo)$, then $\chi(\Aa)=s^2\cdot\chi(\Oo)$;
compare~\cite[4.1.1+4.1.5]{artin:dejong:2004}, with
$\chi(\Aa):=\LF{L}{L}$. Thus one can regard the normalized Euler
characteristic $\chi'$ to be invariant under Morita-equivalence. Note
that in case $k$ is algebraically closed, we have
$\chi(\Hh)=2\chi'(\Hh)$. In case $k=\RR$ and $s(\Hh)=1$ the definition
of $\chi'$ agrees with the topological definition of the Euler
characteristic for the underlying manifold; we refer to further
discussion in~\ref{numb:euler-char-witt}. With this we have
$$\genus=0\ \Leftrightarrow\ \chi'(\Hh)>0\quad\text{and}\quad
\genus=1\ \Leftrightarrow\ \chi'(\Hh)=0.$$ For $F\in\Hh$ we define
\begin{equation}
  \label{eq:def-deg}
  \deg(F)=\frac{1}{\kappa\varepsilon}\cdot\LF{L}{F}-
\frac{1}{\kappa\varepsilon}\cdot\LF{L}{L}\cdot\rk(F), 
\end{equation}
where $\varepsilon\geq 1$ is the natural number such that the
resulting linear form $\deg\colon\Knull(\Hh)\ra\ZZ$ becomes
surjective. We obtain
\begin{proposition}[Riemann-Roch formula]
  Let $(\Hh,L)$ be a noncommutative regular projective curve over a
  field. For all $E,\,F\in\Hh$ we have
  $$\frac{1}{\kappa}\cdot\LF{E}{F}=(1-\genus)\cdot\rk(E)\cdot\rk(F)+\varepsilon\cdot
\begin{vmatrix}
  \rk(E) & \rk(F)\\
  \deg(E) & \deg(F)
\end{vmatrix}.$$
\end{proposition}
\begin{proof}
  First we remark that $\deg$ is additive on short exact sequences,
  and that $\LF{-}{-}$ and the right hand side of the formula induce
  bilinear maps $\Knull(\Hh)\times\Knull(\Hh)\ra\ZZ$. If $E=L$, then
  the formula is just the definition of the degree. If
  $E,\,F\in\Hh_0$, then both sides are zero, by the structure of
  $\Hh_0$ and the $\tau$-invariance of each object in $\Hh_0$. This
  yields, if $L'$ is a line bundle and $S$ is simple, then
  $\LF{L'}{S}=\LF{L}{S}$ (considering $[L]-[L']$). Then, if one of $E$
  or $F$ belongs to $\Hh_0$, it is easy to see that the formula holds,
  by using line bundle filtrations. If $\sigma\in\Aut(\Hh)$ is
  point-fixing, then both sides remain equal, if replacing the pair
  $(E,F)$ by $(\sigma E,\sigma F)$ (for all $E,\,F\in\Hh$). If $L'$ is
  a line bundle, by Lemma~\ref{lem:morphisms-line-bundles-large-n} we
  have an exact sequence $0\ra L(-nx)\ra L'\ra C\ra 0$ ($x$ any point,
  $n\gg 0$, $C\in\Uu_x$). From this the formula holds for $E=L'$ and
  $F\in\Hh$. Using line bundle filtrations, the formula holds for
  every $E\in\Hh_+$ and $F\in\Hh$, and then generally.
\end{proof}
In particular, if $\bom=\bom_{\Aa}=\tau L$ denotes the dualizing sheaf
in $\Hh$, then
\begin{equation}
  \label{eq:deg-dualizing}
  \deg(\bom)=-\frac{2}{\kappa\varepsilon}\cdot\LF{L}{L}=\frac{2}{\varepsilon}\cdot
  (\genus-1). 
\end{equation}
Moreover, if $\delta=\sum_x \delta_x\cdot x\in\ZZ^{(\XX)}$ is a
divisor, and if we define $\ell(\delta)=[\Hom(L,L(\delta)):\End(L)]$,
then (by specializing to $E=L$, $F=L(\delta)$)
$$\ell(\delta)=1-g(\Hh)+\deg\delta+\ell(\omega-\delta),$$
with $L(\omega)=\bom$, which is the Riemann-Roch in more classical
form. Here,
$\deg\delta:=\varepsilon\deg(L(\delta))=\frac{1}{\kappa}\sum_x
\delta_x \cdot [\Aa_x/\Rad(\Aa_x):k]$.
\begin{lemma}\label{lem:finite-field-deg-S}
  Let $\Hh$ be a noncommutative regular projective curve over the
  perfect field $k$. For each $x\in\XX$ we have
  \begin{equation}
  \label{eq:deg-of-simples}
    \deg(S_x)=\frac{s(\Hh)}{\kappa\varepsilon}\cdot e^{\ast}(x)\cdot [k(x):k].
  \end{equation}
\end{lemma}
\begin{proof}
  Follows easily from the skewness principle.
\end{proof}
For a Picard-shift $\sigma\in\Pic(\Hh)$ we call $\deg(\sigma(L))$ the
\emph{degree} of $\sigma$, and we denote by $\Pic_0(\Hh)$ the subgroup
of degree zero Picard-shifts; similarly for $\coh(X)$ and $\Oo$, where
we use the symbols $\deg_X$, $\kappa_X$ and $\varepsilon_X$. From the
lemma we easily get the following.
\begin{proposition}
  For the injective homomorphism
  $\iota\colon\Pic(\coh(X))\ra\Pic(\Hh)$ in~\eqref{eq:Picard-sequence}
  we have
  $$\deg(\iota(s))=s(\Hh)^2\cdot\frac{\kappa_X\varepsilon_X}{\kappa\varepsilon}\cdot\deg_X
  (s).$$ In particular, $\Pic_0(X)$ can be regarded as a subgroup of
  $\Pic_0(\Hh)$. \qed
\end{proposition}
Let $\Hh=\coh(\Aa)$ be a noncommutative regular projective curve with
centre curve $X$ and with $\Aa$ a maximal $\Oo_X$-order in a central
simple $k(X)$-algebra. Following~\cite{artin:dejong:2004} we call
$\gge{e}=(e_1,\dots,e_n)$ the \emph{ramification vector} of $\Aa$ if
$e_1,\dots,e_n$ are all ramification indices $>1$, and moreover, for
each ramification point $x$ of $\Aa$ its ramification index
$e_{\ramif}(x)$ appears precisely $[k(x):k]$ times in $\gge{e}$. If
for all ramification points $x_1,\dots,x_t$ (pairwise different) the
numbers $f_i=[k(x_i):k]$ are given, then we will also write more
precisely $\gge{e}=({e_1}^{f_1},\dots,{e_t}^{f_t})$ and call it the
ramification \emph{sequence}.
\begin{proposition}[{Artin-de Jong~\cite[Lemma~4.1.5]{artin:dejong:2004}}]
  Let $\Hh=\coh(\Aa)$ be a noncommutative regular projective curve over
  a perfect field $k$ with centre curve $X$. Let $\Aa$ be a maximal
  $\Oo_X$-order in $k(\Hh)$ and
  $\gge{e}=({e_1}^{f_1},\dots,{e_n}^{f_n})$ its ramification
  sequence. Then we have for the normalized Euler characteristic
  \begin{equation}
    \label{eq:general-euler-char-formula}
    \chi'(\Hh)=\chi'(X)-\frac{1}{2}\sum_{i=1}^n
    f_i \cdot\Bigl(1-\frac{1}{e_i}\Bigr). 
  \end{equation}
\end{proposition}
\begin{proof}
  Write $s=s(\Hh)$. We consider the exact sequence
  $0\ra\Aa\ra\odual{\Aa}\ra C\ra 0$ in $\coh(\Aa)$. It is also exact
  in $\coh(\Oo)$. With the degree $\deg_X$ over $X$ we obtain, using
  $\deg_X(\Aa)=-\deg_X(\odual{\Aa})$, that
  $\deg_X(\Aa)=-\frac{1}{2}\deg_X(C)=-\frac{s^2}{2\kappa_X\varepsilon_X}\sum_x
  (1-1/e_{\tau}(x))[k(x):k]$,
  by~\eqref{eq:dim-cokernel-trace}. We have $\rk_X(\Aa)=s^2$. By the
  Riemann-Roch, over $X$, we get
  $\LF{\Oo}{\Aa}_X=s^2\chi'(X)-\frac{s^2}{2}\sum_x
  (1-1/e_{\tau}(x))[k(x):k]$.
  Finally, by flatness $\LF{\Aa}{\Aa}=\LF{\Oo}{\Aa}_X$, and
  division by $s^2$ gives the claim.
\end{proof}
\begin{proposition}
  Let $\Hh$ be a noncommutative regular projective curve. If
  $\genus>1$, then all Auslander-Reiten components of
  $\Hh_+=\vect(\XX)$ have as underlying graph $\ZZ A_{\infty}$, and
  the category $\Hh$ is wild.
\end{proposition}
\begin{proof}
  This follows from~\cite[Prop.~4.7]{lenzing:reiten:2006}.
\end{proof}
\subsection*{Noncommutative elliptic curves}
\begin{definition}
  We call a (non-weighted) noncommutative regular projective curve of
  genus one a \emph{(noncommutative) elliptic curve}.
\end{definition}
For elliptic curves there is the following analogue of Atiyah's
classification of vector bundles over an elliptic curve over an
algebraically closed field~\cite{atiyah:1957}:
\begin{theorem}\label{thm:elliptic}
  Let $\Hh=\coh(\XX)$ be a noncommutative elliptic curve. Then the
  following holds:
  \begin{enumerate}
  \item[(1)] Each indecomposable object $E$ in $\Hh$ is semistable of
    slope $\mu(E)=\frac{\deg(E)}{\rk(E)}$ in
    $\widehat{\QQ}:=\QQ\cup\{\infty\}$ and satisfies
    $\Ext^1(E,E)\neq 0$.
  \item[(2)] For each $\alpha$ the subcategory $\bt_{\alpha}$ of
    semistable objects of slope $\alpha$ is non-trivial and forms a
    tubular family, again parametrized by a noncommutative elliptic
    curve $\XX_{\alpha}$, so that $\Hh'=\coh(\XX_{\alpha})$ is
    derived-equivalent to $\Hh$.
  \end{enumerate}
\end{theorem}
We remark that one major difference to Atiyah's result is, that in
general $\Hh'$ may be \emph{not} isomorphic to $\Hh$. We will give an
example later.
\begin{proof}
  (1) Semistability follows like
  in~\cite[Prop.~5.5]{geigle:lenzing:1987}. The condition
  $\Ext^1(E,E)\neq 0$ is a direct consequence of the Riemann-Roch
  formula.

  (2) The proof of~\cite[Prop.~8.1.6]{kussin:2009} works also in this
  situation, with a slight modification: Let
  $\alpha\in\widehat{\QQ}$. It is sufficient to show that
  $\bt_{\alpha}\neq 0$.  In the bounded derived category
  $\bDerived{\Hh}$ we form the \emph{interval category}
  $\Hh'=\Hh\spitz{\alpha}$, the additive closure of
  $\bigcup_{\gamma>\alpha}\bt_{\gamma}[-1]\cup
  \bigcup_{\beta\leq\alpha}\bt_{\beta}$.
  Let $\Hh'_0$ be its subcategory of objects of finite
  length. By~\cite[4.9+5.2]{lenzing:reiten:2006}, $\Hh'$ is again
  noetherian with Serre duality.  Since $\Hh'\neq 0$ is noetherian,
  there are simple objects. Let $S\in\Hh'$ be simple. If
  $\bt_{\alpha}=0$, then $S$ has a slope $\beta<\alpha$.  The
  Riemann-Roch formula in this elliptic case
  implies $$\mu(E)<\mu(F)\quad\Rightarrow\quad\Hom(E,F)\neq 0$$ for
  all indecomposable objects $E$ and $F$.  We conclude, that another
  simple object $S'$ has the same slope $\beta$, and then
  $\Hh'_0=\bt_{\beta}$.  Then there is no indecomposable object in
  $\Hh'$ of slope $\gamma>\beta$.  But we find such an object by
  Picard-shifting a line bundle $L'$ in $\Hh'$ sufficiently far to the
  left, and then applying suspension $[1]$. This contradiction shows,
  that $\bt_{\alpha}\neq 0$.
\end{proof}
\begin{numb}[Fourier-Mukai partners]
  Let $\Hh$ and $\Hh'$ be noncommutative regular projective curves
  over $k$. We call $\Hh'$ a \emph{Fourier-Mukai partner} of $\Hh$, if
  there is an exact equivalence $\bDerived{\Hh}\ra\bDerived{\Hh'}$. We
  are mainly interested in this notion when at the same time the
  categories $\Hh$ and $\Hh'$ are not equivalent.

  We recall that, since $\Hh$ is hereditary, $\bDerived{\Hh}$ is the
  \emph{repetitive category} of $\Hh$, that is,
  $\bDerived{\Hh}=\bigvee_{n\in\ZZ}\Hh[n]$, and for $E,\,F\in\Hh$ we
  have $\Hom_{\bDerived{\Hh}}(E[m],F[n])=\Ext^{n-m}_{\Hh}(E,F)$. In
  the elliptic case,
  $\Hh=\coh(\XX)=\bigvee_{\beta\in\widehat{\QQ}}\bt_{\beta}$, and
  $\XX$ parametrizes the tubular family $\bt_{\infty}$. If
  $\alpha\in\widehat{\QQ}$, we have seen that
  $\Hh\spitz{\alpha}\simeq\coh(\XX_\alpha)$, where $\XX_\alpha$ is the
  elliptic curve parametrizing the tubular family $\bt_{\alpha}$. From
  the explicit description in Theorem~\ref{thm:elliptic} we readily
  obtain $\bDerived{\Hh}=\bDerived{\Hh\spitz{\alpha}}$.

  From this it follows easily, that if $\Hh$ is elliptic and
  $\bDerived{\Hh}\ra\bDerived{\Hh'}$ an exact equivalence, then $\Hh'$
  is equivalent to $\Hh\spitz{\alpha}$ for some $\alpha$.
\end{numb}
\begin{examples}\label{ex:dynkin-vectors}
  (1) Each (commutative) regular projective curve of genus one is
  elliptic in our sense; we do not require the existence of a
  $k$-rational base-point. In particular, the Klein
  bottle~\ref{ex:comm-real-elliptic} is a real elliptic curve without
  $\RR$-rational (= real) points (it is a Klein surface without
  boundary). \medskip
  
  (2) Let $X$ be a regular projective curve over its perfect centre
  $k$, let $K=k(X)$ its function field and $D$ a finite dimensional
  central division $K$-algebra. Let $\Aa$ be a maximal $\Oo_X$-order
  in $D$ and $\gge{e}$ its ramification vector. Let $\Hh=\coh(\Aa)$.

  (a) We have $\chi'(\Hh)>0$ if and only if $X$ has genus zero and
  $\gge{e}=(e)$, $(e_1,e_2)$, $(2,2,e)$, $(2,3,3)$, $(2,3,4)$ or
  $(2,3,5)$.

  (b) $\Hh$ is elliptic, if either $X$ has genus one and $\Aa$ is
  Azumaya, or $X$ has genus zero and $\Aa$ has ramification vector
  $\gge{e}=(2,3,6)$, $(2,4,4)$, $(3,3,3)$ or $(2,2,2,2)$. This follows
  directly from~\eqref{eq:general-euler-char-formula}. Moreover, if
  $g(X)=0$, then we conclude from~\eqref{eq:general-tau-picard} that
  the order of $\tau$ in $\Aut(\Hh)$ is given by the least common
  multiple (= maximum) $n$ of the $\tau$-multiplicities, and
  $n=2,\,3,\,4$, or $6$. If, on the other hand, $g(X)=1$, then
  $\bom_X=\Oo(\omega)$ has degree zero, and from the Riemann-Roch
  formula we get $\ell(\omega)=g(X)=1$, and thus
  $\bom_X=\Oo$. Moreover, $\Aa$ is unramified (that is, Azumaya), thus
  we obtain $\bom_{\Aa}=\bom_X\otimes_{\Oo}\Aa=\Aa$. Thus $\tau=1$.
\end{examples}
\begin{theorem}\label{thm:elliptic-tau-finite-order}
  Let $\Hh$ be elliptic over a perfect field. Then the order of $\tau$
  is $1,\,2,\,3,\,4$, or $6$, given by the maximum of the
  $\tau$-multiplicities. The ramification vector is a derived
  invariant of an elliptic curve.
\end{theorem}
\begin{proof}
  The order of $\tau$ is a derived invariant, and the maximum of the
  ramification indices determines the ramification vector, which is
  one of $()$, $(2,3,6)$, $(2,4,4)$, $(3,3,3)$ or $(2,2,2,2)$.
\end{proof}

\section{Genus zero: ghosts and ramifications}
\label{sec:genus-zero}
Noncommutative projective curves of genus zero are important in the
representation theory of finite dimensional algebras, in particular
(but not only) for the tame algebras. For details we refer
to~\cite{kussin:2009} and~\cite{kussin:2008}. We recall that genus
zero means that $\Ext^1(L,L)=0$, or equivalently, the existence of a
tilting bundle $T\in\Hh$. The noncommutative regular genus zero curves
$\Hh$ over a field $k$ correspond to the so-called tame bimodules
${}_F M_G$, where $F$ and $G$ are finite dimensional division algebras
over $k$, and $M$ is an $F$-$G$-bimodule on which $k$ is acting
centrally and $\dim {}_FM\cdot\dim M_G=4$. The corresponding bimodule
algebra $\Lambda=
\begin{pmatrix}
  G & 0\\
  M & F
\end{pmatrix}$
is a finite dimensional tame hereditary $k$-algebra, whose category
$\mod(\Lambda)$ of finite dimensional right modules is derived
equivalent to $\Hh$. The function field $k(\Hh)$ is isomorphic to the
endomorphism ring of the unique generic $\Lambda$-module. We refer
to~\cite{dlab:ringel:1976,crawley-boevey:1991,baer:geigle:lenzing:1987,ringel:1976} 
as references for the representation theory of tame bimodules and tame
hereditary algebras. We define the \emph{numerical type} $\varepsilon$
of $\XX$ by $\varepsilon=1$, if $(\dim_FM,\dim M_G)=(2,2)$, and
$\varepsilon=2$, if this dimension pair is given by $(1,4)$ or
$(4,1)$. ($\varepsilon$ coincides with the previously defined
normalizing factor of the degree.) We further set
\begin{equation}
  \label{eq:index-f-x}
  f(x)=\frac{1}{\varepsilon}\cdot
  [\Ext^1(S_x,L):\End(L)]=\deg(S_x). 
\end{equation}
A point $x\in\XX$ is called \emph{rational}, if $f(x)=1$. In the genus
zero case, rational points always exist,
\cite[Prop.~4.1]{lenzing:delapena:1999}. It is shown
in~\cite{kussin:2009} that there is a so-called efficient automorphism
$\sigma\in\Aut(\Hh)$. The orbit algebra $R=\Pi(L,\sigma)$ serves as a
homogeneous coordinate algebra for $\Hh$. We recall the definition
from~\cite[Def.~1.1.3]{kussin:2009}: an automorphism
$\sigma\colon\Hh\ra\Hh$ is called \emph{efficient} if it is
point-fixing (that is, $\sigma(\Uu_x)=\Uu_x$ for all $x$), if the
degree of $\sigma L$ is positive, and if there is no point-fixing
automorphism such that the degree of the image of $L$ is positive and
smaller. As a consequence $\Pi(L,\sigma)$ is shown to be graded
factorial (a graded version of the notion of a unique factorization
ring in~\cite{chatters:jordan:1986}), the points $x\in\XX$ in
correspondence with the homogeneous height one prime ideals, each
generated by a normal element (called \emph{prime}) $\pi_x$.
\begin{example}[Commutative case]
  Let $\Hh=\coh(X)$ be a commutative regular projective curve of genus
  zero with centre $k$. If the characteristic is different from $2$,
  then either $k(\Hh)=k(T)$, the rational function field over $k$ in
  one variable (case $\varepsilon=1$), or
  $k(\Hh)=k(U,V)/(-aU^2-bV^2+ab)$, with $a,\,b\in k^{\times}$ such
  that $-aY^2-bZ^2+abX^2$ is anisotropic over $k$ (case
  $\varepsilon=2$). (If $k$ is additionally perfect, this also holds
  in characteristic $2$, where in case $\varepsilon=2$ the quadratic
  equation is slightly different; we refer
  to~\cite[Thm.~6.2]{kussin:2008}.) For this classical result we refer
  to~\cite[Thm.~6.1]{kussin:2008}; there also the characteristic $2$
  case is treated. In all cases (also in characteristic $2$), we have
  $\End(L)=k$ (for all line bundles) and $\chi'(X)=1$.
\end{example}
\begin{lemma}\label{lem:degree-zero-picard}
  Let $\Hh$ have centre curve $X$ with $g(X)=0$. Let
  $\sigma\in\Pic(\Hh)$ be of degree zero and with
  $\sigma_{|\Hh_0}=1_{\Hh_0}$. Then $\sigma=1_{\Hh}$.
\end{lemma}
\begin{proof}
  By the assumptions it follows with Theorem~\ref{thm:Picard-sequence}
  that $\sigma$ lies in $\Pic(\coh(X))$, and is there of degree
  zero. Since $g(X)=0$, this means $\sigma(\Oo)\simeq\Oo$. Since $X$
  is commutative, we get $\sigma\simeq 1_{\coh(X)}$
  from~\cite[Cor.~3.3]{kussin:2008}, and then also
  $\sigma\simeq 1_{\Hh}$ with Theorem~\ref{thm:Picard-sequence}.
\end{proof}
Theorem~\ref{thm:general-tau-picard} yields
\begin{proposition}\label{prop:genus-zero-centre-tau-picard}
  Let $k$ be a perfect field. Let $\Hh$ be a noncommutative regular
  projective curve over $k$. We assume that its centre curve $X$ is of
  genus zero, of numerical type $\varepsilon$. Then
  $$\tau\ =\ {\sigma_{x_0}}^{-2/\varepsilon}\cdot\prod_{x}{\sigma_x}^{e_{\tau}(x)-1}$$
  for any point $x_0\in\XX$ which is rational in $X$ and not
  ramification. \qed
\end{proposition}
The following result follows from
Example~\ref{ex:dynkin-vectors}~(2)(a). We state it here explicitly
because of its influence to representation theory of finite
dimensional algebras.
\begin{theorem}
  Let $\Hh$ be a noncommutative regular projective curve of genus zero
  over a perfect field. There are at most three separation points. \qed
\end{theorem}
\subsection*{Ghosts and ramifications}
We recall from~\cite{kussin:2009}, that the ghost group $\Gg=\Gg(\Hh)$
is defined as the subgroup of the automorphism (class) group
$\Aut(\Hh)$ defined by those elements, called ghosts, which fix (up to
isomorphism) the structure sheaf $L$ and all simple sheaves $S_x$
($x\in\XX$); it follows then, in this genus zero case, that all
objects in $\Hh$ are fixed. In~\cite[Cor.~3.3]{kussin:2008} we have
shown that commutativity (= multiplicity freeness) implies that the
ghost group is trivial. As we shall see now, this last property
follows already when $\XX$ is unramified, but there are further cases
with one ramification point. We remark that
Proposition~\ref{prop:genus-zero-centre-tau-picard} is in particular
applicable in case $\Hh$ is of genus zero. We will not use it in the
following (since it will not simplify our proofs). Instead, we will
recover its validity in the examples treated below.
\begin{proposition}\label{prop:unramif=trivial-ghost}
  Let $\Hh$ be a noncommutative regular projective curve of genus zero
  over the perfect field $k$, which is (without loss of generality)
  the centre of $\Hh$. We assume additionally that the characteristic
  of $k$ is different from $2$. Let $X$ be the centre curve. The
  following are equivalent:
  \begin{enumerate}
  \item[(i)] On $\Hh_0$ the functor $\tau$ is isomorphic to the
    identity functor.
  \item[(ii)] There is a skew field $D$ with centre $k$ and
    $[D:k]=s(\Hh)^2$ such that $k(\Hh)\simeq D\otimes_k k(X)$, a
    constant extension.
  \item[(iii)] There is a skew field $D$ with centre $k$ and
    $[D:k]=s(\Hh)^2$ such that $k(\Hh)$ is either isomorphic to
    $D(T)$, with $T$ central, or to a skew field of the form
    $D(U,V)/(-aU^2-bV^2+ab)$ with central variables $U,\,V$ and
    non-zero elements $a,\,b\in k$ such that the quadratic form
    $-aY^2-bZ^2+abX^2$ is anisotropic over $k$.
  \item[(iv)] There is a rational point $x$ with $e(x)=1$,
    $e_{\tau}(x)=1$ and $\Pic(\Hh)=\spitz{\sigma_x}$.
  \end{enumerate}
  When this holds, then the ghost group is trivial, $\Gg(\Hh)=1$, and
  we have $\tau={\sigma_x}^{-2/\varepsilon}$.
\end{proposition}
\begin{proof}
  Condition~(i) is equivalent to say that $\Hh$ is unramified, by
  Corollary~\ref{cor:unramif=tau-identity}. By a result of van den
  Bergh and van Geel~\cite[Prop.~2.2]{vandenbergh:vangeel:1985} this
  is equivalent to~(ii); we also refer
  to~\cite{witt:1936}. By~\cite[Thm.~6.2]{kussin:2008} condition~(iii)
  is just a more explicit reformulation of~(ii).\medskip

  (iii)$\Rightarrow$(iv) We can rewrite~(iii) (or~(ii)) in terms of
  coordinate algebras. There is a rational and multiplicity free point
  $x$ such that $R=\Pi(L,\sigma_x)$ is either isomorphic to $D[X,Y]$
  or to $D[X,Y,Z]/(-aY^2-bZ^2+abX^2)$ with central variables $X$, $Y$
  and $Z$ of degree one, and $x$ corresponds to the variable $X$. The
  centre $Z(R)$ is given by $k[X,Y]$ and
  $k[X,Y,Z]/(-aY^2-bZ^2+abX^2)$, respectively. From this we infer
  $e_{\tau}(x)=e_{\ramif}(x)=1$. If $\gamma$ is a graded automorphism
  of $R$ fixing all prime elements, then it is easy to see that its
  restriction to the centre has the form $\gamma(s)=a^{|s|}\cdot s$
  for some $a\in k^{\times}$ (independent from $s$), where $|s|$
  denotes the degree of the homogeneous element $s$. From this we
  easily deduce $\Gg=1$ with~\cite[Thm.~3.1]{kussin:2008}. We also
  have $\Pic(\Hh)=\spitz{\sigma_x}$ and
  $\tau^-={\sigma_x}^{2/\varepsilon}$.\medskip

  (iv)$\Rightarrow$(i) The conditions $e(x)=1=f(x)$ imply that
  $\sigma_x$ is an efficient automorphism, and we get
  $\tau^-={\sigma_x}^{2/\varepsilon}$. Since
  $\Pic(\Hh)=\spitz{\sigma_x}$, all points different from $x$ are
  $\tau$-unramified. Since also $e_{\tau}(x)=1$ by assumption, (i)
  follows.
\end{proof}
Condition~(iv) allows us to relate the preceding proposition to the
next theorem and, in particular, to its corollary.
\begin{theorem}
  Let $(\Hh,L)$ be a noncommutative regular projective curve of genus
  zero over the perfect field $k$. Assume that there is an efficient
  tubular shift $\sigma_x$, and let $R=\Pi(L,\sigma_x)$. For a normal
  homogeneous element $r\neq 0$ in $R$ define the graded algebra automorphism
  $\gamma_r$ by $rs=\gamma_r(s)r$ for all $s\in R$, and denote by
  $\gamma_r^{\ast}$ the induced element in $\Aut(\Hh)$.
  \begin{enumerate}
  \item[(1)] Each ghost $\gamma$ is induced by a graded algebra
    automorphism of the form $\gamma_r$ where $r$ is homogeneous
    normal: $\gamma\simeq\gamma_r^{\ast}$. 
  \item[(2)] For each point $y\neq x$ the automorphism
    $$\gamma_{\pi_y}^{\ast}={\sigma_x}^{-d(y)}\circ\sigma_y$$ is a ghost
    of order $e_{\tau}(y)$. Moreover, ${\pi_y}^{e_{\tau}(y)}$ is
    central in $R$ up to multiplication with a unit in $R_0$.
  \item[(3)] We have that the ghost group
    $\Gg(\Hh)=\spitz{\gamma_{\pi_y}^{\ast}\mid y\neq
      x,\,e_{\tau}(y)>1}$
    is finite abelian and coincides with the subgroup
    $\Pic_0(\Hh)\subseteq\Pic(\Hh)$ of Picard-shifts of degree
    zero. Moreover, $\Pic(\Hh)\simeq\spitz{\sigma_x}\times\Gg(\Hh)$ is
    finitely generated abelian of rank one, and we have
    $\tau\in\Pic(\Hh)$.
  \end{enumerate}
\end{theorem}
\begin{proof}
  (1) By~\cite[Thm.~3.1]{kussin:2008} we have $\gamma=\beta^{\ast}$
  for a graded, prime fixing algebra automorphism $\beta$ of $R$;
  there are units $u_y\in R_0$ with $\beta(\pi_y)=u_y\pi_y$ for all
  $y\in\XX$. Since $\pi_x$ is central, and since each central element
  is a product of prime elements, we get $\beta(s)=u^{|s|}\cdot s$ for
  all $s\in S=Z(R)$, with $u=u_x$ central. Thus, the automorphism
  $\beta'=\beta\circ{\varphi_u}^{-1}$ of $R$ is the identity on
  $S$. Since $\beta^{\ast}\simeq\beta'^{\ast}$, we can assume that
  $\beta$ gives the identity on $S$. The induced automorphism
  $\overline{\beta}$ of $k(\Hh)$ is the identity on its centre $k(X)$,
  thus $\overline{\beta}$ is inner by the Skolem-Noether
  theorem. Since $k(\Hh)$ is obtained from $R$ by central
  localization, there is $r\in R$ homogeneous with $\beta(s)=rsr^{-1}$
  for all $s\in R$. The relation $\beta(s)r=rs$ shows that $r$ is
  normal, and $\beta=\gamma_r$.

  (2), (3) Moreover, by~\cite[Thm.~3.2.8]{kussin:2009} we have
  $\gamma_{\pi_y}^{\ast}={\sigma_x}^{-d(y)}\circ\sigma_y\in\Pic(\Hh)$. We
  also obtain that $\gamma_{\pi_y}^{\ast}\in\Gg(\Hh)$ acts on $\Uu_y$
  like $\sigma_y$, and thus the order of $\gamma_{\pi_y}^{\ast}$ is
  $\geq e_{\tau}(x)$. Let $n$ be the least common multiple of
  $e_{\tau}(y)$ and $e_{\tau}(x)/h$, where $h$ is the greatest common
  divisor of $e_{\tau}(x)$ and $d(y)$. Then $n$ is the smallest
  natural number such that
  $\beta^{\ast}=\bigl({\sigma_x}^{-d(y)}\sigma_y\bigr)^n$ is the
  identity functor on $\Hh_0$. We have $\beta(s)=asa^{-1}$ for a
  normal element $a$, say of degree $m$. We can assume that $a$ does
  not have a central divisor of degree $\geq 1$. Assume that there is
  a point $p$ such that the prime $\pi_p$ is a divisor of $a$. The
  element ${a}{\pi_x}^{-m}$ lies in the radical of the localization
  $R_p$. It is then easy to see that $\beta^{\ast}$ cannot be
  isomorphic to the identity functor on the factor module
  $R_p/({\pi_p}{\pi_x}^{-d(p)})\simeq {S_p}^{e(p)}$, giving a
  contradiction. It follows, that $a$ does not have any prime divisor,
  and so $a\in R_0$ is a unit. Thus $\beta$ is an inner automorphism
  of $R$, and thus $\beta^{\ast}\simeq 1_{\Hh}$.

  Since each normal element is a product of prime elements, we obtain
  $\Gg(\Hh)=\spitz{\gamma_{\pi_y}^{\ast}\mid y\neq
    x}\subseteq\Pic(\Hh)$.  Since $\sigma_x$ is efficient,
  $\tau\circ{\sigma_x}^{2/\varepsilon\ell}$ is a ghost and thus an
  element of $\Pic(\Hh)$. If $e_{\tau}(x)=1$, then we obtain
  $\Gg(\Hh)=\spitz{\gamma_{\pi_y}^{\ast}\mid e_{\tau}(y)>1}$ is
  finite. Assume now $e_{\tau}(x)>1$, and let $y\neq x$ be another
  point. Calculations using~\eqref{eq:deg-of-simples} show
  $$e_{\tau}(y)\cdot d(y)=\frac{[k(y):k]}{[k(x):k]}\cdot
  e_{\tau}(x).$$
  Since $\pi_x$ is of degree one in the centre, it is clear that the
  fraction is an integer. We obtain $\gamma_{\pi_y}^{\ast}$ has order
  $e_{\tau}(y)$. In particular, $\gamma_{\pi_y}^{\ast}\simeq 1_{\Hh}$
  unless $y$ is a ramification point. Thus, again
  $\Gg(\Hh)=\spitz{\gamma_{\pi_y}^{\ast}\mid y\neq x,\,e_{\tau}(y)>1}$
  is finite.
\end{proof}
\begin{corollary}\label{cor:trivial-ghost-group}
  Let $\Hh$ be a noncommutative regular projective curve of genus zero
  over the perfect field $k$.
  \begin{enumerate}
  \item[(1)] Assume that $\Gg(\Hh)=1$. Then there is a point $x\in\XX$
    with $\Pic(\Hh)=\spitz{\sigma_x}$. Moreover, either
    \begin{enumerate}
    \item[(a)] $\XX$ is unramified and $e(x)=1=f(x)$ holds; or
    \item[(b)] $x$ is the unique ramification point. 
    \end{enumerate}
  \item[(2)] Assume that there is an efficient tubular shift
    $\sigma_x$ with $\Pic(\Hh)=\spitz{\sigma_x}$. Then the ghost group
    is trivial, $\Gg(\Hh)=1$.
  \end{enumerate}
\end{corollary}
\begin{proof}
  (1) Let $\sigma$ be an efficient automorphism. Since $\Gg=1$, we
  have $\sigma_x=\sigma^{d(x)}$, and $\sigma_x=\sigma_y$ if and only
  if $x=y$. If $\XX$ is unramified, then there exists $x$ with
  $e(x)=1=f(x)$ and $\Pic(\Hh)=\spitz{\sigma_x}$ by the preceding
  proposition. Assume that $x$ is such that $e_{\tau}(x)>1$. Since
  $\sigma_x$ does not lie in the subgroup of $\Pic(\Hh)$ generated by
  all $\sigma_y$, with $y\neq x$, by
  Corollary~\ref{cor:Pic-gen-separation}, we have that all $d(y)$ are
  multiples of $d(x)$, that is, $d(y)=a(y)\cdot d(x)$ with
  $a(y)\geq 1$. Since $\sigma_y$, for $y\neq x$, is the identity on
  $\Uu_x$, we see that $e_{\tau}(x)$ divides $a(y)$. We conclude that
  $\sigma_x$ generates $\Pic(\Hh)$, and $x$ is the only ramification
  point. 

  (2) This follows directly as a special case from the preceding
  theorem. 
\end{proof}
\subsection*{Examples} We illustrate the theory in some genus zero
examples over a perfect field.
\begin{example}[Finite fields]\label{ex:genus-0-finite-field}
  Let $k$ be a finite field and $\Hh$ over $k$ of genus zero. Then
  there is an efficient tubular shift $\sigma=\sigma_x$. Then
  $\tau\sigma^{2/\varepsilon}$ is a ghost. Moreover, for all points
  $p$ we have $[k(p):k]=\frac{\kappa\varepsilon}{s(\Hh)}\cdot f(p)$,
  by~\eqref{eq:deg-of-simples}. Without loss of generality we can
  assume that $k$ is the centre of $\Hh$. There are two possible
  cases, which describe \emph{all} genus zero cases over a finite
  field:\medskip

  (1) $\varepsilon=1$. \cite[Prop.~4.1]{kussin:2008}. Then
  $M=M(K,\alpha)=K\oplus K$ where $K$ acts canonically from the left
  and by $(x,y)\cdot z=(xz,y\alpha(z))$ from the right, and
  $\alpha\colon K\ra K$ is a $k$-automorphism. Since $k$ is the centre
  of $M$, it is the fixed field of $\alpha$, so that $\Gal(K/k)$ is
  cyclic, generated by $\alpha$, and $k$ is the centre of $M$. Let
  $n=[K:k]$. Then $s(\Hh)=n$. The case $n=1$ is the
  Kronecker/projective line over $k$. So assume $n\geq 2$. We have
  that $R=\Pi(L,\sigma_x)$ is isomorphic to $K[X;Y,\alpha]$. The two
  points $x$ and $y$ (corresponding to $X$ and $Y$) are the only
  rational points of multiplicity $1$, and the only rational points
  $p$ such that $e_{\tau}(p)>1$; moreover, $e_{\tau}(p)=s(\Hh)$;
  compare~\cite[Cor.~5.4.2]{kussin:2009}. Since also
  $\tau={\sigma_x}^{-1}{\sigma_y}^{-1}$, these are the only separation
  points. The ramification sequence is $\gge{e}=(n^1,n^1)$. The ghost
  group is cyclic of order $n$, generated by
  ${\sigma_x}^{-1}{\sigma_y}^{ }$. Moreover,
  $\tau={\sigma_x}^{-1}{\sigma_y}^{-1}$.\medskip

  (2) $\varepsilon=2$. Then $M={}_kK_K$, where $[K:k]=4$. Indeed, a
  priori we have $M={}_{F^{\alpha}}K_K$ with $F/k$ a finite field
  extension and $M_K=K_K$, with $[K:F]=4$, and $\alpha\in\Gal(K/k)$;
  the left $F$-structure is given by $f\cdot x=\alpha(f)x$. It is easy
  to see that for all $\alpha$ we obtain isomorphic bimodules (in the
  sense that the induced hereditary bimodule $k$-algebras are
  isomorphic, compare~\cite[5.1.3]{kussin:2009}). Thus we can assume
  that $M={}_FK_K$ is equipped with the canonical structure induced by
  the subring $F\subseteq K$. Then $F$ is the centre (in the sense
  of~\cite[0.5.5]{kussin:2009}) of the bimodule $M$; thus, we can
  assume $F=k$.

  There is a unique intermediate field $F$ of degree two over $k$,
  which is of the form $F=k(\alpha)$. This defines,
  by~\cite[Lem.~2.6]{kussin:2008}, the simple regular
  representation
  $$S_x=(\,k^2 \otimes K\stackrel{(1,\alpha)}\lra K\,)$$ with $f(x)=1$
  and $\End(S_x)=F$, hence $e(x)=1$. By uniqueness of $F$, we have
  that $x$ is the only rational point $p$ with $e(p)=1$. Hence
  $e_{\tau}(x)=2$ and $e_{\tau}(p)=1$ for all other rational
  points. Since $[k(p):k]=f(p)$ for all $p$ and $\End(L)=k$ is the
  field of constants, we deduce then
  from~\eqref{eq:general-euler-char-formula} (since $\genus=0$) that
  besides $x$ there is precisely one additional ramification point
  $p$, and this must satisfy $f(p)=2$ and $e(p)=1$. Thus the
  ramification sequence is $\gge{e}=(2^1,2^2)$. The ghost group $\Gg$
  is cyclic of order $2$, generated by
  ${\sigma_x}^{-2}{\sigma_p}^{ }$. We obtain
  $\tau={\sigma_x}^{ }{\sigma_p}^{-1}$. The automorphism group
  $\Aut(\XX)$ is cyclic of order $4$.

  We now additionally assume $\ch(k)\neq
  2$. By~\cite[Thm.~5.1]{kussin:2008},
  $$\Pi(L,\sigma_x)\simeq k\spitz{X,Y,Z}/\left(
    \begin{array}{c}
      XY-YX,\ XZ-ZX,\\
      YZ+ZY+a_1 X^2,\ Z^2 +c_0 Y^2-a_0 X^2   
    \end{array}\right)$$ for certain $a_0,\,a_1,\,c_0 \in
  k$. The point $x$ corresponds to the prime element given by the
  class of $X$ in $R$. The second ramification point $p$ corresponds
  to a prime element $\pi_p$ in $R$ of degree $2$, and which is also
  irreducible, that is, not a product of two elements of degree
  $1$. It follows that (up to multiplication with a non-zero element
  from $k$) there is precisely one such prime.
\end{example}
\begin{example}[Non-simple tame bimodule]
\label{ex:non-simple-bimodule}
  Let $k$ be perfect and $M$ a non-simple tame bimodule with centre
  $k$. It follows from~\cite[Prop.~11.5]{pierce:1982}
  and~\cite[Thm.~1.1.21]{jacobson:1996} that $M$ is of the form
  $F\oplus F$ with canonical $F$-action from the left, and the right
  $F$-action given by $(a,b)f=(af,b\alpha(f))$ for some
  $k$-automorphism $\alpha$ of $F$, and $F$ is a skew field over $k$
  of finite dimension. Let $n$ be the order of $\alpha$ considered as
  element in $\Gal(F/k)$. There is a rational point $x$ with $e(x)=1$,
  and hence $\sigma_x$ is efficient. We have
  $\Pi(L,\sigma_x)\simeq F[X;Y,\alpha]$, the twisted polynomial ring
  graded by total degree. The prime elements are given by $\pi_x=X$,
  $\pi_y=Y$, and some further polynomials lying in the centre, thus in
  the variables $X^n$ and $Y^n$. It follows that the points $x$ and
  $y$ are the only ramification points. The ramification sequence is
  $\gge{e}=(n^1,n^1)$. The ghost group $\Gg$ is cyclic of order $n$,
  generated by ${\sigma_x}^{-1}{\sigma_y}^{ }$. Moreover, we obtain
  $\tau={\sigma_x}^{-1}{\sigma_y}^{-1}$.
\end{example}
\begin{example}
  Let $M={}_{\QQ}\QQ(\sqrt{2},\sqrt{3})_{\QQ(\sqrt{2},\sqrt{3})}$.
  $$\Pi(L,\sigma_x)\simeq k\spitz{X,Y,Z}/\left(
    \begin{array}{c}
      XY-YX,\ XZ-ZX,\\
      YZ+ZY,\ Z^2 +2 Y^2-3 X^2   
    \end{array}\right).$$ Here, the three rational points
  $x,\,y,\,z$ (corresponding to the variables $X,\,Y,\,Z$,
  respectively) satisfy $e(p)=1$, and thus
  $e_{\tau}(p)=2=s(\Hh)$. Moreover, by~\cite[Prop.~7.1]{kussin:2008}
  we have $$\tau={\sigma_x}^{\ } {\sigma_y}^{-1}{\sigma_z}^{-1},$$ by
  which it follows again, that $x,\,y,\,z$ are the only separation
  points. The ramification sequence is $\gge{e}=(2^1,2^1,2^1)$. The
  ghost group is the Klein four group generated by
  ${\sigma_x}^{-1}{\sigma_y}^{ }$ and ${\sigma_x}^{-1}{\sigma_z}^{
  }$. Moreover, $\tau={\sigma_x}^{ }{\sigma_y}^{-1}{\sigma_z}^{-1}$.
\end{example}
\begin{example}
  \label{ex:K-H-bimodule} We consider the simple
  $(2,2)$-bimodule from~\cite[Ex.~5.7.3]{kussin:2009} with skewness
  $s(\Hh)=4$: Let $k=\mathbb{Q}$ and $F=\quat{-1}{-1}{\mathbb{Q}}$ be
  the skew field of quaternions over $\mathbb{Q}$ on generators
  $\qa{i}$, $\qa{j}$ with relations $\qa{i}^2=-1=\qa{j}^2$,
  $\qa{ij}=-\qa{ji}$, $K=\mathbb{Q}(\sqrt{-3},\sqrt{2})$ and $M$ be
  the bimodule ${}_{K}(K\oplus K)_{F}$ with the canonical $K$-action,
  and where the $F$-action on $M$ is defined by
  $$(x,y)\cdot\qa{i}=\frac{1}{\sqrt{-3}}(\sqrt{2}x+y,x-\sqrt{2}y),\ 
  (x,y)\cdot\qa{j}=(y,-x)$$
  for all $x$, $y\in K$. By~\cite[Prop.~5.7.5]{kussin:2009}, each
  rational point $x$ satisfies $e(x)=2$ or $e(x)=4$ (and both cases
  occur), and moreover $e^{\ast}(x)=1$; those with $e(x)=2$ are
  separation with $e(x)\cdot e^{\ast}(x)=2$ and
  $e_{\tau}(x)=2$. Actually, since the tame bimodule $M$ is linked to
  the tame bimodule in the preceding example via a derived equivalence
  in the tubular case (\cite[Prop.~8.3.1]{kussin:2009}) we will deduce
  from Example~\ref{ex:tubular-over-rationals} below that the
  ramification sequence is given by
  $\gge{e}=(2^1,2^1,2^1)$. Accordingly, the ghost group is the Klein
  four group (but not trivial, \cite[Prop.~5.7.4]{kussin:2009}).
\end{example}
\begin{remark}
  All the preceding examples are \emph{ruled} ($\gge{e}=(n,n)$) or
  \emph{half-ruled} ($\gge{e}=(2,2,2)$), in the terminology
  of~\cite{artin:dejong:2004}, see
  also~\cite[Prop.~4.2.4]{artin:dejong:2004}.
\end{remark}
\subsection*{An inseparable example}
We conclude this section with a detailed analysis of an inseparable
example which nicely illustrates that (and why) for many of the
preceding results the separability assumption was indispensable.
\begin{example}\label{ex:inseparable}
  Let $\FF_2$ be any field of characteristic $2$ and $k=\FF_2(t)$ the
  rational function field in one variable over $\FF_2$. Let
  $K=k(u)/(u^2-t)$ and the $k$-derivation $\delta\colon K\ra K$ be
  given by $\delta(u)=1$, and $\delta_{|k}=0$. We have $\delta^2=0$
  and $\Gal(K/k)=1$. Let $M$ be the tame $K$-$K$-bimodule
  $M=M(1,\delta)=K\oplus K$ with canonical left $K$-action and right
  $K$-action given by $(a,b)f=(af+b\delta(f),bf)$. Since $M$ is a
  non-simple bimodule, the corresponding curve $\Hh$ of genus zero
  admits a point $x$ with $e(x)=1=f(x)$, and the orbit algebra
  $\Pi(L,\sigma_x)$ is isomorphic, as graded ring, to the differential
  polynomial ring $R=K[X;Y,1,\delta]$. Here $Y$ and the central
  variable $X$ have degree $1$, and we have the relations
  $Yf=\delta(f)X+fY$ ($f\in K$). The point $x$ above is associated
  with the prime element $X$, cf.\ \cite[Prop.~1.7.3]{kussin:2009}.
  It is easily shown that the centre is given by $k[X,Y^2]$. We
  conclude that $s(\Hh)=2$.\medskip

  (1) We consider the tube $\Uu=\Uu_x$ associated with $x$. For the
  simple $S=S_x$ we have $\End(S)\simeq K$, in particular
  $e^{\ast}(x)=1$.  We have $V:=\widehat{R}_x\simeq\End(S[\infty])$,
  the complete local ring with $\Uu\simeq\mod_0(V)$. From
  Theorem~\ref{thm:general-skewness-equation} we conclude that $V$ has
  PI-degree $e^{\ast\ast}(x)=2$. By~\cite[7.4]{ringel:1976} the
  $K$-$K$-bimodule $\Ext^1(S,S)$ is isomorphic to
  $M/N\simeq{}_{K}K_{K}$, the canonical one-dimensional
  $K$-$K$-bimodule, where $N$ is the subbimodule $K\oplus 0$ of
  $M$. From this we get $\wgr(V)\simeq K[[T]]$, which is commutative,
  and thus $\wgr(V)\not\simeq V$. Since $\Gal(\End(S)/k)=1$, for the
  $\tau$-multiplicity~\eqref{eq:def-tau-multi} we have
  $e_{\tau}(x)=1$. We thus see that for the inseparable point $x$ the
  conclusions of Propositions~\ref{prop:relationships}
  and~\ref{prop:R=wgrR}, and
  Theorems~\ref{thm:special-skewness-equation}
  and~\ref{thm:pruefer-end-skew-power} do not hold.\medskip

  (2) We shall determine the algebra $V=\widehat{R}_x$
  explicitly. Since $K$ is a subalgebra of $R_x$, it is also a
  subalgebra of the completion $V$. Let $\pi\in V$ be the generator of
  the Jacobson radical which is given by the element $XY^{-1}$ (cf.\
  the proof of \cite[Thm.~2.2.10]{kussin:2009}).  Since
  $V/(\pi)\simeq K$ and $\Gal(K/k)=1$, we get a decomposition
  $V=K\oplus (\pi)$ of $K$-$K$-bimodules. Using $\delta(a)\in k$ the
  relations in $R$ induce the relations
  $$\pi a=a\pi +\delta(a)\pi^2\quad\quad (a\in K).$$ The decomposition
  above also yields the decomposition $V\pi=K\pi\oplus (\pi^2)$ of
  left $K$-modules, but $K\pi$ is not a $K$-$K$-bimodule. By the
  universal property of power series rings, we get a $k$-algebra
  homomorphism $\phi\colon K[[t^{-1},\delta]]\ra V$ sending $t^{-1}$
  to $\pi$, where $K[[t^{-1},\delta]]$ denotes the (pseudo-)
  differential power series ring defined
  in~\cite[Thm.~1.11.8]{jacobson:1996}; it is a local domain with
  Jacobson radical $J=(t^{-1})$ and $\cap_{n\geq 1}J^n =0$.  By the
  above decompositions $\phi$ is surjective. It is also injective,
  since the kernel is a completely prime ideal. Thus $\phi$ is an
  isomorphism. We hence write $$V=K[[\pi,\delta]].$$ Moreover, since
  $\pi^2 a=a\pi^2$ for all $a\in K$, the power series ring
  $K[[T]]\simeq\wgr(V)$ is isomorphic to the subring $K[[\pi^2]]$ of
  $V$. The centre of $V$ is the subring $k[[\pi^2]]$. For the
  ramification index (of the exponential
  valuations~\cite[(13.1)]{reiner:2003}) we have
  $e_{\ramif}(x)=2$.\medskip

  (3) There is an automorphism $\sigma$ of $V$ given by
  $\pi r=\sigma(r)\pi$ (for $r\in R$). We have $\sigma(\pi)=\pi$ and
  $\sigma(f)=f+\delta(f)\pi$ ($f\in K$), and we see that $\sigma$ has
  order $2$, and since the centre is $k[[\pi^2]]$ it is easily seen to
  be not inner. (The property $\sigma(K)\not\subseteq K$ makes the
  difference to the separable case.) By~\cite[Thm.~3.1.2]{kussin:2009}
  multiplication with $X$ yields the natural transformation
  $1_{\Hh}\stackrel{x}\ra\sigma_x$. Extending $\sigma_x$ to the direct
  limit closure of $\Uu_x$, we see that the natural sequence
  in~\cite[0.4.2(5)]{kussin:2009} for the injective object $S[\infty]$
  becomes $0\ra S\ra S[\infty]\stackrel{\pi}\lra S[\infty]\ra 0$. We
  conclude that $\sigma_x$ on $\Uu$ is induced by the automorphism
  $\sigma$ and is thus of order $2$; it acts non-trivially e.g.\ on
  $\End(S[n])$ for $n\geq 2$.\medskip

  (4) Since $\tau^-(L)\simeq {\sigma_x}^2(L)$, the composition
  $\tau\circ{\sigma_x}^2$, on $\Hh$, is an element of the ghost
  group. Using~\cite[Thm.~3.1]{kussin:2008}, computing the graded
  automorphisms $\alpha$ of $R$ which are prime fixing (and hence
  preserve $kX$ and elements of the centre $k[X,Y^2]$), it is easy to
  see that $\Gg(\Hh)=1$. (Indeed, since $\alpha(X)\in kX$, we can
  assume $\alpha(X)=X$. Exploiting $\alpha(Y^2)\in kY^2$ and
  $\alpha(Yb)=\alpha(Y)b$ for all $b\in K$, we obtain $\alpha(Y)=Y+aX$
  with $a\in K$ satisfying $a^2=\delta(a)$. If $a\neq 0$, then
  $a^{-1}Ya=Y+aX$, so that $\alpha$ is inner.) We conclude that
  globally $\tau^-={\sigma_x}^2$ holds. (We remark that here $\tau$ is
  not given by formula~\eqref{eq:general-tau-picard}.)  This shows
  that $\tau^-$ acts, unlike $\sigma_x$, as the identity functor on
  $\Uu_x$. Thus we see that Corollary~\ref{cor:tau-minus=sigma-x}~(2)
  and Corollary~\ref{cor:e-tau=e-ramif} do not extend to inseparable
  points. Moreover, Theorem~\ref{thm:special-skewness-equation} does
  not hold for such a point even if $e_{\tau}(x)$ is replaced by the
  order of $\tau$ in $\Aut(\Uu_x/k)$. We also infer that $\Hh$ is
  $\tau$-unramified but not unramified.
\end{example}

\section{The real case: Witt curves}
\label{sec:Witt-curves}
\begin{numb}[Real smooth projective curves and Klein surfaces]
  If $k$ is algebraically closed, then by
  Corollary~\ref{cor:algebr-closed} each noncommutative regular
  projective curve over $k$ is actually commutative. For $k=\CC$ the
  field of complex numbers it is well-known that the three concepts
  regular (=smooth) projective curves $X$ over $\CC$, algebraic
  function fields $K$ in one variable over $\CC$, and compact Riemann
  surfaces $S$ are equivalent/dual to each other; here $K$ is the
  field of meromorphic functions on $S$ (which are the holomorphic
  functions $\alpha\colon S\ra\SS^2$ to the Riemann sphere) and also
  the function field $k(X)$. Over the field $k=\RR$ of real numbers
  there are similar correspondences, where the Riemann surfaces are
  replaced by the Klein surfaces $\Kk$,
  \cite{alling:greenleaf:1971,natanzon:1990}. Each (compact) Klein
  surface $\Kk$ is of the form $S/\sigma$, where $S$ is a compact
  Riemann surface and $\sigma\colon S\ra S$ an antiholomorphic
  involution, \cite[Thm.~1.1]{natanzon:1990}; the Riemann surface $S$
  is also called the complex double of $\Kk$. It should be noted that
  in such a case $\chi_{top}(S)=2\cdot\chi_{top}(\Kk)$ holds
  (\cite[1.6.9]{alling:greenleaf:1971}); since $k=\RR$, we also have
  $\chi'(Y)=2-2g=\chi_{top}(S)$ and, by~\cite[Thm.~1.1]{alling:1974},
  $\chi'(X)=1-g=\chi_{top}(\Kk)$, where $Y$ and $X$ are the
  corresponding \emph{real} regular projective curves, respectively,
  and $g=g(S)=g(\Kk)$; here, $\chi'$ is the normalized Euler
  characteristic as defined in~\eqref{eq:char-genus}, and $\chi_{top}$
  the usual Euler characteristic for surfaces defined topologically
  via triangulations. The real points on $\Kk$ (if any) form the
  boundary $\partial\Kk$. By Harnack's theorem $\partial\Kk$ has at
  most $g(\Kk)+1$ components, called \emph{ovals}, since they are
  homeomorphic to a circle $\SS^1$.  The ovals are given by the set
  $S^{\sigma}$ of fixed points of $\sigma$. By a theorem of
  Weichold~\cite[p.~56]{natanzon:1990}, every $\Kk$ is, topologically,
  uniquely determined by a triple $(g,t,s)$, where $g=g(\Kk)=g(S)$ is
  the genus of the Riemann surface $S$, $t$ is the number of ovals,
  and $s=0$ if $S\setminus S^{\sigma}$ is connected, and $s=1$
  otherwise. Moreover, precisely the triples $(g,t,s)$ with $s=0$ and
  $t\leq g$, or $s=1$, $t\equiv g+1\ (\mod 2)$ and $1\leq t\leq g+1$
  occur.
\end{numb}
\begin{numb}[Noncommutative function fields and configurations on
  Klein surfaces]
  The field $k(\Kk)$ of meromorphic functions
  $\alpha\colon\Kk\ra\SS^2$ on a Klein surface $\Kk$ is an algebraic
  function field in one variable over $\RR$, and each such real
  function field is of this form. (For the precise definition of a
  meromorphic function on a Klein surface we refer to~\cite[Ch.~1,\,\S
  3]{alling:greenleaf:1971}.)  Finite dimensional central skew fields
  over the meromorphic function field $K=k(\Kk)$ of a Klein surface
  $\Kk$ are, if non-trivial, quaternion skew fields of the form
  $\quat{\alpha}{-1}{K}$, where $0\neq\alpha\colon\Kk\ra\SS^2$ is a
  meromorphic function. This follows easily from Tsen's theorem,
  \cite{witt:1934}. We recall that such an algebra is of dimension
  four over $K$ on generators $\gge{i},\,\gge{j}$ and relation
  $\gge{i}^2=-1$, $\gge{j}^2=\alpha$ and $\gge{ji}=-\gge{ij}$. Such a
  function $\alpha$ is real-valued on the boundary $\delta\Kk$, that
  is, on each oval. On each of the ovals $\alpha$ might have an even
  number (or zero) \emph{sign-changes}, that is, zeros or poles of odd
  order. Thus $\alpha$ determines on $\Kk$ what we call a
  $\pm$-\emph{configuration}, which is given by
  \begin{itemize}
  \item an even number ($\geq 0$) of points on each oval, called
    \emph{segmentation points};
  \item each open \emph{segment} between segmentation points, and each
    oval without segmentation points, labelled by either the sign $+$
    or the sign $-$, in an alternating way (changing the sign at each
    segmentation point).
  \end{itemize}
  We call it \emph{clean} if there are no segmentation points at
  all. The $\pm$-configuration is \emph{induced} by the function
  $\alpha$, if the $\pm$-configuration reflects the sign-behaviour of
  $\alpha$ on the ovals; a segment then has a $+$, if $\alpha$ is
  non-negative on this segment (we write $\alpha(x)>0$, locally), and
  has a $-$, if $\alpha$ is non-positive on this segment
  ($\alpha(x)<0$).  We recall from~\cite{witt:1934} that
  $\alpha\neq 0$ is called \emph{positive definite} if it never
  becomes negative on any oval. (\cite[I.]{witt:1934} says that then
  $\alpha$ is of the form $\alpha=\beta^2+\gamma^2$.) Moreover,
  $\alpha$ is \emph{definite}, if there is no sign-change on any
  oval. So, $\quat{\alpha}{-1}{K}$ does not split (it is a skew field)
  if and only if $\alpha$ is not positive definite.
\end{numb}
\begin{example}
  Let $\DD$ be the compact unit disc. We have $K=k(\DD)=\RR(t)$, the
  rational function field over $\RR$ in one variable. We consider
  three different $\pm$-configurations on $\DD$.\medskip

  (a) Let $\alpha=-1\colon\DD\ra\SS^2$ be the function with constant
  value $-1$. This gives the clean $\pm$-configuration as shown in
  Figure~\ref{fig:disc-configs}. The associated quaternion skew field
  $\quat{-1}{-1}{K}$ is the function field $\HH(t)$ in one (central)
  variable $t$ over $\HH$.\medskip

  (b) Let
  $\alpha\colon\DD\ra\SS^2$, $z\mapsto z$ be the canonical
  identification of $\DD$ with the ``northern'' half ball, the
  ``equator'' defining $\RR\cup\{\infty\}$. On the unique oval
  $\alpha$ has sign-changes in $z=0$ (zero of order $1$) and
  $z=\infty$ (pole of order $1$). These two segmentation points yield
  the two segments of the oval, the negative real numbers marked by
  $-$, the positive real points marked by $+$. The quaternion skew
  field $\quat{\alpha}{-1}{K}$ is isomorphic to $\CC(u,\sigma)$, the
  skew function field ($uz=\sigma(z)u$ for all $z\in\CC$, with
  $\sigma$ the complex conjugation), in the variable
  $u=t^{1/2}$.\medskip

  (c) Let similarly $\alpha$ be a meromorphic function associated with
  the element $t(t-1)(t+1)$ in $\RR(t)$
  (cf.~\cite[Thm.~1.4.6]{alling:greenleaf:1971}). This gives rise to
  four segmentation points $z=0,\,1,\,-1,\,\infty$, and the skew
  function field $\CC(u,t)/(u^2-t(t-1)(t+1),u\gge{i}+\gge{i}u)$.
 \begin{figure}[h]
 \begin{tikzpicture}[line width=1.1pt]
 \coordinate [label=left:$-$] (C) at (1.4,1.0);
\coordinate [label=below:$$] (B) at (0.0,-1.55);
 \draw[black] (0,0) circle (1.2);
 \fill[gray,opacity=.25] (0,0) circle (1.19);
 \end{tikzpicture}\quad\quad
\begin{tikzpicture}[line width=1.1pt]
\coordinate [label=right:$-$] (C) at (-1.8,0.0);
\coordinate [label=left:$+$] (D) at (1.8,0.0);
\coordinate [label=above:$\infty$] (A) at (0.0,1.3);
\coordinate [label=below:$0$] (B) at (0.0,-1.3);
\draw[black] (0,0) circle (1.2);
\fill[gray,opacity=.25] (0,0) circle (1.19);
\fill[red,opacity=1.0] (0,1.2) circle (2.0pt);
\fill[red,opacity=1.0] (0,-1.2) circle (2.0pt);
\end{tikzpicture}\quad\quad
\begin{tikzpicture}[line width=1.1pt]
\coordinate [label=right:$+$] (E) at (1.0,0.8);
\coordinate [label=right:$-$] (F) at (1.0,-0.8);
\coordinate [label=left:$+$] (G) at (-1.0,-0.8);
\coordinate [label=left:$-$] (H) at (-1.0,0.8);
\coordinate [label=right:$-1$] (C) at (-1.9,0.0);
\coordinate [label=left:$1$] (D) at (1.7,0.0);
\coordinate [label=above:$\infty$] (A) at (0.0,1.3);
\coordinate [label=below:$0$] (B) at (0.0,-1.3);
\draw[black] (0,0) circle (1.2);
\fill[gray,opacity=.25] (0,0) circle (1.19);
\fill[red,opacity=1.0] (0,1.2) circle (2.0pt);
\fill[red,opacity=1.0] (0,-1.2) circle (2.0pt);
\fill[red,opacity=1.0] (1.2,0) circle (2.0pt);
\fill[red,opacity=1.0] (-1.2,0) circle (2.0pt);
\end{tikzpicture}
\caption{The disc $\DD$ with different
  $\pm$-configurations. Left: (a), middle: (b), right: (c)}
\label{fig:disc-configs}
\end{figure}
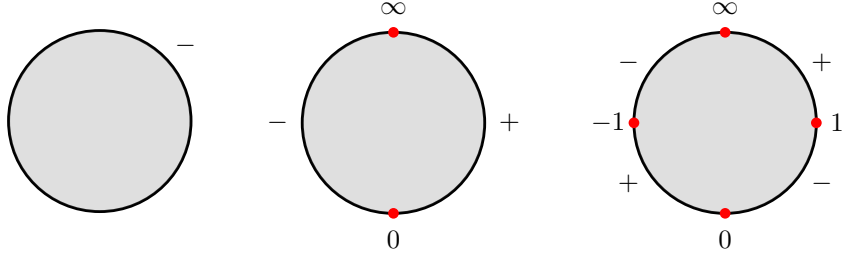
\end{example}
Witt's theorem~\cite{witt:1934}, here formulated in our language,
shows that all noncommutative real algebraic function fields in one
variable are obtained by Klein surfaces with $\pm$-configurations.
\begin{theorem}[{Witt;
  cf.~\cite[Thm.~2.4.5]{alling:greenleaf:1971}}]
Let $\Kk$ be a Klein surface whose boundary is given by $t$ ovals. We
assume $t\geq 1$. Let $K$ be the field of meromorphic functions on
$\Kk$.
  \begin{enumerate}
  \item[(1)] Every (clean) $\pm$-configuration on $\Kk$ is induced by
    a (definite) meromorphic function
    $\alpha\colon\Kk\ra\SS^2$. \cite[II.+III.]{witt:1934}
  \item[(2)] The resulting quaternion algebra $A=\quat{\alpha}{-1}{K}$
    is uniquely determined already by the $\pm$-configuration. It is a
    skew field if and only if $\alpha$ is not positive definite. It is
    unramified if and only if $\alpha$ is definite. Otherwise its
    ramification points are just the segmentation
    points. \cite[p.~10]{witt:1934}
  \item[(3)] Each finite dimensional central skew field extension of
    $K$ is obtained in this way. \cite[III'.]{witt:1934}
  \end{enumerate}
\end{theorem}
\begin{numb}[Local data]\label{numb:Witt-local-data}
  Witt~\cite[p.~10]{witt:1934} described (function-theo\-reti\-cally)
  also the local data. We assume that $\alpha\neq 0$ on
  $\Kk$ is \emph{not} positive definite. That is, there is an oval on
  which
  $\alpha$ becomes negative. For convenience we already use the
  notions $S_x$ and
  $e(x)$ in each concluding statement (``Thus...''), which will get a
  proper meaning only below when we define the notion of a Witt curve.
\begin{itemize}
\item If $x$ is inner then $\quat{\alpha}{-1}{K}_x$ splits. Thus
  $\End(S_x)=\CC$ and $e(x)=2$.
\end{itemize}
If $x$ is a boundary point ($x=x^{\sigma}$), then
\begin{itemize}
\item If $\alpha(x)>0$ (that is, $\alpha$ does not change its positive
  sign in a neighbourhood of $x$) then $\quat{\alpha}{-1}{K}_x$
  splits. Thus $\End(S_x)=\RR$ and $e(x)=2$.
\item If $\alpha(x)<0$ (that is, $\alpha$ does not change its negative
  sign) then $\quat{\alpha}{-1}{K}_x$ does not split, and $x$ is inert
  in $\quat{\alpha}{-1}{K}$. Thus $\End(S_x)=\HH$ and $e(x)=1$.
\item Otherwise (if $\alpha$ changes the sign in $x$, that is, $x$ is
  segmentation) $\quat{\alpha}{-1}{K}_x$ does not split. Thus
  $\End(S_x)=\CC$ and $e(x)=1$.
\end{itemize}  
\end{numb}
Thus the interior of a connected closed segment of an oval, whose
endpoints are segmentation points, is ``coloured'' \emph{real}, in
case
$\alpha$ is nonnegative on this segment, and \emph{quaternion}, in
case
$\alpha$ is negative on this segment. From now on we will use these
colourings of segments by $\RR$ and $\HH$, instead of $+$ and
$-$, respectively, cf.\ Figure~\ref{fig:genus-zero} below. (Inner
points and segmentation points are always complex.) If $\alpha\neq
0$ is \emph{not} positive definite then we call
$\alpha\colon\Kk\ra\SS^2$ a \emph{Witt function} (and
$\Kk$ with the induced
$\pm$-configuration, formally, a \emph{Witt surface}); we will give
$(\Kk,\alpha)$ a canonical structure of a noncommutative regular
projective curve below. Some of the following considerations are
reformulations of results of Section~\ref{sec:orders}.

Let $\Kk$ be a Klein surface with function field $K=k(\Kk)$ and
$\alpha$ a Witt function. Let $A=\quat{\alpha}{-1}{K}$ be the
corresponding quaternion skew field. Let
$A(x)=A_x\otimes_{\Oo_x}k(x)=A_x/\mathfrak{m}_xA_x$ with
$k(x)=\Oo_x/\mathfrak{m}_x$ be the geometric fibre. If $\alpha>0$,
then $A(x)\simeq\matring_2(\RR)$ is split on an oval $O$, if
$\alpha<0$, then $A(x)\simeq\HH$ on $O$. The segmentation points are
just the ramification points of $A$. There is the injective
homomorphism
\begin{equation}
  \label{eq:brauer}
  \beta\colon\Br(\Kk)\ra\Br(K)
\end{equation}
of Brauer groups. Here, $\Br(\Kk)$ consists of (classes of) Azumaya
algebras $\Aa$. The homomorphism $\beta$ sends the class of $\Aa$ to
the class of $\Aa_\xi$, where $\xi$ is the generic point. In the image
of $\beta$ are precisely those $A=\quat{\alpha}{-1}{K}$, which are
unramified on $\Kk$. (We refer to~\cite{demeyer:knus:1976}.)  In other
words, if $A$ is unramified on $\Kk$, then it can be equipped with the
structure of a unique Azumaya algebra
$\Aa$. In~\cite[Thm.~1.3.7]{caldararu:2000}
(also~\cite{caldararu:2002}) it is shown that the category
$\Hh=\coh(\Aa)$ of coherent $\Aa$-modules is equivalent to the
category $\coh(\Kk,\alpha)$ of $\alpha$-twisted coherent sheaves on
$\Kk$.
\begin{proposition}
  Assume the Witt function $\alpha$ is definite. Let $\Aa$ be the
  corresponding Azumaya algebra. The category $\Hh=\coh(\Aa)$ of
  coherent $\Aa$-modules is a noncommutative regular projective curve
  with $s(\Hh)=2$.
\end{proposition}
\begin{proof}
  Since $\alpha$ is not positive definite, by Witt's
  theorem~\cite{witt:1934} the quaternion algebra
  $A=\quat{\alpha}{-1}{K}$ does not split. Since $\alpha$ is
  definite, $A$ is unramified. The assertion follows from
  Theorem~\ref{thm:structure}.
\end{proof}
We now treat the general (ramified) situation. Denote by
$U=\Kk\setminus\{x_1,\dots,x_n\}$ the Zariski-open unramified
locus. The given quaternion algebra $A=\quat{\alpha}{-1}{K}$ defines
an Azumaya algebra $\Aa_U$ in $\Br(U)$.
\begin{theorem}\label{thm:WA-is-NC}
  Each Witt function $\alpha\colon\Kk\ra\SS^2$ gives rise to a
  noncommutative regular projective curve $\Hh$ of skewness $s(\Hh)=2$,
  and with the following properties:
  \begin{enumerate}
  \item[(1)] The centre curve is $\Kk$. 
  \item[(2)] The function field is $k(\Hh)=\quat{\alpha}{-1}{k(\Kk)}$.
  \item[(3)] Up to an equivalence of categories, $\Hh$ is uniquely
    determined by $\Kk$ and the coloured segments of the ovals.
  \end{enumerate}
\end{theorem}
\begin{proof}
  Let $\Kk$ and $\alpha$ be given. Denote by $\Oo=\Oo_{\Kk}$ the
  structure sheaf of $\Kk$. With $K=k(\Kk)$ let $A$ be the quaternion
  skew field $\quat{\alpha}{-1}{K}$ over $K$. Let $U\subseteq\Kk$ be
  the unramified locus and $j\colon U\ra\Kk$ the inclusion. The
  function field of $U$ is
  $K$. By~\cite[Cor.~1.9.6]{artin:dejong:2004} there exists an Azumaya
  $\Oo_U$-algebra $\Aa'$ in
  $A$. By~\cite[Prop.~1.8.1]{artin:dejong:2004} there is an
  $\Oo$-order $\Bb$ in $A$ with
  $j^{\ast}\Bb=\Aa'$. By~\cite[Prop.~1.8.2]{artin:dejong:2004} there
  is a maximal $\Oo$-order $\Aa$ in $A$ containing $\Bb$. Then, by
  Theorem~\ref{thm:structure}, $\Hh=\coh(\Aa)$ is a noncommutative
  regular projective curve over $\RR$. Moreover, conditions (1) and (2)
  are clearly satisfied.

  (3) This follows from part~(3) of Witt's theorem above, and
  Theorem~\ref{thm:function-field-determines}.
\end{proof}
\begin{definition}
  Let $(\Kk,\alpha)$ be a Klein surface with a Witt function
  $\alpha\colon\Kk\ra\SS^2$. (Recall that this means that $\alpha$ is
  not positive definite.) We call the noncommutative regular projective
  curve constructed in the preceding theorem \emph{Witt curve}
  (associated with $(\Kk,\alpha)$).
\end{definition}
\begin{theorem}
  Each noncommutative regular projective curve $\Hh$ over $k=\RR$ with
  $s(\Hh)>1$ is a Witt curve.
\end{theorem}
\begin{proof}
  The centre curve of $\Hh$ is a real regular projective curve, thus a
  Klein surface $\Kk$. The function field $k(\Hh)$ is a skew field of
  quaternions over $K=k(\Kk)$, thus of the form
  $\quat{\alpha}{-1}{K}$. By the uniqueness part of
  Theorem~\ref{thm:WA-is-NC} then $\Hh$ is the Witt curve associated
  with $(\Kk,\alpha)$.
\end{proof}
We will call \emph{commutative} real regular projective curves with
centre $\RR$ (thus corresponding to the Klein surfaces) also
\emph{Klein curves}, and will use the letter $X$ instead of
$\Kk$.\medskip

The main result for the $\tau$-multiplicities for Witt curves is the
following.
\begin{theorem}
  Let $\Hh$ be a Witt curve. Then $e_{\tau}(x)=2$ if and only if $x$
  is a segmentation point. In other words: precisely for the
  segmentation points $\tau$ is acting non-trivially -- by complex
  conjugation -- on the corresponding tubes.
\end{theorem}
\begin{proof}
  This follows from Corollary~\ref{cor:separ-ramif}.
\end{proof}
\begin{numb}\label{numb:table-of-local-data}
  Let $\Hh$ be a Witt curve. Table~\ref{tab:local-fields} summarizes
  some local data.
  \begin{table}[h!]
    \centering
    $$\begin{array}{l|ccc|cccc}
      \text{point}\ x & e(x) & e^{\ast}(x) & e_{\tau}(x) & k(x) &
      Z(D_x) & D_x & \widehat{D}_x\\ 
      \hline
      \text{inner} & 2 & 1 & 1 & \CC & \CC & \CC & \CC((T))\\
      \text{real} & 2 & 1 & 1 & \RR & \RR & \RR & \RR(T))\\
      \text{quaternion} & 1 & 2 & 1 & \RR & \RR & \HH & \HH((T))\\
      \text{segmentation} & 1 & 1 & 2 & \RR & \CC & \CC & \CC((T,\sigma))
    \end{array}$$
    \caption{Local data of a Witt curve}
    \label{tab:local-fields}
  \end{table}
\end{numb}
A Klein surface has the constants field $\RR$. We will now determine
the field of constants $\End(L)$ of a Witt curve $(\Hh,L)$. This
terminology is justified since it follows from properties of maximal
orders that the endomorphism ring of any line bundle can be embedded
into the endomorphism ring of $L$.
\begin{lemma}[Field of constants]\label{lem:constants-field}
  Let $(\Hh,L)$ be a Witt curve. Let $r$ be the number of completely
  real coloured ovals, and $n=2m$ the total number of segmentation
  points. Then
    $$\End(L)\simeq
    \begin{cases}
      \CC & \text{if}\ m>0\ \ \text{or}\ \ r>0,\\
      \HH & \text{if}\ m=0\ \ \text{and}\ \ r=0.
    \end{cases}$$ For the normalization factor $\varepsilon$
    from~\eqref{eq:def-deg} we have $\varepsilon=1$.
  \end{lemma}
\begin{proof}
  Let $K=k(X)$ be the function field of the centre curve, so
  $k(\Hh)=\quat{\alpha}{-1}{K}$. Scalar extension by $\CC$, denoted by
  the overline symbol, then gives
  $\overline{\Aa}\otimes\overline{K}=\matring_2(\overline{K})$. Assume
  first that $m=0$. Then $\overline{\Aa}$ is Azumaya, and
  by~\cite[Cor.~1.7.6]{artin:dejong:2004} there is a locally-free
  $\overline{\Oo}$-module $\Ee$ of rank $2$ such that
  $\overline{\Aa}\simeq\ShEnd_{\overline{O}}(\Ee)$. Then
  $\End(L)=\End(\Aa)=\RR$ is not possible, since otherwise
  $\End_{\overline\Aa}(\overline\Aa)\simeq\CC$, which is not the case
  by the sentence before. This also holds in case $m>0$: then
  $\overline{\Aa}$ is not Azumaya, instead it is weighted by an even
  number of the weight $2$, and the endomorphism ring of the structure
  sheaf is not changed by insertion of weights. So in any case
  $\End(L)=\CC$ or $\HH$.

  Thus $\kappa=\dim_k\End(L)\geq 2$, and we have $\kappa\varepsilon=2$
  in cases $m>0$ or $r>0$ and $\kappa\varepsilon=4$ in the cases
  $m=0,\,r=0$: this follows from the existence of a simple object of
  degree $1$ (by the definition of $\varepsilon$) and the
  formula~\eqref{eq:deg-of-simples}.

  Thus it remains to show that in case $m=0,\,r=0$ we have
  $\kappa=4$. In this case $\Aa$ is Azumaya with each of its ovals
  coloured quaternion. Then $k(\Hh)=k(X)\otimes\HH$ and
  $\Aa=\Oo\otimes\HH$, hence $\End(\Aa)=\HH$ follows.
\end{proof}
\begin{proposition}[Hurwitz genus formula for Witt
  curves]\label{prop:hurwitz}
  Let $\Hh$ be a Witt curve with $n=2m$ segmentation points and
  underlying Klein curve $X$, and let $r$ be the number of completely
  real coloured ovals. Then
  \begin{equation}
    \label{eq:hurwitz}
       \genus =
       \begin{cases}
         2g(X)-1+m & \text{if}\ m>0\ \ \text{or}\ \ r>0,\\
          g(X) & \text{if}\ m=0\ \ \text{and}\ \ r=0.
       \end{cases}
  \end{equation}
\end{proposition}
\begin{proof}
  By formula~\eqref{eq:general-euler-char-formula} we have for the
  normalized Euler characteristics $\chi'(\Hh)=\chi'(X)-n/4$. With
  $\kappa=\dim_k\End(L)$ we obtain
  $\kappa(1-g(\Hh))=4\chi'(\Hh)=4\chi'(X)-n=4(1-g(X))-n$, thus
  \begin{equation}
    \label{eq:general-real-hurwitz}
    g(\Hh)=\frac{4}{\kappa}g(X)-\frac{4}{\kappa}+1+\frac{2m}{\kappa}.
  \end{equation}
  Now the assertion follows with Lemma~\ref{lem:constants-field}. 
\end{proof}
\begin{remark}
  (1) Our definition of the genus differs from the
    definition of the genus of function skew fields
    in~\cite{witt:1934b}, since we always have $\genus\geq 0$. The
    analogues formula obtained in~\cite{marubayashi:vanoystaeyen:2012}
    is $4g(X)-3+n$, which we would get for $\kappa=1$. For
    example, the function fields $\HH(T)$ and $\CC(T,\sigma)$ have
    genus $0$ by our definition, but genus $-3$ by the other
    definitions. But the conditions $g\leq 0$ and $g=1$ are equivalent
    for both definitions.\medskip

    (2) Let $\Kk$ be a Klein surface with a $\pm$-configuration
    induced by $\alpha$. We form a double cover $\pi\colon\Kk'\ra\Kk$
    of Klein surfaces as in the proof of Witt's theorem
    in~\cite[Thm.~2.4.5]{alling:greenleaf:1971}: $\Kk'$ is obtained by
    gluing two copies of $\Kk$ together at the closures of the
    segments, or complete ovals, where $\alpha(x)<0$. The boundary of
    $\Kk'$ is then given by two copies of the segments (or ovals),
    where $\alpha(x)>0$, and each pair of these segments (if not an
    entire oval) is bounded by two segmentation points of $\Kk$. The
    connected components are then ovals, each of which contains
    precisely zero or two of the segmentation points from
    $\Kk$. Setting $\alpha'=\pi^{\ast}(\alpha)=\alpha\circ\pi$, then
    $\alpha'$ is positive on the segments. Thus the former
    segmentation points are not longer segmentation points on
    $\Kk'$. By construction $k(\Kk')=k(\Kk)(\sqrt{\alpha})$. Since
    $\Kk'$ is a Riemann surface only in case $m=0$ and $r=0$,
    formula~\eqref{eq:hurwitz} (for $\Hh$) coincides with the Hurwitz
    equation for the covering $\pi\colon\Kk'\ra\Kk$
    in~\cite[Thm.~2]{may:1975}.
\end{remark}
\begin{corollary}\label{cor:hurwitz-consequence}
  Let $n>0$ be the number of segmentation points. Then:
  \begin{itemize}
  \item $\genus=0$ if and only if $g(X)=0$ (hence $\Kk$ is the compact
    disc) and $n=2$.
  \item $\genus=1$ if and only if $g(X)=0$ and $n=4$. \qed
  \end{itemize}
\end{corollary}
\begin{example}[Genus zero]
  There are two Witt curves of genus zero; we describe the
  corresponding Witt surfaces, Figure~\ref{fig:genus-zero}.\medskip

  (1) The compact disc $\DD_{\HH}$, the boundary coloured
  quaternion. This is the projective spectrum of the graded polynomial
  ring $\HH[X,Y]$ with central variables $X$, $Y$ of degree $1$. The
  function field is given by $\HH(T)$.\medskip

  (2) The compact disc $\DD_{2,2}$ with two segmentation points on its
  boundary. This is the projective spectrum of the graded
  skew-polynomial algebra $\CC[X;Y,\sigma]$. Accordingly, the function
  field is $\CC(T,\sigma)$.
\begin{figure}[h]
 \begin{tikzpicture}[line width=1.1pt]
 \coordinate [label=left:$\HH$] (C) at (1.4,1.0);
 \draw[black] (0,0) circle (1.2);
 \fill[gray,opacity=.25] (0,0) circle (1.19);
 \end{tikzpicture}\quad\quad
\begin{tikzpicture}[line width=1.1pt]
\coordinate [label=right:$\HH$] (C) at (-1.8,0.0);
\coordinate [label=left:$\RR$] (D) at (1.8,0.0);
\draw[black] (0,0) circle (1.2);
\fill[gray,opacity=.25] (0,0) circle (1.19);
\fill[red,opacity=1.0] (0,1.2) circle (2.0pt);
\fill[red,opacity=1.0] (0,-1.2) circle (2.0pt);
\end{tikzpicture}
\caption{The Witt curves with $\genus=0$: $\DD_{\HH}$ and $\DD_{2,2}$}
\label{fig:genus-zero}
\end{figure}
\end{example}
\begin{numb}[Euler characteristic of a Witt curve]
\label{numb:euler-char-witt}
  We now come back to the question how to normalize the Euler
  characteristic ``correctly''.

  If $\Kk$ is a Klein surface and $S$ the corresponding complex double
  then, since $\Kk$ is a $\ZZ_2$-quotient of $S$, we have
  $\chi(S)=2\cdot\chi(\Kk)$. This can also be expressed in the
  following way: if $\Hh$ is a real regular projective curve, then
  \begin{equation}
    \label{eq:chi-tensor-formula}
    \chi(\Hh)=\chi(\Hh\otimes\CC)/2.
  \end{equation}
  We assume that this should also hold for the Euler characteristic of
  Witt curves $\Hh$. Here, $\Hh\otimes\CC$ (tensor product over
  $k=\RR$) is the karoubian closure of the category with the same
  objects as in $\Hh$, and where Hom-spaces are tensored with $\CC$
  and then considered ``modulo Morita-equivalence'' $\sim_M$. Two
  examples: (1) $\DD_{\HH}$. At the boundary, $\HH$ becomes
  $\HH\otimes\CC=\matring_2(\CC)\sim_M\CC$. In the inner, $\CC$
  becomes $\CC\otimes\CC=\CC\times\CC$, two copies of $\CC$. We hence
  get two discs, all points complex, and the boundaries
  identified. This gives the Riemann sphere. Alternatively:
  $\HH[X,Y]\otimes\CC=\matring_2(\CC)[X,Y]\sim_M\CC[X,Y]$. (2)
  $\DD_{2,2}$. Here we get, modulo $\sim_M$ two copies of discs with
  all points complex, the boundaries identified; the two ramification
  points (their simples having endomorphism ring $\CC$) are
  ``doubled'', and become weighted by $2$. So we have the weighted
  projective line over $\CC$ with weight sequence $(2,2)$. The
  examples suggest that $\Hh$ is somehow a $\ZZ_2$-quotient of
  $\Hh\otimes\CC$ (justifying~\eqref{eq:chi-tensor-formula}), but this
  is yet not well understood. In general, if $\Hh$ has $2n$
  ramification points, and if $\Kk$ is the underlying Klein surface,
  $S$ the complex double, then $\Hh\otimes\CC$ is $S$ weighted with
  the $2n$-sequence $(2,2,\dots,2)$. Thus,
  $\chi(\Hh\otimes\CC)=\chi(S)-\sum_{i=1}^{2n}(1-1/2)=\chi(S)-n$. Moreover,
  if $\chi(\Hh):=\LF{L}{L}$, then
  by~\eqref{eq:general-euler-char-formula},
  $\chi(\Hh)=s(\Hh)^2 \bigl(\chi(\Kk)-1/2
  \sum_{i=1}^{2n}(1-1/2)\bigr)=s(\Hh)^2/2\cdot\chi(\Hh\otimes\CC)$.
  Therefore, if~\eqref{eq:chi-tensor-formula} should hold, we have to
  replace $\chi(\Hh)$ by the normalized $\chi'(\Hh)$.\medskip

  Because of these considerations we regard $\chi'$ as the correct
  Euler characteristic for a Witt curve.
\end{numb}

\section{Real elliptic curves}\label{sec:elliptic} 
In the following examples we treat all real elliptic curves, that is,
the Klein and Witt surfaces with $\genus=1$. Here, we classify them
only topologically, not up to isomorphism, where one has to add a real
parameter to each topological case.
\begin{example}[Commutative real elliptic
  curves]\label{ex:comm-real-elliptic} There are (up to 
  parameters) three real elliptic curves with $s(\Hh)=1$, with
  corresponding Klein surfaces given by the annulus $\AA$, the Klein
  bottle $\KK$ and the M\"obius band $\MM$. We refer to the
  book~\cite{alling:1981}.
\end{example}
\begin{example}[Elliptic Witt curves] As only real elliptic curves
  with $s(\Hh)>1$ we have (up to parameters) the corresponding Witt
  surfaces: the annuli $\AA_{\RR,\HH}$ and $\AA_{\HH,\HH}$, where one
  and both ovals, respectively, are coloured quaternion, the M\"obius
  band $\MM_{\HH}$, with quaternion coloured boundary, and the compact
  disc $\DD_{2,2,2,2}$ with four segmentation points on the boundary,
  Figure~\ref{fig:genus-one}; there is a moduli parameter $\lambda>0$
  involved, so the general case is not as symmetric as in the figure.
\begin{figure}[h]
\begin{tikzpicture}[line width=1.1pt]
\coordinate [label=left:$\HH$] (C) at (-1.12,0.75);
\coordinate [label=right:$\RR$] (D) at (1.12,0.75);
\coordinate [label=left:$\HH$] (C) at (1.2,-1.1);
\coordinate [label=right:$\RR$] (D) at (-1.5,-0.8);
\draw[black] (0,0) circle (1.2);
\fill[gray,opacity=.25] (0,0) circle (1.19);
\fill[red,opacity=1.0] (0,1.2) circle (2.0pt);
\fill[red,opacity=1.0] (0,-1.2) circle (2.0pt);
\fill[red,opacity=1.0] (-1.2,0) circle (2.0pt);
\fill[red,opacity=1.0] (1.2,0) circle (2.0pt);
\end{tikzpicture}
\quad
\begin{tikzpicture}[line width=1.1pt]
\coordinate [label=left:$\HH$] (C) at (0,0.2);
\coordinate [label=right:$\RR$] (D) at (0.9,0.9);
\draw[black] (0,0) circle (1.2);
\draw[black] (0,0) circle (0.6);
\filldraw[gray,opacity=.25,even odd rule] (0,0) circle (1.19) circle (0.61);
\end{tikzpicture}
\quad
\begin{tikzpicture}[line width=1.1pt]
\coordinate [label=left:$\HH$] (C) at (0,0.2);
\coordinate [label=right:$\HH$] (D) at (1.1,0.7);
\draw[black] (0,0) circle (1.2);
\draw[black] (0,0) circle (0.6);
\filldraw[gray,opacity=.25,even odd rule] (0,0) circle (1.19) circle (0.61);
\end{tikzpicture}
\quad
\begin{tikzpicture}[scale=0.8]
\pgfplotsset{
        colormap={whitegrey}{
            color(0cm)=(white);
            color(1cm)=(gray!50!black)
        }
    }
\coordinate [label=left:$\HH$] (C) at (5.5,1.2);
\begin{axis}[
    hide axis,
    view={45}{45}
]
\addplot3 [
    surf, 
    colormap name=whitegrey,
    point meta=x,
    samples=100,
    samples y=3,
    z buffer=sort,
    domain=0:360,
    y domain=-0.5:0.5
] (
    {(1+0.5*y*cos(x/2))*cos(x)},
    {(1+0.5*y*cos(x/2))*sin(x)},
    {0.5*y*sin(x/2)});
\end{axis}
\end{tikzpicture}
\caption{The elliptic Witt curves: $\DD_{2,2,2,2}$, $\AA_{\RR,\HH}$,
  $\AA_{\HH,\HH}$, $\MM_{\HH}$}
\label{fig:genus-one}
\end{figure}
\end{example}
\begin{lemma}
  For the real elliptic curves $(\Hh,L)$ Table~\ref{tab:elliptic-data}
  describes the endomorphism ring of the structure sheaf $L$, the
  endomorphism ring of a certain simple sheaf $S=S_x$ of degree $1$,
  the multiplicity $e(x)$, the number $\varepsilon$, and the number of
  orbits in $\widehat{\QQ}$ of the action of $\Aut(\bDerived{\Hh})$ on
  the slopes.
\begin{table}[h!]
    \begin{center}
     $$ \begin{array}{c|ccc|cccc}
        & \AA & \MM & \KK & \AA_{\RR,\HH} & \AA_{\HH,\HH} & \MM_{\HH} & \DD_{2,2,2,2}\\
       \hline
       \End(L) & \RR & \RR & \RR & \CC & \HH & \HH & \CC\\
       \End(S) & \RR & \RR & \CC & \RR & \HH & \HH & \CC\\
        e(x) & 1 & 1 & 1 & 2 & 1 & 1 & 1\\ 
       \varepsilon & 1 & 1 & 2 & 1 & 1  & 1 & 1\\  
       \# \text{orbits} & 1 & 1 & 2 & 2 & 1  & 1 & 1  
     \end{array}
     $$ \caption{Data of real elliptic curves}
    \label{tab:elliptic-data}
    \end{center}
\end{table}
For $\KK$ the two orbits are given by those fractions $a/b$ (with
$a,\,b$ coprime) with $b$ even or odd, respectively, in case
$\AA_{\RR,\HH}$ with $a$ even or odd, respectively; in both cases the
slopes $0$ and $\infty$ belong to different orbits.
\end{lemma}
\begin{proof}
  We get $\End(L)$ and the value of $\varepsilon$ from
  Lemma~\ref{lem:constants-field}. For $\KK$ we have $\varepsilon=2$
  since there is no boundary. There exists a simple $S$ as
  claimed. The objects $L$ and $S$ define tubular mutations
  $\sigma_L,\,\sigma_S\colon\bDerived{\Hh}\ra\bDerived{\Hh}$,
  respectively, which on $\Knull(\Hh)$ act as follows (with
  $\gge{a}:=[L]$, $\gge{s}:=[S]$ and $\kappa=\dim_{\RR}\End(L)$):
  \begin{equation}
    \label{eq:mutations}
    \sigma_L(\gge{y})=\gge{y}\pm\frac{\LF{\gge{a}}{\gge{y}}}{\kappa}\gge{a},
    \quad\sigma_S(\gge{y})=\gge{y}\pm\frac{\LF{\gge{s}}{\gge{y}}}{|\End(S)|}\gge{s}. 
  \end{equation}
  With
  $\deg(\gge{y})=\frac{1}{\kappa\varepsilon}\LF{\gge{a}}{\gge{y}}$ and
  $\rk(\gge{y})=\frac{1}{\kappa\varepsilon\deg(S)}\LF{\gge{y}}{\gge{s}}$
  we obtain the induced actions on the slopes
  $$q\mapsto \frac{q}{1\pm\varepsilon q}\quad\text{and}\quad q\mapsto
  q\pm\deg(S)e(x)=q\pm e(x)\quad\text{(resp.)}.$$ The claimed numbers
  and shapes of orbits follow from~\cite[Lem.~6.1]{kussin:2000b}.
\end{proof}
It was already observed in~\cite{lenzing:2001} that the Klein bottle
must have a Fourier-Mukai partner different from a Klein bottle. The
first part of the following statement is a non-weighted analogue
of~\cite{kussin:2000}.
\begin{theorem}
  \begin{enumerate}
  \item[(1)] The Klein bottle $\KK$ (with any parameter) has as a
    Fourier-Mukai partner a Witt curve given by the annulus
    $\AA_{\RR,\HH}$ with two differently coloured ovals (with a
    suitable parameter).
  \item[(2)] If $\Hh$ is a noncommutative real elliptic curve, which is
    neither a Klein bottle, nor an annulus $\AA_{\RR,\HH}$, then each
    Fourier-Mukai partner of $\Hh$ is isomorphic to $\Hh$ itself.
  \end{enumerate}
\end{theorem}
\begin{proof}
  (1) A Klein bottle $\KK$ has no boundary. Thus all simple sheaves
  have endomorphism ring $\CC$, the complex numbers. On the other
  hand, the structure sheaf $\Oo_{\KK}$ (which is stable and of slope
  $0$) has endomorphism ring $\RR$. Thus by Theorem~\ref{thm:elliptic}
  the subcategory of semistable bundles of slope $0$ is parametrized
  by a noncommutative projective curve $\Hh$ of $g(\Hh)=1$ with
  $\Hh\not\simeq\KK$, and $\Hh$ is derived-equivalent to $\KK$. There
  must be a simple sheaf $S$ in $\Hh$ with $\End(S)\simeq\RR$. By
  Table~\ref{tab:elliptic-data} the only possibility is then
  $\Hh=\AA_{\RR,\HH}$ (with a suitable parameter).

  (2) This follows from Table~\ref{tab:elliptic-data}.
\end{proof}
\begin{corollary}
  The skewness, and thus the function field, of a noncommutative
  regular projective curve is in general \emph{not} a derived
  invariant. \qed
\end{corollary}
\begin{proposition}
  Let $\Hh$ be an elliptic Witt curve. 
  \begin{enumerate}
  \item[(1)] If $\Hh$ is unramified, then $\tau=1_{\Hh}$. Moreover,
    $\Pic_0(\Hh)$ is not finitely generated.
  \item[(2)] Otherwise, that is, if $\Hh$ is given by $\DD_{2,2,2,2}$
    (with some parameter), then
    $\tau=({\sigma_{x_1}}^{ }\circ{\sigma_{y_1}}^{-1})\circ
    ({\sigma_{x_2}}^{ }\circ{\sigma_{y_2}}^{-1})$
    (where $x_1,\,y_1,\,x_2,\,y_2$ are the four ramification points)
    is of order $2$. Moreover, in this case the Picard-shift group
    $\Pic(\Hh)$ is finitely generated abelian of rank one, and
    $\Pic_0(\Hh)$ is finite.
  \end{enumerate}
\end{proposition}
\begin{proof}
  In all cases, the result on the order of $\tau$ is given by
  Theorem~\ref{thm:elliptic-tau-finite-order}. For the Picard-shift
  group we get in the Azumaya cases $\Pic(\Hh)\simeq\Pic(X)$ by
  Theorem~\ref{thm:Picard-sequence}, with $X$ the underlying Klein
  curve, which is elliptic. It is well-known that $\Pic_0(X)$ is not
  finitely generated in this case (we refer
  to~\cite[Thm.~5.7]{alling:1974}). In the ramified case
  $\DD_{2,2,2,2}$ the centre curve is $X=\Pone(\RR)$, of genus zero,
  and $\Pic(\Pone(\RR))\simeq\ZZ$, and the last claim follows from
  Theorem~\ref{thm:Picard-sequence}. Actually,
  $\Pic_0(\Hh)\simeq\ZZ_2\times\ZZ_2\times\ZZ_2$, generated by
  ${\sigma_{x_1}}^{ }\circ{\sigma_{y_1}}^{-1}$,
  ${\sigma_{y_1}}^{ }\circ{\sigma_{x_2}}^{-1}$ and
  ${\sigma_{x_2}}^{ }\circ{\sigma_{y_2}}^{-1}$.
\end{proof}
\begin{remark}[Calabi-Yau]
  Let $\Tt$ be a triangulated $k$-category with finite dimensional
  Hom-spaces and with Serre duality $\Hom_\Tt(X,Y)=\D\Hom_\Tt(Y,SX)$,
  where $S$ is an exact autoequivalence of $\Tt$, the (triangulated)
  \emph{Serre functor}. If $S^m \simeq [n]$, the $n$-th suspension
  functor (with $m\geq 1$ minimal), then $\Tt$ is called
  \emph{triangulated Calabi-Yau of (fractional) dimension}
  $\frac{n}{m}$ (we refer to~\cite{keller:2005}). If
  $\Tt=\bDerived{\Hh}$, with $\Hh$ a noncommutative regular projective
  curve with Auslander-Reiten translation $\tau$, then
  $S=\tau\circ [1]$ is the Serre functor of $\Tt$. If $\genus=1$, then
  the functor $\tau$ is of finite order $p$, and then $S^p\simeq [p]$,
  that is, $\Tt$ is Calabi-Yau of dimension $\frac{p}{p}$.

  The preceding discussion shows that the derived category of the
  elliptic Witt curve $\DD_{2,2,2,2}$ has Calabi-Yau dimension
  $\frac{2}{2}$; all the others have dimension $\frac{1}{1}$.
\end{remark}
\begin{proposition}
  Let $(\Hh,L)=\coh(\XX)$ be the Witt curve $\XX=\DD_{2,2,2,2}$ for
  some parameter. The stable bundles of degree $0$ are parametrized by
  $\XX$. The line bundles of degree $0$ are in bijection with the
  non-quaternion boundary points; the structure sheaf $L$ corresponds
  to one of the ramification points. 
\end{proposition}
\begin{proof}
  The composition of tubular mutations $\sigma=\sigma_S\circ\sigma_L$
acts on slopes $\mu=
  \begin{pmatrix}
    \deg\\
    \rk
  \end{pmatrix}$ like the matrix $
  \begin{pmatrix}
    0 & -1\\
    1 & -1
  \end{pmatrix}$,
  sends $\infty$ to $0$ and preserves endomorphism rings. If
  $E\in\Hh_0$, then $\rk(\sigma(E))=\deg(E)$. Thus the degree $0$ line
  bundles correspond to the simples sheaves of degree $1$. Moreover, $L$
  is one of them with $\End(L)=\CC$. 
\end{proof}

\section{Weighted curves, and noncommutative $2$-orbifolds} 
\label{sec:weighted}
We conclude the article by treating the weighted case. We will first
show that each weighted noncommutative regular projective curve $\Hh$
arises from a non-weighted one, $\Hh_{nw}$, by insertion of weights
$p(x)>1$ in a finite number of points $x$. This insertion of weights
is described, in abstract terms, by the $p$-cycle
construction~\cite{lenzing:1998}; the inverse technique, reducing
weights, is the perpendicular calculus~\cite{geigle:lenzing:1991},
cf.\ Proposition~\ref{prop:reduction}. Data like the function field
$k(\Hh)$, the skewness $s(\Hh)$, the multiplicities $e(x)$, the
endomorphism rings $\End(S_x)$ of simples, and the ghost group
$\Gg(\Hh)$ remain unchanged by these processes. The change of the
Picard-shift group $\Pic(\Hh)$ is easy to describe: one adjoints the
$p(x)$-th root of the tubular shift $\sigma_x$, for each weight point
$x$. We refer to~\cite[Ch.~6]{kussin:2009}. We also remark that there
is an analogues result to Lemma~\ref{lem:isomorphic-point-functors}.
\begin{numb}[$p$-cycle construction]
  Let $\Hh=\coh(\Aa)$ be a weighted noncommutative regular projective
  curve of a field $k$, with $\Aa$ a hereditary order in a full matrix
  algebra $A$ over $k(\Hh)$. Let $x$ be a point such that the tube
  $\Uu_x$ is homogeneous. Let $\sigma_x\colon\Hh\ra\Hh$ be the
  Picard-shift with respect to $x$, with natural transformation
  $1_{\Hh}\stackrel{x}\lra\sigma_x$. Let $p>1$ be a
  ``weight''. Following~\cite{lenzing:1998} the category $\Hh
  \begin{pmatrix}
    p\\
    x
  \end{pmatrix}$ of $p$-\emph{cycles in} $x$ has \emph{objects}
  $E=(E_i,e_i)_{i\in\ZZ}$, where $E_i\in\Hh$ and
  $e_i\in\Hom(E_i,E_{i+1})$ such that $E_{i+p}=\sigma_x(E_i)$ and the
  composition $e_{i+p-1}\circ\ldots\circ e_{i+1}\circ e_i$ is the
  natural map $x_{E_i}\colon E_i\ra E_i(x)$ for each $i$. If
  $F=(F_i,f_i)_{i\in\ZZ}$ is another $p$-cycle in $x$, then a
  \emph{morphism} between $E$ and $F$ is a tuple $(h_i)_{i\in\ZZ}$
  with morphisms $h_i\in\Hom(E_i,F_i)$ satisfying $h_{i+1}e_i=f_ih_i$
  and $h_{i+p}=\sigma_x(h_i)$ for all $i$. The category of $p$-cycles
  in $x$ over $\Hh$ is, like $\Hh$ itself, abelian, noetherian,
  hereditary and does not contain non-zero projectives or injectives,
  \cite[Thm.~4.3]{lenzing:1998}.\medskip
  
  Let $\Aa(p,x)$ be the (hereditary) order in $\matring_p(A)$
  given by
  $$\begin{pmatrix}
    \Aa & \Aa & \dots & \Aa\\
     \Aa(-x)  & \Aa & \dots & \Aa\\
      \vdots & & \ddots & \vdots\\
     \Aa(-x) & \Aa(-x) & \dots & \Aa
   \end{pmatrix}.$$ There is an
   equivalence~\cite[Prop.~6.1]{lerner:oppermann:2015}
   $$\coh\bigl(\Aa(p,x)\bigr)\simeq\Hh
  \begin{pmatrix}
    p\\
    x
  \end{pmatrix},$$ and we have then $p(x)=p$. This can be iterated:
\end{numb}
\begin{proposition}\label{prop:weight-insertion}
  \begin{enumerate}
  \item[(1)] Let $\Hh_{nw}=\coh(\Aa)$ be a (non-weighted) noncommutative
  regular projective curve over $k$. Let $x_1,\dots,x_t$ be distinct
  points, and let $p_1,\dots,p_t>1$. Then
  $$\Hh=\coh\bigl(\bigotimes_{i=1}^t\Aa(p_i,x_i)\bigr)\simeq\Hh_{nw}
  \begin{pmatrix}
    p_1,\dots,p_t\\
    x_1,\dots,x_t
  \end{pmatrix}$$ is a weighted noncommutative regular projective
  curve with weight points $x_1,\dots,x_t$, having weights
  $p(x_i)=p_i$. 
\item[(2)] Each weighted noncommutative regular projective curve $\Hh$
  is obtained in this way from its underlying non-weighted curve
  $\Hh_{nw}$ (cf.\ Proposition~\ref{prop:reduction}).
  \end{enumerate}
\end{proposition}
\begin{proof}
  (1) As in~\cite[Prop.~6.1+6.5]{lerner:oppermann:2015}. (2) Clearly
  $\Hh$ and $\Hh_{nw}
  \begin{pmatrix}
    p_1,\dots,p_t\\
    x_1,\dots,x_t
  \end{pmatrix}$ have the same underlying non-weighted curve, namely
  $\Hh_{nw}$. In particular they have the same centre curve and the
  same function field, and also the same weight function $p\colon
  X\ra\NN$. The statement then follows
  from~\cite[Thm.~6.7]{burban:drozd:gavran:2015}. 
\end{proof}
Concerning the $\tau$-multiplicities the following observation is
fundamental.
\begin{proposition}\label{prop:tau-multiplicity-weighted}
  Let $\Hh$ be a noncommutative regular projective curve over a field
  $k$ and $x$ a separable point. Let $p>1$ be a weight and let
  $\overline{\Hh}$ be a weighted noncommutative curve arising from
  $\Hh$ by insertion of the weight $p$ into $x$. Let
  $\overline{\Uu}_x$ be the corresponding tube of rank $p$. Then the
  order of $\tau_{\overline{\Hh}}$ in $\Aut(\overline{\Uu}_x/k)$ is
  $e_{\tau}(x)\cdot p$.
\end{proposition}
\begin{proof}
  Working with $p$-cycles in $x$ one sees easily that
  $(\tau_{\overline{\Hh}})^p$ acts on $\End(S_x)$ like $\sigma_x^{-1}$
  and hence like $\tau_{\Hh}$.
\end{proof}
\begin{proposition}[The complete local rings]
  Let $x$ be a point of weight $p=p(x)$. Then
  $\Uu_x\simeq\mod_0 (H_{p}(\widehat{R}_x))$, with $\widehat{R}_x$ the
  complete local ring as in
  Proposition~\ref{prop:complete-local-rings-general-homogeneous} (and
  for separable $x$ the skew power series ring as in
  Theorem~\ref{thm:pruefer-end-skew-power}), and $$H_p(\widehat{R}_x)=
  \begin{pmatrix}
    \widehat{R}_x & \widehat{R}_x & \dots & \widehat{R}_x\\
     P_x  & \widehat{R}_x & \dots & \widehat{R}_x\\
      \vdots & & \ddots & \vdots\\
     P_x & P_x & \dots & \widehat{R}_x
   \end{pmatrix}$$
   of size $p\times p$ and with $P_x=\Rad(\widehat{R}_x)$, generated
   by $\pi_x$. The ring $H_p(\widehat{R}_x)$ is a semiperfect bounded
   hereditary noetherian prime ring, whose radical is generated as
   left and right ideal by the element $$\overline{\pi}_x=
  \begin{pmatrix}
    0 & 1 & 0 & \dots & 0\\
     0  & 0 & 1 & \dots & 0\\
      \vdots & & & \ddots & \vdots\\
      0 &  & & & 1\\
     \pi_x & 0 & \dots & & 0
   \end{pmatrix},$$ which satisfies ${\overline{\pi}_x}^p=\pi_x$.
\end{proposition}
\begin{proof}
  This follows from~\cite[4.4]{ringel:1979}.
\end{proof}
Let $\Aa$ be a hereditary $\Oo$-order with $\Hh=\coh(\Aa)$, by
Theorem~\ref{thm:structure}. By $\ovp$ we will always denote the least
common multiple of the weights. We can assume that $\Aa$ is a
hereditary $\Oo$-order in
$\matring_{\ovp}(k(\Hh))$. (By~\cite[Rem.~6.8]{burban:drozd:gavran:2015}
even the maximum of the weights can be chosen as the matrix size.)
Moreover, we can and will always assume that the structure sheaf $L$
is a special line bundle, corresponding to the structure
sheaf~\eqref{eq:structure-sheaf} of $\Hh_{nw}$ via
Proposition~\ref{prop:reduction}.  For a point $x$ we denote by $S_x$
the (up to isomorphism) unique simple sheaf concentrated in $x$ with
$\Hom(L,S_x)\neq 0$. For a point $x\in\XX$ let $\sigma_x(L)=L(x)$, and
then the bimodule $\Aa(x)$ is defined as in the unweighted case. One
also shows that the functors $\sigma_x$ and $-\!\otimes_{\Aa}\!\Aa(x)$
are isomorphic. From the preceding results we get
$\mathfrak{m}_xH_{p(x)}(\widehat{R}_x)={(\overline{\pi}_x)}^{p(x)e_{\ramif}(x)}$
and
\begin{equation}
  \label{eq:compare-twists-Oo-Aa}
  \Oo(x)\!\otimes_{\Oo}\Aa\simeq\Aa(p(x)e_{\ramif}(x)\cdot x).
\end{equation}
With this one obtains the more general, weighted version of
Theorem~\ref{thm:Picard-sequence}. We note that there is a formal
similarity to~\cite[(40.9)]{reiner:2003}.
\begin{theorem}\label{thm:Picard-sequence-weighted}
  Let $\Hh$ be a weighted noncommutative regular projective curve over
  a field $k$. Let $X$ be the (non-weighted) centre curve. Then there
  is an exact sequence
  \begin{equation}
    \label{eq:Picard-sequence-weighted}
    1\ra\Pic(X)\stackrel{\iota}\lra\Pic(\Hh)\stackrel{\phi}
    \lra\prod_{x}\ZZ/p(x)e_{\ramif}(x)\ZZ\ra 1 
 \end{equation}
 of abelian groups. Here, $\phi(\sigma)=\sigma_{|\Hh_0}$ and $\iota$
 sends a (class of a) line bundle $\Oo(x)$ of $\coh(X)$, for a point
 $x\in X$, to ${\sigma_x}^{p(x)e_{\ramif}(x)}$, for the corresponding
 point $x\in\XX$.\qed
\end{theorem}
\begin{theorem}
  Let $(\Hh,L)$ be a weighted noncommutative regular projective curve
  over a perfect field $k$. Let $\gamma=\sum_{x\in X}\gamma_x\cdot x$
  be the canonical divisor of the (non-weighted) centre curve $X$. For
  $\overline{\gamma}=\sum_{x\in\XX}\gamma_xp(x)e_{\tau}(x)\cdot x$ we
  write $\sigma^{|\overline{\gamma}|}$ for the corresponding
  Picard-shift. Then
  \begin{equation}
    \label{eq:tau-picard-general}
    \tau\ =\ \sigma^{|\overline{\gamma}|}\cdot\prod_{x}{\sigma_x}^{p(x)e_{\ramif}(x)-1}\ =\ 
    \prod_{x}{\sigma_x}^{p(x)e_{\ramif}(x)(\gamma_x+1)-1}\quad\in\Pic(\Hh).
  \end{equation}
\end{theorem}
\begin{proof}
  We just remark that here the different is given by
  $\Delta=\sum_x (p(x)e_{\ramif}(x)-1)\cdot x$, which can be seen as
  in the unweighted case, and that the dualizing sheaf $\bom_{\Aa}$
  also here is given by $\ShHom_{\Oo}(\Aa,\bom_X)$,
  see~\cite[III.2]{reiten:vandenbergh:2002},
  and~\eqref{eq:dualizing-sheaf-isos} also holds here.
\end{proof}
Let $(\Hh,L)$ be a weighted noncommutative regular projective curve
over $k$ of skewness $s=s(\Hh)$. Let $\kappa=\dim_k\End(L)$ and
$\varepsilon$ as defined before (for the underlying non-weighted curve
$\Hh_{nw}$). We define the \emph{average Euler form} and the
\emph{(orbifold) degree}
$$\DLF{E}{F}=\sum_{j=0}^{\ovp-1}\LF{\tau^j E}{F},\quad
\deg(F)=\frac{1}{\kappa\varepsilon}\cdot\DLF{L}{F}
-\frac{1}{\kappa\varepsilon}\cdot\DLF{L}{L}\cdot\rk(F)$$ and the
\emph{normalized orbifold Euler characteristic} $\chi'_{orb}(\Hh)$ and
the \emph{orbifold genus} $g_{orb}(\Hh)$ by the equations
$$\chi'_{orb}(\Hh):=\frac{1}{\ovp^2s^2}\cdot\DLF{L}{L}\stackrel{(\ast)}=
-\frac{\kappa\varepsilon}{2\ovp s^2}\cdot\deg(\tau
L)=:\frac{\kappa}{\ovp s^2}\cdot(1-g_{orb}(\Hh)).$$ The equality $(\ast)$
is given by the following.
\begin{lemma} 
  We have $\tau L=L\otimes_{\Aa}\bom_{\Aa}$ and
  $\deg\tau L=-\frac{2}{\ovp\kappa\varepsilon}\DLF{L}{L}$.
\end{lemma}
\begin{proof}
  The first equality is clear. For the second we show like
  in~\cite[Prop.~3.2]{lenzing:reiten:2006} that the difference
  $\deg\tau^nL-\deg\tau^{n-1}L$ does not depend on $n$, and as
  in~\cite[Lem.~9.1]{lenzing:1996} the claim then follows.
\end{proof}
Like in the unweighted case we obtain
(compare~\cite[Thm.~9.2]{lenzing:1996}):
\begin{theorem}[Riemann-Roch formula]
  Let $\Hh$ be a weighted noncommutative regular projective curve over
  the field $k$. Then
  $$\frac{1}{\kappa\ovp}\cdot\DLF{E}{F}=(1-g_{orb}(\Hh))\cdot\rk(E)\cdot\rk(F)+
  \frac{\varepsilon}{\ovp}\cdot\begin{vmatrix}
    \rk(E) & \rk(F)\\
    \deg(E) & \deg(F)
\end{vmatrix}$$ holds for all $E,\,F\in\Hh$. \qed
\end{theorem}
\begin{lemma}\label{lem:deg-S-general}
  Let $(\Hh,L)$ be a weighted noncommutative regular projective curve
  over the perfect field $k$. For each $x\in\XX$ we have
  \begin{equation*}
    \deg(S_x)=\frac{\ovp\cdot s(\Hh)}{p(x)\cdot\kappa\cdot\varepsilon}
  \cdot e^{\ast}(x)\cdot [k(x):k].
  \end{equation*}
\end{lemma}
\begin{proof}
  In general ($\Hh$ regular over any field) we have (with $e'(x)$ as
  in~\eqref{eq:e-prime})
  \begin{equation}
    \label{eq:generalissimo-deg-simple}
    \deg(S_x)=\frac{\ovp}{p(x)\kappa\varepsilon}\dim_k
    \Hom(L,S_x)=\frac{\ovp}{p(x)\kappa\varepsilon}e(x)e^{\ast}(x)^2e'(x)[k(x):k].
  \end{equation}
  In the perfect case we know $e'(x)=e_{\tau}(x)$ and
  $e(x)e^{\ast}(x)e_{\tau}(x)=s(\Hh)$. 
\end{proof}
\begin{theorem}[Noncommutative Riemann-Hurwitz formulae]
\label{thm:orbifold-euler-char}
Let $\Hh$ be a weight\-ed noncommutative regular projective curve over
the perfect field $k$. Let $X$ be the centre curve, $\Hh_{nw}$ the
underlying non-weighted curve. For the normalized orbifold Euler
characteristic $\chi'_{orb}(\Hh)$ we have
 \begin{eqnarray}
    \label{eq:orbifold-euler-char-formula}
    \chi'_{orb}(\Hh) & = & \chi'(X)-\frac{1}{2}
    \sum_{x}\Bigl(1-\frac{1}{p(x)e_{\tau}(x)}\Bigr)[k(x):k]\\
    \label{eq:orbifold-euler-char-formula-weights}
      & = & \chi'(\Hh_{nw})-\frac{1}{2}
      \sum_{x}\frac{1}{e_{\tau}(x)}\Bigl(1-\frac{1}{p(x)}\Bigr)[k(x):k].
  \end{eqnarray}
  If we assume $k$ to be the centre of $\Hh$, then $\chi'_{orb}(\Hh)$
  is an invariant of $\Hh$.
\end{theorem}
\begin{proof}
  We set $s=s(\Hh)$ and denote by $\gge{a}=[L]$ and $\gge{s_x}=[S_x]$
  the classes in the Grothendieck group. We denote the Coxeter
  transformation on $\Knull(\Hh)$ also by $\tau$. We further set
  $\gge{w_x}=\sum_{j=0}^{p(x)-1}\tau^j\gge{s_x}$. By~\eqref{eq:tau-picard-general}
  we have
  $$\tau\gge{a}=\gge{a}+\sum_x \gamma_x
  p(x)e_{\tau}(x)e(x)\gge{s_x}+\sum_x
  (p(x)e_{\tau}(x)-1)e(x)\gge{s_x}.$$
  Then 
  $$\tau^{\ovp}\gge{a}-\gge{a} = \sum_x \frac{\ovp}{p(x)} \gamma_x
     p(x)e_{\tau}(x)e(x)\gge{w_x}
  +\sum_x\frac{\ovp}{p(x)}(p(x)e_{\tau}(x)-1)e(x)\gge{w_x}.$$
  By Riemann-Roch and the preceding lemma,
  $$\LF{\gge{a}}{\gge{w_x}}=\LF{\gge{a}}{\gge{s_x}}
  =\frac{p(x)}{\ovp}\DLF{\gge{a}}{\gge{s_x}}
  =\frac{p(x)\kappa\varepsilon}{\ovp}\deg\gge{s_x}=se^{\ast}(x)[k(x):k].$$
  We obtain
  $$\LF{\gge{a}}{\tau^{\ovp}\gge{a}-\gge{a}}=s^2\ovp\sum_x \gamma_x
  [k(x):k]+s^2\ovp\sum_x
  \bigl(1-\frac{1}{p(x)e_{\tau}(x)}\bigr)[k(x):k].$$
  Similarly we compute
  $\LF{\tau^j\gge{a}}{\tau^{\ovp}\gge{a}-\gge{a}}$. With this we get,
  as in~\cite[Lem.~9.1]{lenzing:1996},
  $$\DLF{\gge{a}}{\gge{a}}=-\frac{1}{2}\ovp^2s^2
  \sum_x\gamma_x[k(x):k]-\frac{1}{2}\ovp^2s^2
  \sum_x\bigl(1-\frac{1}{p(x)e_{\tau}(x)}\bigr)[k(x):k].$$
  Since
  $\sum_x\gamma_x[k(x):k]=[\End(\Oo):k]\varepsilon_X\deg_X(\bom_X)=-2\chi'(X)$
  we obtain
  $$\DLF{\gge{a}}{\gge{a}}=\ovp^2s^2\chi'(X)-\frac{1}{2}\ovp^2s^2
  \sum_x\bigl(1-\frac{1}{p(x)e_{\tau}(x)}\bigr)[k(x):k].$$
  Division by $\ovp^2s^2$ yields the first equation.\medskip

  Then, the second follows with~\eqref{eq:general-euler-char-formula},
  \eqref{eq:deg-of-simples} and using the equation
  $1-\frac{1}{pe}=(1-\frac{1}{e})+\frac{1}{e}(1-\frac{1}{p})$.
\end{proof}
\begin{remark}\label{rem:normalizing-euler-char}
  (1) Let $k$ be algebraically closed.  As in the unweighted cases,
  the orbifold Euler characteristic $\chi_{orb}(\Hh)$ satisfies
  $\chi_{orb}(\Hh)=2\chi'_{orb}(\Hh)$. Since moreover $s(\Hh)=1$ and
  $\Hh_{nw}=\coh(X)$, equations~\eqref{eq:orbifold-euler-char-formula}
  and~\eqref{eq:orbifold-euler-char-formula-weights} yield
  $$\chi_{orb}(\Hh)=\chi(X)-\sum_x\Bigl(1-\frac{1}{p(x)}\Bigr).$$
  In case $k=\CC$ and $\Hh$ a weighted complex regular projective
  curve, or a complex $2$-orbifold, then
  $\chi_{orb}(\Hh)=\chi'_{orb}(\Hh)$, if the values are computed over
  the field $\RR$ of real numbers.\medskip

  (2) The factor $[k(x):k]$ (a datum of the centre curve) equals
  $\frac{\kappa\cdot\varepsilon\cdot\deg(S_x)}{e^{\ast}(x)\cdot
    s(\Hh)}$,
  with the degree of $S_x$ in $\Hh_{nw}$ (a datum of the underlying
  non-weighted curve).\medskip

  (3) In the preceding theorem, we made the assumption that $k$ is
  perfect since we used the skewness equation,
  Theorem~\ref{thm:special-skewness-equation}. Of course, we only need
  that the involved points are separable. But also in full generality
  we may have a ``nice''
  formula. In~\eqref{eq:special-orbifold-euler-any-field} below we
  have a still compact formula, in a special case, but over any field.
\end{remark}
\subsection*{Weighted Klein and Witt curves}
The following formula, obtained
from~\eqref{eq:orbifold-euler-char-formula-weights}, can be regarded
as the extension of the Riemann-Hurwitz
formula~\cite[Thm.~13.3.4]{thurston:2002} in Thurston's book (we also
refer to~\cite{scott:1983}) to \emph{noncommutative real
  $2$-orbifolds}.
\begin{corollary}
  Let $\Hh$ be a weighted noncommutative regular projective curve over
  $k=\RR$. Then
    $$\begin{array}{l}
        \chi'_{orb}(\Hh)=\chi'(\Hh_{nw})-
        \frac{1}{4}\cdot\sum_{x}\bigl(1-\frac{1}{p(x)}\bigr)
        -\frac{1}{2}\cdot\sum_{y}
        \bigl(1-\frac{1}{p(y)}\bigr)-\sum_{z}
        \bigl(1-\frac{1}{p(z)}\bigr),
  \end{array}$$
  where $x$ runs over the ramification points, $y$ over the other
  boundary points, and $z$ over the inner points.\qed
\end{corollary}
We remark that in case $s(\Hh)=1$ we have
$\chi'(\Hh_{nw})=\chi_{top}(S)$, where $S$ is the underlying Klein or
Riemann surface.
\subsection*{Multiplicity freeness and line bundles}
The following fact on line bundles may be of general interest. In case
of genus zero it was first shown in~\cite{kussin:1997},
compare~\cite[Prop.~7.3.5]{chen:krause:2009}.
\begin{proposition}\label{prop:multiplicity-free-picard-shifts-transitive}
  Let $(\Hh,L)$
  be a weighted noncommutative regular projective curve over a field
  $k$.
  If $\Hh$
  is multiplicity free, then each line bundle is a Picard-shift of the
  structure sheaf $L$.
  That is, $\Pic(\Hh)$
  acts transitively on the set of isomorphism classes of line bundles.
\end{proposition}
\begin{proof}
  Let $L'$ be a line bundle. Let $x\in\XX$ be any point. For $n\gg 0$
  we have a short exact sequence $0\ra L\ra L'(nx)\ra E\ra 0$ with $E$
  of finite length; this follows by a weighted version of
  Lemma~\ref{lem:morphisms-line-bundles-large-n}. Applying (a weighted
  version of) Lemma~\ref{lem:cokernel-pi-n} (with $e=1$) to each
  indecomposable summand of $E$ gives the result.
\end{proof}
The statement is not true in general, if $\Hh$ has
multiplicities. In~\cite{kussin:2009} examples of genus zero (even
non-weighted) are given, where there are two line bundles having
non-isomorphic endomorphism rings.
\subsection*{Tilting objects}
The following well-known (see~\cite{lenzing:delapena:1999}) fact shows
why genus zero (in the non-orbifold sense) curves are important in the
representation theory of finite dimensional algebras.
\begin{theorem}
  Let $\Hh$ be a weighted noncommutative regular projective over the
  field $k$, with underlying non-weighted curve $\Hh_{nw}$. The
  following are equivalent:
  \begin{enumerate}
  \item[(1)] $g(\Hh_{nw})=0$.
  \item[(2)] $\Hh$ admits a tilting object $T$. 
  \end{enumerate}
  If this is the case, then there is even a tilting bundle $T$ in
  $\Hh$ such that its endomorphism ring is a canonical algebra in the
  sense of
  Ringel-Crawley-Boevey~\cite{ringel:crawley-boevey:1990}. \qed
\end{theorem}
These curves were called exceptional curves
in~\cite{lenzing:1998,kussin:2009}, and noncommutative curves of genus
zero in~\cite{kussin:2009}. Special cases are the weighted projective
lines introduced by Geigle-Lenzing~\cite{geigle:lenzing:1987}. If $k$
is algebraically closed, these notions coincide.\medskip

If $\Hh_{nw}$ is of genus zero, then so is the centre curve. 
\begin{corollary}[of~\eqref{eq:tau-picard-general}]
  \label{cor:genus-zero-centre-tau-picard-weighted}
  Let $k$ be a perfect field. We assume that the centre curve $X$ is
  of genus zero, of numerical type $\varepsilon$. Then
  \begin{equation}
    \label{eq:genus-zero-centre-tau-picard-weighted}
  \tau\ =\ {\sigma_{x_0}}^{-2/\varepsilon}\cdot\prod_{x}{\sigma_x}^{p(x)e_{\tau}(x)-1}  
  \end{equation}
  for any point $x_0\in\XX$ which is rational in $X$ and neither
  ramification nor weight. \qed
\end{corollary}
In the following special case, the formulae for orbifold Euler
characteristic and genus are the well-known ones. These can be simply
obtained as special cases
from~\eqref{eq:orbifold-euler-char-formula-weights}, or they can be
proved directly (compare~\cite[Thm.~9.2]{lenzing:1996}), even over any
field.
\begin{corollary}
  Let $\Hh$ be a weighted noncommutative regular projective curve over
  a field $k$.  Assume that the non-weighted curve $\Hh_{nw}$ is of
  genus zero. Then
  \begin{equation}
   \label{eq:special-orbifold-euler-any-field}
   \chi'_{orb}(\Hh)=\frac{\kappa}{s(\Hh)^2}-\frac{\kappa\varepsilon}{2s(\Hh)^2}\sum_x 
   e(x)f(x)\Bigl(1-\frac{1}{p(x)}\Bigr) 
  \end{equation}
  and
  \begin{equation*}
    g_{orb}(\Hh)=1+\frac{\varepsilon\ovp}{2}\biggl(\sum_x
    e(x)f(x)\Bigl(1-\frac{1}{p(x)}\Bigr)-\frac{2}{\varepsilon}\biggr). 
  \end{equation*}
\end{corollary}
\begin{proof}
  In case $x$ is separable, then (invoking~\eqref{eq:char-genus}
  and~\eqref{eq:index-f-x})
  \begin{equation}
    \label{eq:e-f-formula}
    e(x)f(x)=\frac{s(\Hh)^2}{\kappa\varepsilon}\frac{1}{e_{\tau}(x)}\cdot[k(x):k],
  \end{equation}
  and the result follows
  from~\eqref{eq:orbifold-euler-char-formula-weights}.
\end{proof}
\subsection*{Negative orbifold Euler characteristic}
The orbifold Euler characteristic is strongly linked to the Gorenstein
parameter in singularity theory; we refer
to~\cite{kussin:lenzing:meltzer:2013}. In that more general context,
the case of negative orbifold Euler characteristic is also called the
\emph{anti-Fano} case, or \emph{of general type}. This situation is
the most complicated, in terms of complexity of the category $\Hh$.
\begin{proposition}[{\cite[Prop.~4.7]{lenzing:reiten:2006}}]
  Let $\Hh$ be a weighted noncommutative regular projective curve over
  a field. If $\chi'_{orb}(\Hh)<0$, then each Auslander-Reiten component in
  $\Hh_+=\vect(\XX)$ is of type $\ZZ A_{\infty}$, and $\Hh$ is of wild
  representation type. \qed
\end{proposition}
The weighted noncommutative regular projective curves of nonnegative
(orbifold) Euler characteristic are the (non-weighted) elliptic
curves, the domestic and the tubular curves, which we will consider now.
\subsection*{Domestic curves} This case is also called the \emph{Fano} case.
  \begin{definition}
  We call a weighted noncommutative regular projective curve
  \emph{domestic}, if $\chi'_{orb}(\Hh)>0$ (equivalently:
  $g_{orb}(\Hh)<1$). 
\end{definition}
\begin{theorem}
  \begin{enumerate}
  \item[(1)] Let $\Hh$ be domestic. Then $\Hh$ admits a tilting
    bundle, and each indecomposable vector bundle is stable and
    exceptional. The endomorphism rings of tilting bundles (sheaves)
    are just the (almost) concealed canonical algebras of
    tame-domestic type.
  \item[(2)] Let $k$ be perfect. A weighted noncommutative regular
    projective curve with centre $k$ is domestic if and only if the
    centre curve $X$ is of genus zero and the
    \emph{weight-ramification vector} of all numbers
    $p(x)e_{\tau}(x)>1$, each counted $[k(x):k]$-times, is $()$,
    $(p)$, $(p,q)$, $(2,2,n)$, $(2,3,3)$, $(2,3,4)$ or $(2,3,5)$.
  \end{enumerate}
\end{theorem}
\begin{proof}
  (1) follows from~\cite[Thm.~6.1+6.6]{lenzing:reiten:2006}. (2)
  follows easily from Theorem~\ref{thm:orbifold-euler-char}.
\end{proof}
\begin{corollary}[The real domestic zoo]
  Let $k=\RR$ be the field of real numbers. There are the following
  $38$ (families of) weighted noncommutative regular projective curves
  of positive orbifold Euler characteristic:
  \begin{itemize}
  \item Non-weighted ($\ovp=1$). With centre $\RR$: the Klein curves
    $\DD$ and $\SS^2/\pm$ (sometimes called the real projective plane),
    the Witt curves $\DD_\HH$ and $\DD_{2,2}$. With centre $\CC$: the
    Riemann sphere $\SS^2$.
  \item Weighted ($\ovp>1$). With centre $\RR$: the $27$ (families of)
    curves shown in the tables~\cite[Appendix~A]{kussin:2009}. With
    centre $\CC$: the weighted projective lines of weight types $(p)$,
    $(p,q)$, $(2,2,n)$, $(2,3,3)$, $(2,3,4)$ and $(2,3,5)$.
  \end{itemize}
  There are no parameters (\cite[Prop.~A.1.1]{kussin:2009}). The
  commutative cases are just the \emph{elliptic} and \emph{bad}
  $2$-orbifolds listed in~\cite[Thm.~13.3.6]{thurston:2002}. \qed
\end{corollary}
\subsection*{Tubular curves}
The weighted noncommutative regular projective curves of orbifold
Euler characteristic zero (also called the \emph{Calabi-Yau} case) are
the noncommutative elliptic curves (non-weighted, $\ovp=1$) and the
tubular curves ($\ovp>1$).
\begin{definition}
  We call a weighted noncommutative regular projective curve
  \emph{tubular}, if $\ovp>1$ and $\chi'_{orb}(\Hh)=0$ (equivalently:
  $g_{orb}(\Hh)=1$).
\end{definition}
\begin{theorem}
\label{thm:tubular-case}
  \begin{enumerate}
  \item[(1)] If $\Hh=\coh(\XX)$ is tubular, then $\Hh$ admits a
    tilting bundle, and each indecomposable coherent sheaf is
    semistable. The endomorphism rings of tilting sheaves are just the
    tubular algebras. Moreover, for each $\alpha\in\widehat{\QQ}$ the
    full category of semistable sheaves of slope $\alpha$ is a tubular
    family, again parametrized by a tubular curve $\XX'$, with
    $\Hh'=\coh(\XX')$ derived equivalent to $\Hh$.
  \item[(2)] Let $k$ be perfect. A weighted noncommutative regular
    projective curve with $\ovp>1$ and centre $k$ is tubular if and
    only if the centre curve $X$ is of genus zero and the
    weight-ramification vector of all numbers $p(x)e_{\tau}(x)>1$,
    each counted $[k(x):k]$-times, is $(2,3,6)$, $(2,4,4)$, $(3,3,3)$
    or $(2,2,2,2)$.
  \item[(3)] If $\Hh$ is tubular over a perfect field, then the order
    of $\tau$ in $\Aut(\Hh)$ is the maximum of the numbers
    $p(x)e_{\tau}(x)$.
  \item[(4)] Let $k$ be perfect. The weight sequence and the
    weight-ramification vector of a tubular curve, with centre $k$,
    are derived invariants. The ramification vector is not a derived
    invariant.
  \end{enumerate}
\end{theorem}
\begin{proof}
  (1) is well-known. We refer to~\cite[Ch.~8]{kussin:2009},
  also~\cite[Thm.~5.3]{lenzing:reiten:2006}. (2) follows easily from
  Theorem~\ref{thm:orbifold-euler-char}, and (3) from
  Proposition~\ref{prop:tau-multiplicity-weighted}
  and~\eqref{eq:tau-picard-general}. (4) It is well-known that the
  weight sequence is even a K-theoretic invariant (this follows
  from~\cite[Prop.~7.8]{lenzing:1996},
  also~\cite{kussin:2000b}). Clearly the fractional Calabi-Yau
  dimension by its very definition is a derived invariant, and thus so
  is the maximum of the numbers $p(x)e_{\tau}(x)$. Since for all
  possible weight-ramification vectors, $(2,3,6)$, $(2,4,4)$,
  $(3,3,3)$ or $(2,2,2,2)$, this maximum is different, it follows that
  the weight-ramification vector is uniquely determined, in its
  derived class, by the weight sequence. That the ramification vector
  is not a derived invariant follows from the real example
  in~\cite{kussin:2000}, see Example~\ref{ex:a-b-c}~(c) below.
\end{proof}
\begin{corollary}[The real tubular zoo]
  Let $k=\RR$ be the field of real numbers. There are (up to
  parameters) $39$ real weighted noncommutative regular projective
  curves of orbifold Euler characteristic zero:
  \begin{itemize}
  \item Non-weighted ($\ovp=1$). $8$ elliptic curves. With centre
    $\RR$: the Klein bottle $\KK$, the M\"obius band $\MM$ (the oval
    coloured real or quaternion), the annulus $\AA$ (there are three
    possibilities to colour the two ovals). The disc $\DD_{2,2,2,2}$
    with four segmentation points. With centre $\CC$: the torus $\TT$.
  \item Weighted ($\ovp>1$). $31$ tubular curves. Those $27$ with
    centre $\RR$ are shown in the
    tables~\cite[Appendix~A]{kussin:2009}; with centre $\CC$ there are
    the tubular weighted projective lines of the $4$ weight types
    $(2,4,4)$, $(2,3,6)$, $(3,3,3)$ and $(2,2,2,2)$.
  \end{itemize}
  $17$ of these have $s(\Hh)=1$ (these are the \emph{parabolic} (or
  \emph{flat}, that is, of curvature zero) $2$-orbifolds shown
  in~\cite[Thm.~13.3.6]{thurston:2002}, and they correspond to the
  $17$ wallpaper patterns~\cite[App.~A]{montesinos:1987}), and $22$
  have $s(\Hh)=2$.  Moreover, all these cases are fractional
  Calabi-Yau of dimension $n/n$ with $n$ the maximum of the numbers
  $p(x)e_{\tau}(x)$, and thus $n=1,\,2,\,3,\,4$ or $6$.\qed
\end{corollary}
The preceding discussion can be regarded as a classification of
noncommutative $2$-orbifolds of nonnegative Euler characteristic.
\begin{example}
\label{ex:a-b-c}
  We exhibit two tubular examples, (a) and (b) below, over the field
  $k=\RR$ of real numbers. In both cases the underlying non-weighted
  curve $\Hh_{nw}$ is the Witt curve $\DD_{2,2}$ in
  Figure~\ref{fig:genus-zero}. The weight-points $z$ are drawn in
  blue, their weights $p(z)$ are indicated besides. In both cases, one
  of the weight-points $z$ is also a segmentation point, the other
  weighted point is not.\medskip

  (a) Weight sequence $(3,3)$. (See Figure~\ref{fig:tubular}.) The
  least common multiple is $\ovp=3$. The second weight-point is
  real. (The case when the second weight-point is quaternion is
  similar.) The weight-ramification vector is $(2,3,6)$. It follows,
  that the Calabi-Yau dimension is $\frac{6}{6}$ (and not
  $\frac{\ovp}{\ovp}=\frac{3}{3}$).\medskip

  (b) Same situation, but with weight sequence $(2,4)$. (See
  Figure~\ref{fig:tubular}.) Here, $\ovp=4$. The weight-ramification
  vector is $(2,4,4)$. Therefore the Calabi-Yau dimension is
  $\frac{\ovp}{\ovp}=\frac{4}{4}$.\medskip

  (c) In a third real example we have a situation of two
  derived-equivalent tubular curves. The weighted real projective
  plane $\Hh$, given by $\SS^2/\pm$ with weight sequence $(2,2)$ (for
  certain weight points $x_1,\,x_2$), is derived-equivalent to $\Hh'$,
  given by the disc $\DD_{2,2}$ with two weights $2$, one on the real
  coloured boundary, the other on the quaternion coloured
  boundary. (See Figure~\ref{fig:tubular}.) In both cases, the weight
  sequence is $(2,2)$ and the weight-ramification vector $(2,2,2,2)$;
  in case $\Hh$ each weight appears twice, since $[k(x_i):k]=2$; in
  case $\Hh'$ the weight sequence $(2,2)$ is complemented
  ``disjointly'' by the ramification indices. The Calabi-Yau dimension
  is $\frac{2}{2}$. In $\Hh_0$ there are precisely $2$ tubes, on which
  $\tau$ has order $2$, in $\Hh'_0$ there are precisely $4$ such
  tubes.

  For further, similar examples we refer
  to~\cite[Table~A.5]{kussin:2009}.
\end{example}
\begin{figure}[h]
\begin{tikzpicture}[line width=1.1pt]
\coordinate [label=left:$\HH$] (C) at (-1.2,0);
\coordinate [label=right:$\RR$] (D) at (1.3,0);
\coordinate [label=above:$3$] (C) at (0,-1.1);
\coordinate [label=left:$3$] (D) at (1.1,0);
\draw[black] (0,0) circle (1.2);
\fill[gray,opacity=.25] (0,0) circle (1.19);
\fill[red,opacity=1.0] (0,1.2) circle (2.0pt);
\fill[red,opacity=1.0] (0,-1.2) circle (2.0pt);
\fill[blue,opacity=1.0] (1.2,0) circle (3.2pt);
\filldraw[blue,opacity=1.0,even odd rule] (0,-1.2) circle (3.2pt) circle (2.0pt);
\end{tikzpicture}\quad\quad
\begin{tikzpicture}[line width=1.1pt]
\coordinate [label=left:$\HH$] (C) at (-1.2,0);
\coordinate [label=right:$\RR$] (D) at (1.3,0);
\coordinate [label=above:$2$] (C) at (0,-1.1);
\coordinate [label=left:$4$] (D) at (1.1,0);
\draw[black] (0,0) circle (1.2);
\fill[gray,opacity=.25] (0,0) circle (1.19);
\fill[red,opacity=1.0] (0,1.2) circle (2.0pt);
\fill[red,opacity=1.0] (0,-1.2) circle (2.0pt);
\fill[blue,opacity=1.0] (1.2,0) circle (3.2pt);
\filldraw[blue,opacity=1.0,even odd rule] (0,-1.2) circle (3.2pt) circle (2.0pt);
\end{tikzpicture}\quad\quad
\begin{tikzpicture}[line width=1.1pt]
\coordinate [label=left:$\HH$] (C) at (-1.3,0);
\coordinate [label=right:$\RR$] (D) at (1.3,0);
\coordinate [label=above:$2$] (C) at (-0.9,-0.25);
\coordinate [label=left:$2$] (D) at (1.1,0);
\draw[black] (0,0) circle (1.2);
\fill[gray,opacity=.25] (0,0) circle (1.19);
\fill[red,opacity=1.0] (0,1.2) circle (2.0pt);
\fill[red,opacity=1.0] (0,-1.2) circle (2.0pt);
\fill[blue,opacity=1.0] (1.2,0) circle (3.2pt);
\fill[blue,opacity=1.0] (-1.2,0) circle (3.2pt);
\end{tikzpicture}
\caption{Some tubular cases. Left: (a), middle: (b), right: (c)}
\label{fig:tubular}
\end{figure}
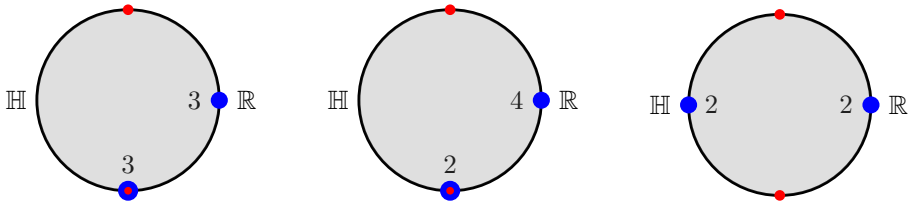
\begin{example}\label{ex:tubular-over-rationals}
  (1) In~\cite[Prop.~8.3.1]{kussin:2009} we discussed an example of a
  triple of tubular curves over the field $k=\QQ$ of rational numbers,
  each with weight sequence $(2)$, which are Fourier-Mukai
  partners. One derives from Theorem~\ref{thm:tubular-case}, or
  computes directly using~\eqref{eq:e-f-formula}, that in each case
  the weight-ramification vector is given by $(2,2,2,2)$. Since one of
  these three tubular curves arises by insertion of weights from the
  curve in Example~\ref{ex:K-H-bimodule}, that curve has three
  ramification points.\medskip

  (2) In~\cite[Prop.~8.4.1]{kussin:2009} we discussed an example of a
  tubular curve over the field $\QQ(\gge{i})$, where
  $\gge{i}=\sqrt{-1}$, also with weight sequence $(2)$. Its
  Grothendieck group is isometric-isomorphic to the Grothendieck group
  of the curves from part~(1). By invoking
  Example~\ref{ex:non-simple-bimodule} we see that here the weight-ramification
  vector is given by $(2,4,4)$, in contrast to part~(1).
\end{example}
\subsection*{Acknowledgements}
I thank Helmut Lenzing for drawing my attention to the
paper~\cite{witt:1934} of E.~Witt, and also for giving me access to
his slides~\cite{lenzing:2001}; there were also fruitful discussions
about the ``correct'' Euler characteristic of noncommutative real
curves. I also thank David Ploog and Dieter Vossieck for various
useful critical comments, and an anonymous referee for her/his helpful
proposals which lead to an improvement of the presentation. I am
grateful to Osamu Iyama for his hospitality I could enjoy at the
Graduate School of Mathematics of Nagoya University where the final
version of the paper was produced.

\bibliographystyle{elsarticle-harv}

\begin{thebibliography}{88}
\expandafter\ifx\csname natexlab\endcsname\relax\def\natexlab#1{#1}\fi
\expandafter\ifx\csname url\endcsname\relax
  \def\url#1{\texttt{#1}}\fi
\expandafter\ifx\csname urlprefix\endcsname\relax\def\urlprefix{URL }\fi

\bibitem[{Alling(1974)}]{alling:1974}
Alling, N.~L., 1974. Analytic geometry on real algebraic curves. Math. Ann.
  207, 23--46.

\bibitem[{Alling(1981)}]{alling:1981}
Alling, N.~L., 1981. Real elliptic curves. Vol.~54 of North-Holland Mathematics
  Studies. North-Holland Publishing Co., Amsterdam-New York, notas de
  Matem{\'a}tica [Mathematical Notes], 81.

\bibitem[{Alling and Greenleaf(1971)}]{alling:greenleaf:1971}
Alling, N.~L., Greenleaf, N., 1971. Foundations of the theory of {K}lein
  surfaces. Lecture Notes in Mathematics, Vol. 219. Springer-Verlag, Berlin-New
  York.

\bibitem[{Altman and Kleiman(1970)}]{altman:kleiman:1970}
Altman, A., Kleiman, S., 1970. Introduction to {G}rothendieck duality theory.
  Lecture Notes in Mathematics, Vol. 146. Springer-Verlag, Berlin-New York.

\bibitem[{Amdal and Ringdal(1968)}]{amdal:ringdal:1968ab}
Amdal, I.~K., Ringdal, F., 1968. Cat\'egories unis\'erielles. C. R. Acad. Sci.
  Paris S\'er. A-B 267, A85--A87, A247--A249.

\bibitem[{Amitsur(1967)}]{amitsur:1967}
Amitsur, S.~A., 1967. Prime rings having polynomial identities with arbitrary
  coefficients. Proc. London Math. Soc. (3) 17, 470--486.

\bibitem[{Artin and de~Jong(2004)}]{artin:dejong:2004}
Artin, M., de~Jong, A.~J., 2004. Stable orders over surfaces, unpublished
  manuscript.
\newline\urlprefix\url{http://www.math.lsa.umich.edu/courses/711/ordersms-num.pdf}

\bibitem[{Artin and Stafford(1995)}]{artin:stafford:1995}
Artin, M., Stafford, J.~T., 1995. Noncommutative graded domains with quadratic
  growth. Invent. Math. 122~(2), 231--276.
\newline\urlprefix\url{http://dx.doi.org/10.1007/BF01231444}

\bibitem[{Artin and Zhang(1994)}]{artin:zhang:1994}
Artin, M., Zhang, J.~J., 1994. Noncommutative projective schemes. Adv. Math.
  109~(2), 228--287.
\newline\urlprefix\url{http://dx.doi.org/10.1006/aima.1994.1087}

\bibitem[{Atiyah(1957)}]{atiyah:1957}
Atiyah, M.~F., 1957. Vector bundles over an elliptic curve. Proc. London Math.
  Soc. (3) 7, 414--452.

\bibitem[{Auslander and Goldman(1960)}]{auslander:goldman:1960}
Auslander, M., Goldman, O., 1960. Maximal orders. Trans. Amer. Math. Soc. 97,
  1--24.
\newline\urlprefix\url{http://dx.doi.org/10.1090/S0002-9947-1960-0117252-7}

\bibitem[{Auslander and Reiten(1975)}]{auslander:reiten:1975}
Auslander, M., Reiten, I., 1975. Representation theory of {A}rtin algebras.
  {III}. {A}lmost split sequences. Comm. Algebra 3, 239--294.

\bibitem[{Baer et~al.(1987)Baer, Geigle, and
  Lenzing}]{baer:geigle:lenzing:1987}
Baer, D., Geigle, W., Lenzing, H., 1987. The preprojective algebra of a tame
  hereditary {A}rtin algebra. Comm. Algebra 15~(1-2), 425--457.
\newline\urlprefix\url{http://dx.doi.org/10.1080/00927878708823425}

\bibitem[{Bass(1968)}]{bass:1968}
Bass, H., 1968. Algebraic {$K$}-theory. W. A. Benjamin, Inc., New
  York-Amsterdam.

\bibitem[{Bondal and Orlov(2001)}]{bondal:orlov:2001}
Bondal, A., Orlov, D., 2001. Reconstruction of a variety from the derived
  category and groups of autoequivalences. Compositio Math. 125~(3), 327--344.
\newline\urlprefix\url{http://dx.doi.org/10.1023/A:1002470302976}

\bibitem[{Bourbaki(1989)}]{bourbaki:1989}
Bourbaki, N., 1989. Commutative algebra. {C}hapters 1--7. Elements of
  Mathematics (Berlin). Springer-Verlag, Berlin, translated from the French,
  Reprint of the 1972 edition.

\bibitem[{Braun(1990)}]{braun:1990}
Braun, A., 1990. Completions of {N}oetherian {PI} rings. J. Algebra 133~(2),
  340--350.
\newline\urlprefix\url{http://dx.doi.org/10.1016/0021-8693(90)90273-Q}

\bibitem[{Brumer(1963)}]{brumer:1963}
Brumer, A., 1963. Structure of hereditary orders. Bull. Amer. Math. Soc. 69,
  721--724.
\newline\urlprefix\url{http://dx.doi.org/10.1090/S0002-9904-1963-11002-2}

\bibitem[{Brumer(1964)}]{brumer:1964}
Brumer, A., 1964. Addendum to ``{S}tructure of hereditary orders``. Bull. Amer.
  Math. Soc. 70, 185.
\newline\urlprefix\url{http://dx.doi.org/10.1090/S0002-9904-1964-11070-3}

\bibitem[{Burban et~al.(2015)Burban, Drozd, and
  Gavran}]{burban:drozd:gavran:2015}
Burban, I., Drozd, Y., Gavran, V., 2015. Minors of non-commutative schemes,
  preprint.
\newline\urlprefix\url{http://arxiv.org/abs/1501.06023}

\bibitem[{C{\u{a}}ld{\u{a}}raru(2000)}]{caldararu:2000}
C{\u{a}}ld{\u{a}}raru, A., 2000. Derived categories of twisted sheaves on
  {C}alabi-{Y}au manifolds. ProQuest LLC, Ann Arbor, MI, thesis
  (Ph.D.)--Cornell University.
\newline\urlprefix\url{http://search.proquest.com/docview/304592237}

\bibitem[{C{\u{a}}ld{\u{a}}raru(2002)}]{caldararu:2002}
C{\u{a}}ld{\u{a}}raru, A., 2002. Derived categories of twisted sheaves on
  elliptic threefolds. J. Reine Angew. Math. 544, 161--179.
\newline\urlprefix\url{http://dx.doi.org/10.1515/crll.2002.022}

\bibitem[{Chan and Ingalls(2004)}]{chan:ingalls:2004}
Chan, D., Ingalls, C., 2004. Non-commutative coordinate rings and stacks. Proc.
  London Math. Soc. (3) 88~(1), 63--88.
\newline\urlprefix\url{http://dx.doi.org/10.1112/S0024611503014278}

\bibitem[{Chatters and Jordan(1986)}]{chatters:jordan:1986}
Chatters, A.~W., Jordan, D.~A., 1986. Noncommutative unique factorisation
  rings. J. London Math. Soc. (2) 33~(1), 22--32.
\newline\urlprefix\url{http://dx.doi.org/10.1112/jlms/s2-33.1.22}

\bibitem[{Chen and Krause(2009)}]{chen:krause:2009}
Chen, X.-W., Krause, H., 2009. Introduction to coherent sheaves on weighted
  projective lines, preprint.
\newline\urlprefix\url{http://arxiv.org/abs/0911.4473}

\bibitem[{Crawley-Boevey(1991)}]{crawley-boevey:1991}
Crawley-Boevey, W.~W., 1991. Regular modules for tame hereditary algebras.
  Proc. London Math. Soc. (3) 62~(3), 490--508.
\newline\urlprefix\url{http://dx.doi.org/10.1112/plms/s3-62.3.490}

\bibitem[{Demazure and Gabriel(1980)}]{demazure:gabriel:1980}
Demazure, M., Gabriel, P., 1980. Introduction to algebraic geometry and
  algebraic groups. Vol.~39 of North-Holland Mathematics Studies. North-Holland
  Publishing Co., Amsterdam-New York, translated from the French by J. Bell.

\bibitem[{Demeyer and Knus(1976)}]{demeyer:knus:1976}
Demeyer, F.~R., Knus, M.~A., 1976. The {B}rauer group of a real curve. Proc.
  Amer. Math. Soc. 57~(2), 227--232.

\bibitem[{Dlab and Ringel(1976)}]{dlab:ringel:1976}
Dlab, V., Ringel, C.~M., 1976. Indecomposable representations of graphs and
  algebras. Mem. Amer. Math. Soc. 6~(173), v+57.
\newline\urlprefix\url{http://dx.doi.org/10.1090/memo/0173}

\bibitem[{Gabriel(1962)}]{gabriel:1962}
Gabriel, P., 1962. Des cat\'egories ab\'eliennes. Bull. Soc. Math. France 90,
  323--448.

\bibitem[{Gabriel(1973)}]{gabriel:1973}
Gabriel, P., 1973. Indecomposable representations. {II}. In: Symposia
  {M}athematica, {V}ol. {XI} ({C}onvegno di {A}lgebra {C}ommutativa, {INDAM},
  {R}ome, 1971). Academic Press, London, pp. 81--104.

\bibitem[{Geigle and Lenzing(1987)}]{geigle:lenzing:1987}
Geigle, W., Lenzing, H., 1987. A class of weighted projective curves arising in
  representation theory of finite-dimensional algebras. In: Singularities,
  representation of algebras, and vector bundles ({L}ambrecht, 1985). Vol. 1273
  of Lecture Notes in Math. Springer, Berlin, pp. 265--297.
\newline\urlprefix\url{http://dx.doi.org/10.1007/BFb0078849}

\bibitem[{Geigle and Lenzing(1991)}]{geigle:lenzing:1991}
Geigle, W., Lenzing, H., 1991. Perpendicular categories with applications to
  representations and sheaves. J. Algebra 144~(2), 273--343.
\newline\urlprefix\url{http://dx.doi.org/10.1016/0021-8693(91)90107-J}

\bibitem[{Goodearl and Warfield(2004)}]{goodearl:warfield:2004}
Goodearl, K.~R., Warfield, Jr., R.~B., 2004. An introduction to noncommutative
  {N}oetherian rings, 2nd Edition. Vol.~61 of London Mathematical Society
  Student Texts. Cambridge University Press, Cambridge.

\bibitem[{Grothendieck(1961{\natexlab{a}})}]{grothendieck:1961b}
Grothendieck, A., 1961{\natexlab{a}}. \'{E}l\'ements de g\'eom\'etrie
  alg\'ebrique. {II}. \'{E}tude globale \'el\'ementaire de quelques classes de
  morphismes. Inst. Hautes \'Etudes Sci. Publ. Math.~(8), 222.

\bibitem[{Grothendieck(1961{\natexlab{b}})}]{grothendieck:1961ca}
Grothendieck, A., 1961{\natexlab{b}}. \'{E}l\'ements de g\'eom\'etrie
  alg\'ebrique. {III}. \'{E}tude cohomologique des faisceaux coh\'erents. {I}.
  Inst. Hautes \'Etudes Sci. Publ. Math.~(11), 167.

\bibitem[{Harada(1963)}]{harada:1963}
Harada, M., 1963. Hereditary orders. Trans. Amer. Math. Soc. 107, 273--290.
\newline\urlprefix\url{http://dx.doi.org/10.1090/S0002-9947-1963-0151489-9}

\bibitem[{Hartshorne(1977)}]{hartshorne:1977}
Hartshorne, R., 1977. Algebraic geometry. Springer-Verlag, New York-Heidelberg,
  graduate Texts in Mathematics, No. 52.

\bibitem[{Jacobson(1996)}]{jacobson:1996}
Jacobson, N., 1996. Finite-dimensional division algebras over fields.
  Springer-Verlag, Berlin.
\newline\urlprefix\url{http://dx.doi.org/10.1007/978-3-642-02429-0}

\bibitem[{Kashiwara and Schapira(2006)}]{kashiwara:schapira:2006}
Kashiwara, M., Schapira, P., 2006. Categories and sheaves. Vol. 332 of
  Grundlehren der Mathematischen Wissenschaften [Fundamental Principles of
  Mathematical Sciences]. Springer-Verlag, Berlin.
\newline\urlprefix\url{http://dx.doi.org/10.1007/3-540-27950-4}

\bibitem[{Keller(2005)}]{keller:2005}
Keller, B., 2005. On triangulated orbit categories. Doc. Math. 10, 551--581.

\bibitem[{Kupisch(1975)}]{kupisch:1975}
Kupisch, H., 1975. Einreihige {A}lgebren \"uber einem perfekten {K}\"orper. J.
  Algebra 33, 68--74.
\newline\urlprefix\url{http://dx.doi.org/10.1016/0021-8693(75)90132-5}

\bibitem[{{Kussin}(1997)}]{kussin:1997}
{Kussin}, D., 1997. {Graduierte Faktorialit\"at und die Parameterkurven
  tubularer Familien.} Dissertation, Univ. Paderborn.

\bibitem[{Kussin(2000{\natexlab{a}})}]{kussin:2000}
Kussin, D., 2000{\natexlab{a}}. Non-isomorphic derived-equivalent tubular
  curves and their associated tubular algebras. J. Algebra 226~(1), 436--450.
\newline\urlprefix\url{http://dx.doi.org/10.1006/jabr.1999.8200}

\bibitem[{Kussin(2000{\natexlab{b}})}]{kussin:2000b}
Kussin, D., 2000{\natexlab{b}}. On the {$K$}-theory of tubular algebras.
  Colloq. Math. 86~(1), 137--152.

\bibitem[{Kussin(2008)}]{kussin:2008}
Kussin, D., 2008. Parameter curves for the regular representations of tame
  bimodules. J. Algebra 320~(6), 2567--2582.
\newline\urlprefix\url{http://dx.doi.org/10.1016/j.jalgebra.2008.05.022}

\bibitem[{Kussin(2009)}]{kussin:2009}
Kussin, D., 2009. Noncommutative curves of genus zero: related to finite
  dimensional algebras. Mem. Amer. Math. Soc. 201~(942), x+128.
\newline\urlprefix\url{http://dx.doi.org/10.1090/memo/0942}

\bibitem[{Kussin et~al.(2013)Kussin, Lenzing, and
  Meltzer}]{kussin:lenzing:meltzer:2013}
Kussin, D., Lenzing, H., Meltzer, H., 2013. Triangle singularities,
  {ADE}-chains, and weighted projective lines. Adv. Math. 237, 194--251.
\newline\urlprefix\url{http://dx.doi.org/10.1016/j.aim.2013.01.006}

\bibitem[{Lam(1991)}]{lam:1991}
Lam, T.~Y., 1991. A first course in noncommutative rings. Vol. 131 of Graduate
  Texts in Mathematics. Springer-Verlag, New York.
\newline\urlprefix\url{http://dx.doi.org/10.1007/978-1-4684-0406-7}

\bibitem[{Lenagan(1994)}]{lenagan:1994}
Lenagan, T.~H., 1994. Domains with linear growth. Bull. Belg. Math. Soc. Simon
  Stevin 1~(1), 107--109.
\newline\urlprefix\url{http://projecteuclid.org/euclid.bbms/1103408458}

\bibitem[{Lenzing(1996)}]{lenzing:1996}
Lenzing, H., 1996. A {K}-theoretic study of canonical algebras. In: Bautista,
  R., Mart{\'{\i}}nez-Villa, R., de~la Pe{\~{n}}a, J.~A. (Eds.), Representation
  Theory of Algebras (Cocoyoc, 1994). Vol.~18 of CMS Conf. Proc. Amer. Math.
  Soc., Providence, RI, pp. 433--473.

\bibitem[{Lenzing(1997)}]{lenzing:1997b}
Lenzing, H., 1997. Hereditary {N}oetherian categories with a tilting complex.
  Proc. Amer. Math. Soc. 125~(7), 1893--1901.
\newline\urlprefix\url{http://dx.doi.org/10.1090/S0002-9939-97-04122-1}

\bibitem[{Lenzing(1998)}]{lenzing:1998}
Lenzing, H., 1998. Representations of finite-dimensional algebras and
  singularity theory. In: Trends in ring theory ({M}iskolc, 1996). Vol.~22 of
  CMS Conf. Proc. Amer. Math. Soc., Providence, RI, pp. 71--97.

\bibitem[{Lenzing(2001)}]{lenzing:2001}
Lenzing, H., 2001. Kleinsche {F}lasche, 2-{T}orus, {M}\"obiusband und ihre
  nichtkommutativen {V}erwandten, slides of a talk given in Erlangen.

\bibitem[{Lenzing and de~la Pe{\~n}a(1999)}]{lenzing:delapena:1999}
Lenzing, H., de~la Pe{\~n}a, J.~A., 1999. Concealed-canonical algebras and
  separating tubular families. Proc. London Math. Soc. (3) 78~(3), 513--540.
\newline\urlprefix\url{http://dx.doi.org/10.1112/S0024611599001872}

\bibitem[{Lenzing and Reiten(2006)}]{lenzing:reiten:2006}
Lenzing, H., Reiten, I., 2006. Hereditary {N}oetherian categories of positive
  {E}uler characteristic. Math. Z. 254~(1), 133--171.
\newline\urlprefix\url{http://dx.doi.org/10.1007/s00209-006-0938-6}

\bibitem[{Lenzing and Zuazua(2004)}]{lenzing:zuazua:2004}
Lenzing, H., Zuazua, R., 2004. Auslander-{R}eiten duality for abelian
  categories. Bol. Soc. Mat. Mexicana (3) 10~(2), 169--177 (2005).

\bibitem[{Lerner and Oppermann(2015)}]{lerner:oppermann:2015}
Lerner, B., Oppermann, S., 2015. A recollement approach to {G}eigle-{L}enzing
  weighted projective varieties, preprint.
\newline\urlprefix\url{http://arxiv.org/abs/1505.01931}

\bibitem[{L{\'o}pez~Mart{\'{\i}}n(2014)}]{lopez_martin:2014}
L{\'o}pez~Mart{\'{\i}}n, A.~C., 2014. Fourier-{M}ukai partners of singular
  genus one curves. J. Geom. Phys. 83, 36--42.
\newline\urlprefix\url{http://dx.doi.org/10.1016/j.geomphys.2014.05.011}

\bibitem[{Marubayashi and Van~Oystaeyen(2012)}]{marubayashi:vanoystaeyen:2012}
Marubayashi, H., Van~Oystaeyen, F., 2012. Prime divisors and noncommutative
  valuation theory. Vol. 2059 of Lecture Notes in Mathematics. Springer,
  Heidelberg.
\newline\urlprefix\url{http://dx.doi.org/10.1007/978-3-642-31152-9}

\bibitem[{May(1975)}]{may:1975}
May, C.~L., 1975. Automorphisms of compact {K}lein surfaces with boundary.
  Pacific J. Math. 59~(1), 199--210.

\bibitem[{McConnell and Robson(2001)}]{mcconnell:robson:2001}
McConnell, J.~C., Robson, J.~C., 2001. Noncommutative {N}oetherian rings,
  revised Edition. Vol.~30 of Graduate Studies in Mathematics. American
  Mathematical Society, Providence, RI, with the cooperation of L. W. Small.

\bibitem[{Meltzer(1997)}]{meltzer:1997}
Meltzer, H., 1997. Tubular mutations. Colloq. Math. 74~(2), 267--274.

\bibitem[{Montesinos(1987)}]{montesinos:1987}
Montesinos, J.~M., 1987. Classical tessellations and three-manifolds.
  Universitext. Springer-Verlag, Berlin.
\newline\urlprefix\url{http://dx.doi.org/10.1007/978-3-642-61572-6}

\bibitem[{Natanzon(1990)}]{natanzon:1990}
Natanzon, S.~M., 1990. Klein surfaces. Uspekhi Mat. Nauk 45~(6(276)), 47--90,
  189.
\newline\urlprefix\url{http://dx.doi.org/10.1070/RM1990v045n06ABEH002713}

\bibitem[{Pierce(1982)}]{pierce:1982}
Pierce, R.~S., 1982. Associative algebras. Vol.~88 of Graduate Texts in
  Mathematics. Springer-Verlag, New York, studies in the History of Modern
  Science, 9.

\bibitem[{Reiner(2003)}]{reiner:2003}
Reiner, I., 2003. Maximal orders. Vol.~28 of London Mathematical Society
  Monographs. New Series. The Clarendon Press, Oxford University Press, Oxford,
  corrected reprint of the 1975 original, With a foreword by M. J. Taylor.

\bibitem[{Reiten and Van~den Bergh(2001)}]{reiten:vandenbergh:2001}
Reiten, I., Van~den Bergh, M., 2001. Grothendieck groups and tilting objects.
  Algebr. Represent. Theory 4~(1), 1--23, special issue dedicated to Klaus
  Roggenkamp on the occasion of his 60th birthday.
\newline\urlprefix\url{http://dx.doi.org/10.1023/A:1009902810813}

\bibitem[{Reiten and Van~den Bergh(2002)}]{reiten:vandenbergh:2002}
Reiten, I., Van~den Bergh, M., 2002. Noetherian hereditary abelian categories
  satisfying {S}erre duality. J. Amer. Math. Soc. 15~(2), 295--366.
\newline\urlprefix\url{http://dx.doi.org/10.1090/S0894-0347-02-00387-9}

\bibitem[{Ringel(1975)}]{ringel:1975}
Ringel, C.~M., 1975. Unions of chains of indecomposable modules. Comm. Algebra
  3~(12), 1121--1144.

\bibitem[{Ringel(1976)}]{ringel:1976}
Ringel, C.~M., 1976. Representations of {$K$}-species and bimodules. J. Algebra
  41~(2), 269--302.
\newline\urlprefix\url{http://dx.doi.org/10.1016/0021-8693(76)90184-8}

\bibitem[{Ringel(1979)}]{ringel:1979}
Ringel, C.~M., 1979. Infinite-dimensional representations of finite-dimensional
  hereditary algebras. In: Symposia {M}athematica, {V}ol. {XXIII} ({C}onf.
  {A}belian {G}roups and their {R}elationship to the {T}heory of {M}odules,
  {INDAM}, {R}ome, 1977). Academic Press, London, pp. 321--412.

\bibitem[{Ringel(1990)}]{ringel:crawley-boevey:1990}
Ringel, C.~M., 1990. The canonical algebras. In: Topics in Algebra, Part 1
  (Warsaw, 1988). No.~26 in Banach Center Publ. PWN, Warsaw, pp. 407--432, with
  an appendix by William Crawley-Boevey.

\bibitem[{Rowen(1980)}]{rowen:1980}
Rowen, L.~H., 1980. Polynomial identities in ring theory. Vol.~84 of Pure and
  Applied Mathematics. Academic Press, Inc. [Harcourt Brace Jovanovich,
  Publishers], New York-London.

\bibitem[{Schilling(1950)}]{schilling:1950}
Schilling, O. F.~G., 1950. The {T}heory of {V}aluations. Mathematical Surveys,
  No. 4. American Mathematical Society, New York, N. Y.

\bibitem[{Scott(1983)}]{scott:1983}
Scott, P., 1983. The geometries of {$3$}-manifolds. Bull. London Math. Soc.
  15~(5), 401--487.
\newline\urlprefix\url{http://dx.doi.org/10.1112/blms/15.5.401}

\bibitem[{Seidel and Thomas(2001)}]{seidel:thomas:2001}
Seidel, P., Thomas, R., 2001. Braid group actions on derived categories of
  coherent sheaves. Duke Math. J. 108~(1), 37--108.
\newline\urlprefix\url{http://dx.doi.org/10.1215/S0012-7094-01-10812-0}

\bibitem[{Small(1967)}]{small:1967}
Small, L.~W., 1967. Semihereditary rings. Bull. Amer. Math. Soc. 73, 656--658.

\bibitem[{Smith(1982)}]{smith_pf:1982}
Smith, P.~F., 1982. The {A}rtin-{R}ees property. In: Paul {D}ubreil and
  {M}arie-{P}aule {M}alliavin {A}lgebra {S}eminar, 34th {Y}ear ({P}aris, 1981).
  Vol. 924 of Lecture Notes in Math. Springer, Berlin-New York, pp. 197--240.

\bibitem[{Stafford and van~den Bergh(2001)}]{stafford:vandenbergh:2001}
Stafford, J.~T., van~den Bergh, M., 2001. Noncommutative curves and
  noncommutative surfaces. Bull. Amer. Math. Soc. (N.S.) 38~(2), 171--216.
\newline\urlprefix\url{http://dx.doi.org/10.1090/S0273-0979-01-00894-1}

\bibitem[{Thurston(2002)}]{thurston:2002}
Thurston, W.~P., 2002. The geometry and topology of three-manifolds, electronic
  version 1.1. Extended version of the book published by Princeton University
  Press.
\newline\urlprefix\url{http://library.msri.org/books/gt3m}

\bibitem[{{Tsen}(1933)}]{tsen:1933}
{Tsen}, C.~C., 1933. {Divisionsalgebren \"uber Funktionenk\"orpern.} {Nachr.
  Ges. Wiss. G\"ottingen, Math.-Phys. Kl.} 1933, 335--339.

\bibitem[{Van~den Bergh(2001)}]{vandenbergh:2001}
Van~den Bergh, M., 2001. Blowing up of non-commutative smooth surfaces. Mem.
  Amer. Math. Soc. 154~(734), x+140.
\newline\urlprefix\url{http://dx.doi.org/10.1090/memo/0734}

\bibitem[{Van~den Bergh and Van~Geel(1984)}]{vandenbergh:vangeel:1984}
Van~den Bergh, M., Van~Geel, J., 1984. A duality theorem for orders in central
  simple algebras over function fields. J. Pure Appl. Algebra 31~(1-3),
  227--239.
\newline\urlprefix\url{http://dx.doi.org/10.1016/0022-4049(84)90088-4}

\bibitem[{Van~den Bergh and Van~Geel(1985)}]{vandenbergh:vangeel:1985}
Van~den Bergh, M., Van~Geel, J., 1985. Algebraic elements in division algebras
  over function fields of curves. Israel J. Math. 52~(1-2), 33--45.
\newline\urlprefix\url{http://dx.doi.org/10.1007/BF02776077}

\bibitem[{{Witt}(1934{\natexlab{a}})}]{witt:1934b}
{Witt}, E., 1934{\natexlab{a}}. {Riemann-Rochscher Satz und $Z$-Funktion im
  Hyperkomplexen.} {Math. Ann.} 110, 12--28.

\bibitem[{{Witt}(1934{\natexlab{b}})}]{witt:1934}
{Witt}, E., 1934{\natexlab{b}}. {Zerlegung reeller algebraischer Funktionen in
  Quadrate. Schiefk\"orper \"uber reellem Funktionenk\"orper.} {J. Reine Angew.
  Math.} 171, 4--11.

\bibitem[{{Witt}(1936)}]{witt:1936}
{Witt}, E., 1936. {Schiefk\"orper \"uber diskret bewerteten K\"orpern.} {J.
  Reine Angew. Math.} 176, 153--156.

\end{thebibliography}

\end{document}